\documentclass[letterpaper,reqno]{amsart}
\usepackage{graphicx}

\pdfoutput=1

\usepackage{amssymb,mathrsfs, amsmath,amsfonts,verbatim, amsthm}
\usepackage{bm}
\usepackage[usenames]{color}

\usepackage{enumerate}
\usepackage{enumitem}

\usepackage{tikz}

\usepackage{microtype}
\usepackage{stackengine}
\usepackage{cmap}
\usepackage{bbm}
\usepackage{thmtools}
\usepackage{thm-restate}

\usepackage[colorlinks=true,allcolors=blue]{hyperref}

\newtheorem{theorem}{Theorem}[section]
\newtheorem{corollary}[theorem]{Corollary}
\newtheorem{lemma}[theorem]{Lemma}
\newtheorem{proposition}[theorem]{Proposition}

\theoremstyle{definition}
\newtheorem{definition}[theorem]{Definition}
\theoremstyle{remark}

\theoremstyle{remark}
\newtheorem{remark}[theorem]{Remark}
\theoremstyle{definition}
\newtheorem{example}[theorem]{Example}
\theoremstyle{definition}
\newtheorem{fact*}{Fact}

\newcounter{theoremx}

\theoremstyle{plain}
\newtheorem{thmx}{Theorem}[theoremx]
\DeclareMathOperator{\dom}{dom}

\DeclareMathOperator\vecc{{\bf vec}}

\newcommand{\B}{\mathcal{B}}

\newcommand{\Z}{\mathbb{Z}}
\newcommand{\N}{\mathbb{N}}

\newcommand{\cA}{\mathcal{A}}

\newcommand{\C}{\mathbb{C}}

\newcommand{\R}{\mathbb{R}}

\newcommand{\abs}[1]{{\left\vert#1\right\vert}}
\newcommand{\set}[1]{{\left\{#1\right\}}}

\newcommand{\til}{\raise.17ex\hbox{$\scriptstyle\mathtt{\sim}$}}

\newcommand{\bbm}{\left[ \begin{smallmatrix}}
\newcommand{\ebm}{\end{smallmatrix} \right]}

\newcommand{\bpm}{\begin{pmatrix}}
\newcommand{\epm}{\end{pmatrix}}
\newcommand{\bal}{\begin{align*}}
\newcommand{\eal}{\end{align*}}
\numberwithin{equation}{section}

\newlength{\Mheight}
\newlength{\cwidth}

\newcommand{\dfn}[1]{{\bf #1}\index{#1}}

\newcommand\psub{\mathrel{%
  \ooalign{\raise0.2ex\hbox{$\subset$}\cr\hidewidth\raise-0.8ex\hbox{\box{0.9}{$\sim$}}\hidewidth\cr}}}
\newcommand\piso{\mathrel{%
  \ooalign{\raise0.2ex\hbox{$=$}\cr\hidewidth\raise-0.5ex\hbox{\scalebox{0.9}{$\sim$}}\hidewidth\cr}}}

\newcommand\mfq{\mathfrak{q}}

\newcommand\mfa{\mathfrak{a}}
\newcommand{\bbp}{\mathbbm{p}}
\newcommand{\bbd}{\mathbbm{d}}
\newcommand{\bbq}{\reflectbox{$\mathbbm{p}$}}
\newcommand{\bbx}{\mathbbm{x}}
\newcommand{\bby}{\mathbbm{y}}
\newcommand{\bbz}{\mathbbm{z}}

\newcommand{\bbw}{\mathbbm{w}}

\newcommand{\bbG}{\mathbbm{G}}
\newcommand{\bba}{\mathbbm{a}}

\newcommand*{\defeq}{\mathrel{\vcenter{\baselineskip0.6ex \lineskiplimit0pt
                     \hbox{\small.}\hbox{\small.}}}%
                     =}

\makeatletter
\def\moverlay{\mathpalette\mov@rlay}
\def\mov@rlay#1#2{\leavevmode\vtop{
                \baselineskip\z@skip \lineskiplimit-\maxdimen
                \ialign{\hfil$#1##$\hfil\cr#2\crcr}}}
\makeatother

\makeatletter
\def\moverlay{\mathpalette\mov@rlay}
\def\mov@rlay#1#2{\leavevmode\vtop{
                \baselineskip\z@skip \lineskiplimit-\maxdimen
                \ialign{\hfil$#1##$\hfil\cr#2\crcr}}}
\makeatother

\makeatletter
\DeclareFontFamily{OMX}{MnSymbolE}{}
\DeclareSymbolFont{MnLargeSymbols}{OMX}{MnSymbolE}{m}{n}
\SetSymbolFont{MnLargeSymbols}{bold}{OMX}{MnSymbolE}{b}{n}
\DeclareFontShape{OMX}{MnSymbolE}{m}{n}{
    <-6>  MnSymbolE5
   <6-7>  MnSymbolE6
   <7-8>  MnSymbolE7
   <8-9>  MnSymbolE8
   <9-10> MnSymbolE9
  <10-12> MnSymbolE10
  <12->   MnSymbolE12
}{}
\DeclareFontShape{OMX}{MnSymbolE}{b}{n}{
    <-6>  MnSymbolE-Bold5
   <6-7>  MnSymbolE-Bold6
   <7-8>  MnSymbolE-Bold7
   <8-9>  MnSymbolE-Bold8
   <9-10> MnSymbolE-Bold9
  <10-12> MnSymbolE-Bold10
  <12->   MnSymbolE-Bold12
}{}

\let\llangle\@undefined
\let\rrangle\@undefined
\DeclareMathDelimiter{\llangle}{\mathopen}%
                     {MnLargeSymbols}{'164}{MnLargeSymbols}{'164}
\DeclareMathDelimiter{\rrangle}{\mathclose}%
                     {MnLargeSymbols}{'171}{MnLargeSymbols}{'171}
\makeatother

\newcommand*{\mpsymbol}{
        \!\protect\raisebox{0.534ex}{\protect\scalebox{.28}{
                        \begin{tikzpicture}[line width=0.6ex]
                        \draw[arrows={<->}]
                        (0,0) -- (5ex,0);
                        \end{tikzpicture}
                }}\!
        }

\newcommand{\lra}{\mpsymbol}
\newcommand{\fps}[1]{\llangle #1 \rrangle}
\newcommand{\fpsx}{\fps{\bbx}}
\newcommand{\fralg}[1]{\langle #1 \rangle}
\newcommand{\fax}{\fralg{\bbx}}

\newcommand{\plangle}{\moverlay{(\cr<}}
\newcommand{\prangle}{\moverlay{)\cr>}}

\newcommand{\skf}[1]{\plangle #1 \prangle}

\newcommand{\hypj}[1]{J^{\textup{hyp}}_{#1}}

\newcommand{\ew}{\mathbbm{1}}
\newcommand{\x}{\bm{x}}
\newcommand{\y}{\bm{y}}
\newcommand{\z}{\bm{z}}
\newcommand{\h}{\bm{h}}
\newcommand{\w}{\bm{w}}
\renewcommand{\k}{\bm{k}}
\renewcommand{\u}{\bm{u}}

\newcommand{\GM}{\textup{GM}}
\newcommand{\UD}{\textup{UD}}

\makeatletter
\def\BState{\State\hskip-\ALG@thistlm}
\makeatother

\showboxdepth=\maxdimen
\showboxbreadth=\maxdimen

\makeindex

\allowdisplaybreaks

\title[The free Grothendieck theorem]%
{The free Grothendieck theorem}
\author{Meric L. Augat}

\begin{document}

\begin{abstract}
   The main result of this article establishes the free analog
 of Grothendieck's Theorem on bijective polynomial
  mappings of $\mathbb C^g$. Namely, we show if $p$ is a  polynomial
   mapping in $g$ freely noncommuting variables 
  sending $g$-tuples of matrices (of the same size) to
 $g$-tuples of matrices (of the same size) that is injective,
 then it has a free polynomial inverse.
	
	Other results include an algorithm that tests if a free polynomial mapping 
$p$ has a polynomial inverse
 (equivalently is injective;  equivalently is bijective).
   Further, a  class of free algebraic functions, called hyporational, lying strictly between the
 free rational functions and the free algebraic functions are 
		identified. They  play a significant 
 role in the proof of the main result.
\end{abstract}

\maketitle
\section{Introduction}
\label{sec:Introduction}
A remarkable pair of theorems of Grothendieck \cite{Gro66,Gro67} say if $p:\C^g\to \C^g$ is an injective polynomial, 
	then $p$ is bijective and its inverse is a polynomial. 
Later degree bounds on the inverse were discovered.
These results are of course intimately connected with the fascinating Jacobian conjecture (see for instance \cite{BCW82,vdE00}) 
	and the question of tame versus wild automorphisms of the polynomial ring (see for instance \cite{SU03,SU04,UY04,U06}).

In this article we prove the free Grothendieck theorem.
Our approach involves careful analysis of the noncommutative Jacobian matrix as found in \cite{Reu92}, the theories of free rational series 
	\cite{BerReu11}  and their realizations \cite{KlVo17,Vol15}, formal power series in noncommuting variables \cite{Sta11},
	free analysis \cite{HM04,HKM11}, proper algebraic systems \cite{SalSoi78,PetSal09}, 
	free derivatives\cite{Pas14} and skew fields \cite{Cohn95}.
We also make use of some new machinery including the hyporational functions and the hypo-Jacobian matrix defined later in this paper.

To state the result, fix a positive integer $g$ and let $\x=(x_1,\dots,x_g)$ denote a tuple of freely noncommuting indeterminants.
A free polynomial (in $g$-variables) is a finite $\C$ linear combination of words in $x$. For positive integers $n$,
	let $M_n(\C)^g$ denote the $g$-tuples of $n\times n$ matrices over $\C$ and let $M(\C)^g$ denote the sequence $(M_n(\C)^g)_n$. 
A free polynomial $f$ induces a sequence of maps $f[n]:M_n(\C)^g\to M_n(\C)$ by evaluation, $X\mapsto f(X)$. 
We let $f:M(\C)^g\to M(\C)$ denote this sequence.
A \dfn{free polynomial mapping} $p:M(\C)^g\to M(\C)^g$ is thus a $g$-tuple of sequences, $p=(p^1,\dots, p^{g})$, 
	that is, each $p^i$ is a free polynomial.
The polynomial mapping $p$ is injective (resp. surjective, bijective) if each $p[n]$ is injective (resp. surjective, bijective).  
Of course if $p[n]$ is injective, then considered as a polynomial in $g n^2$ commuting variables, 
	it is bijective and has a polynomial inverse.
 The following free polynomial analog of Grothendieck's Theorem was implicitly
 conjectured in \cite{Pas14}.
 
\begin{restatable*}[(Free Grothendieck Theorem)]{theorem}{THMfreeGroth}
	\label{thm:Free Groth Thm}
	If $p:M(\C)^g\to M(\C)^g$ is an injective free polynomial mapping, 
  then there is a free polynomial mapping $q$ such that $p\circ q(x)=x =
 q\circ p(x)$; that is, $p$ has a free polynomial inverse.
\end{restatable*}

% %It is natural to connect free polynomial mappings to endomorphisms of the free algebra $\C\fax$.
% %If $\phi$ is an endomorphism of $\C\fax$ then we say it induces the polynomial mapping $p$ defined by
% %$p_i(\x) = \phi(x_i)$, $1\leq i\leq g$. 
% %Applying results from \cite{DL82} and \cite{Sch85} we arrive at the following complete version of the free Jacobian Conjecture.
% %
% %\begin{theorem}
% %	Suppose $\phi$ is an endomorphism of $\C\fax$ and $\phi$ induces the free polynomial mapping $p:M(\C)^g\to M(\C)^g$.
% %	The following are equivalent.
% %	\begin{enumerate}
% %		\item $p$ is injective;
% %		\item $J(\phi)$, the Jacobian matrix of $\phi$ is invertible over $\C\fax\otimes \C\fralg{\bby}$;
% %		\item $\phi$ is an epic endomorphism;
% %		\item $\phi$ is an automorphism.
% %	\end{enumerate}
% %\end{theorem}

% We show that in general the (formal) inverse of a \dfn{free polynomial mapping} $p = (p^1,\dots, p^{\,g}):M(\C)^g\to M(\C)^g$ 
% 	is a free algebraic function and if the inverse is hyporational (the relevant definitions appear later in this introduction), 
% 	then it is a polynomial if and only if $p$ is injective.
Before describing our methods in further detail, we pause to note that Theorem~\ref{thm:Free Groth Thm} is of course
	related to the study of automorphisms, and the question of tame versus wild automorphisms of the free algebra 
	(see for instance \cite{D82,DY07,DY07b,M70,U07}).
Pascoe \cite{Pas14} proves a free (freely noncommutative) inverse function theorem and uses this theorem to establish a 
	free analog of the Jacobian conjecture, stated below.

\begin{restatable}[Free Jacobian Conjecture]{thmx}{THMfreeJacConj}
	\label{thm:Free Jacobian Conjecture}
	Suppose $p:M(\C)^g\to M(\C)^g$ is a free polynomial mapping.
	The following are equivalent:
	\begin{enumerate}[label=(\roman*)]
		\item $Dp(X)[H]$ is nonsingular for each $X\in M(\C)^g$; that is, for each positive integer $n$ and
                  each tuple $X\in M_n(\C)^g$, the linear  mapping $M_n(\C)^g \ni H\mapsto Dp(X)[H]$ is non-singular; 
		\item \label{it:injective} $p$ is injective;
		\item \label{it:bijective} $p$ is bijective;
		\item \label{it:agrees w poly} $p^{-1}$ exists as a {\it free function},
    and for each $n$, $p^{-1}\lvert_{M_n(\C)^g}$ agrees with a free polynomial mapping;
    	namely, there exists a free mapping $g:M(\C)^g\to M(\C)^g$ and free polynomial
  		mappings $q_n$ such that $p(g(X))=X=g(p(X))$ for all $X\in M(\C)^g$ and $g(X)=q_n(X)$ for each $n$
 		and $X\in M_n(\C)^g.$
	\end{enumerate}
\end{restatable}
	
%In particular, he shows that a \dfn{free polynomial mapping} $p:M(\C)^g\to M(\C)^g,$
%\[
%	p = \begin{pmatrix} p^1 &\cdots & p^g\end{pmatrix},
%\]
%	is injective if and only if its derivative is
%pointwise non-singular if and only if it is invertible as a free function.
The notion of a free function is defined in Subsection \ref{ssec:Free Bird} and the free derivative $Dp$ is defined
	in Subsection \ref{ssec:scions}. 
We will often use the equivalence of items \ref*{it:injective}
 and \ref*{it:bijective}.

Three results in this article require little or no additional overhead to state. 
Assuming $p$ is injective, Theorem~\ref{thm:deg p bound} produces bounds for the degree of its inverse $q$.
As a concrete example, $\deg(q)\leq (3^g\prod_{i=1}^{g}i^3)(\deg(p)-1)+1$.
Using the degree bound, Corollary~\ref{cor:eitheror} describes an algorithm that takes as input a free polynomial $p$ and, 
	after a number of iterations depending only on the number of variables $g$ and the degree of $p$, either produces a 
	polynomial $q$ - the inverse of $p$ - or $p$ is not injective and has no polynomial inverse.

The derivative $Dp(\x)[\h]$ is a $g$-tuple of polynomials in the $2g$ freely noncommuting variables 
$(\h,\x) = (h_1,\dots,h_g,x_1,\dots,x_g)$
defined as the free analog of the directional derivative in the obvious way.
% is a $g\times g$ matrix of polynomials in the $2g$
%freely noncommuting variables $(x_1,\dots,x_g,h_1,\dots,h_g)$ defined in
%the obvious fashion.  
The result of Pascoe mentioned above -- $p$ is bijective if and only if $\h\mapsto Dp(\x)[\h]$ is pointwise non-singular --  
is strengthened by the following result.

\begin{restatable*}{corollary}{CORFpolyinv}
	\label{cor:F poly inv}
	A free polynomial $p:M(\C)^g\to M(\C)^g$ is bijective if and only if $(\h,\x)\mapsto (Dp(\x)[\h],\x)$ has a polynomial inverse.
\end{restatable*}

We use the free derivative to state and prove Theorem \ref{thm:fps implicit func thm}, 
	the implicit function theorem for nc formal power series.
It is mostly a consequence of Lemma \ref{lem:ak and bbd}.
%Theorem \ref{thm:fps implicit func thm} is a mostly a consequence of Lemma \ref{lem:ak and bbd}, we do not prove it until
%Section \ref{sec:Reduce} since its statement involves the free derivative.
We refer to \cite{AgMc16} for an in depth analysis of the implicit function theorem for several topologies on $M(\C)^g$.

\begin{restatable*}[(Implicit function theorem)]{theorem}{THMfreeIFT}
	\label{thm:fps implicit func thm}
	Suppose $f(\x,\z) \in (\C\fps{\bbx\cup\bbz})^h$.
	If $f(0,0) = 0$ and ${\partial f}/{\partial \z}(0,0)\in M_h(\C)$ is invertible, then there exists a unique 
	$\mathfrak{g}\in (\C\fpsx)^h$ such that $\mathfrak{g}(0) = 0$ and $f(\x,\mathfrak{g}(\x)) = 0$.	
\end{restatable*}

\subsection{The Jacobian,  free algebraic functions and proper algebraic systems}
The \dfn{left Jacobian matrix} \cite{Reu92} of a free polynomial 
mapping $p:M(\C)^g\to M(\C)^g$ with no constant term is the unique 
$g\times g$ matrix 
$J_p$ with free polynomial  entries such that
\[
	p(\x) = \begin{pmatrix} p^1(\x) &\cdots & p^{ g}(\x)\end{pmatrix}
		= \x J_p(\x) = \begin{pmatrix} \x_1 &\cdots & \x_g\end{pmatrix} J_p(\x).
\]
In particular,
\[
	p^{j}(\x) = \sum_{s=1}^g \x_s (J_p)_{s,j}(\x).
\]

The definition of the Jacobian matrix $J_p$ extends naturally to the case
where each $p^{j}$ is a free formal power series with no constant term. 
In this case $J_p$ is a $g\times g$ matrix with free formal power series
entries. It has a \dfn{multiplicative inverse} if there is a $g\times g$
matrix $\mathcal{J} _p$ of free formal power series such that $J_p(\x)\mathcal{J}_p(\x) =I_g$ 
and $\mathcal{J}_p(\x) J_p(\x)=I_g$. In this case $\mathcal{J}_p$ is unique
and denoted $J_p^{-1}$. The following proposition combines
Corollary~\ref{cor:inv mat reps} and Lemma~\ref{lem:ak and bbd}.

\begin{restatable*}{proposition}{PROPcompinv}
	\label{prop:comp inv}
 Suppose $p$ is a free formal power series mapping without constant term.
 There is a free formal power series mapping $q$ (without constant term) such 
 that $p$ and $q$ are compositional inverses if and only if $J_p$ has
 a (free formal power series) multiplicative inverse. In this case, $J_p(q(\x))=J_q^{-1}(\x)$.
 Moreover, $q$ is the unique solution of 
\begin{equation}
	\label{eq:qisalgebraic}
	q(\x) = \x J_p^{-1}(q(\x)).
\end{equation}
%\end{proposition}
\end{restatable*}

In the case that $J_p^{-1}$ is a polynomial, equation \eqref{eq:qisalgebraic} 
implies $q$, the inverse of $p$, is algebraic. 
To state the result more precisely requires a definition. 
Suppose $\z = (z_1,\dots,z_h)$ is an additional  tuple of freely noncommuting
 variables and 
	$\alpha(\x)[\z] =\begin{pmatrix} \alpha^1 & \cdots & \alpha^h\end{pmatrix}$ is a polynomial mapping.
We say $\alpha$ is a {\bf proper algebraic polynomial mapping}\index{proper algebraic polynomial!mapping} if $\alpha$ 
	has no constant terms and each monomial appearing in $\alpha$ with degree in $z$ of at least one, has total degree of at least two.
A tuple of free formal power series without constant term, $\beta(\x)=\begin{pmatrix} \beta^1 & \cdots & \beta^h\end{pmatrix}$,
	is a \dfn{solution} to the proper algebraic polynomial $\alpha$ if
\[
 \alpha(\x)[\beta(\x)]=\beta(\x).
\]
We say each $\beta^i$ is a \dfn{component} of the solution.
By \cite[Theorem~6.6.3]{Sta11}, every proper algebraic polynomial mapping has a unique solution.
A formal power series $\gamma(x)$ is \dfn{algebraic} if $\gamma - c$ is a component of the solution to 
	some proper algebraic polynomial mapping.

%We would like to find conditions that guarantee a polynomial mapping $p$ has a Jacobian matrix $J_p$ whose multiplicative inverse
%	$J_p^{-1}$ is a polynomial matrix.
If both $p$ and $p^{-1}$ are a polynomial mappings, then the chain rule implies $J_p^{-1}$ is a polynomial matrix 
	(Remark~\ref{rem:poly inv PMI} and \cite[Corollary~1.4]{Reu92}).
Thus the following theorem follows immediately from 
 Theorem~\ref{thm:Free Groth Thm}. We give an independent proof
 and the result itself is a key ingredient in the proof of Theorem 
 \ref{thm:Free Groth Thm}.

\begin{restatable*}{theorem}{THMpolymatinv}
		\label{thm:PMI is poly}
 If $p$ is a bijective polynomial mapping, then $J_p^{-1}$ is a polynomial (matrix).
\end{restatable*}

Example \ref{ex:classic alg} concisely points out the limitations of the Jacobian matrix; it does not detect the non-injectivity
of a polynomial.
	
%Combining Theorem~\ref{thm:PMI is poly} and Proposition~\ref{prop:comp inv} obtains the following corollary.
%
%\begin{restatable*}{corollary}{CORbijpolyalginv}
%		\label{cor:bij poly alg inv}
% If $p$ is a bijective polynomial, then its inverse $q$ is algebraic.
%\end{restatable*}

%Def of proper algebraic system and free algebraic function.

%If $J_p$ inverse is a polynomial, then $q$ is algebraic: Combine Propositions\ref{prop:q alg} and
%\ref{thm:PMI is poly}. 

\subsection{Hyporational functions}
If $p$ is a bijective free polynomial, then necessarily its inverse $q$ is an algebraic mapping.
If in addition, $q$ is rational, then  \cite[Theorem~4.2]{KlVo17} implies $q$ is a polynomial. 
In Section~\ref{sec:Hyporationals} we identify, in terms of proper algebraic polynomial mappings, free rational functions amongst 
	free algebraic mappings and extend,
 in Theorem~\ref{thm:hyporat domain poly} below, 
 \cite[Theorem~4.2]{KlVo17} to a larger class of free algebraic functions.

In fact, a formal power series is rational if and only if it is a component of the solution of some proper algebraic polynomial $\alpha$
	of the form,
\[
	\alpha(\x)[\z] = \mfa(\x) + \z \mathscr{A}(\x),
\]
	where $\mfa$ is a polynomial mapping and $\mathscr{A}$ is a polynomial matrix.
On the other hand,
example \ref{ex:classic alg revisited} shows that the solution to a
 proper algebraic polynomial mapping having degree one in $z$ is not
necessarily rational. 
% sufficient	for its solution to be rational, thus
%Thus we must extend our study to a more general class of formal power series, 
%	the hyporational series, defined below.

%A tuple of formal power series, $\beta(\x)=\begin{pmatrix} \beta^1 & \cdots & \beta^h\end{pmatrix}$
%is a \dfn{hyporational mapping} if it is the solution to a proper
A formal power series $a(\x)$ with constant term $a_\ew$ is a \dfn{hyporational series} if $a- a_\ew$ 
		is a component of the solution to a proper
algebraic polynomial mapping $\alpha(\x)[\z]$ of degree one in $z$.
Every rational series is a hyporational series and Example \ref{ex:classic alg revisited} 
shows this inclusion is proper. Hence hyporational functions lie
properly between free rational functions and free algebraic functions.
The following result shows that hyporationals enjoy some of the same regularity properties as rationals.

\begin{restatable*}{theorem}{THMhyporatdomainpoly}
	\label{thm:hyporat domain poly}
	Suppose $a$ is hyporational.
	If $\dom_n(a) = M_n(\C)^g$ for all $n$, then $a$ is a free polynomial.
\end{restatable*}

%\begin{restatable*}{corollary}{CORpbijqhyp}
%	\label{cor:bij hyp is poly}
%	If  $p:M(\C)^g\to M(\C)^g$ is a bijective free polynomial and 
%its inverse
%	$q$ is a free hyporational mapping,  then $q$ is a free polynomial 
%mapping.
%\end{restatable*}

In Section~\ref{sec:Hyporationals} we introduce the \dfn{hypo-Jacobian matrix}
$\hypj{p}$ of a free polynomial mapping $p$. It is a $g\times g$ matrix 
whose entries are bipartite polynomials; that is polynomials in 
the two $g$-tuples of freely noncommuting variables $\x$ and $\y$, but
where $x_j y_k = y_k x_j$ for all $1\le j,k\le g$. 
See 
Lemma~\ref{lem:hyp jac matrix}  and Definition \ref{def:hypo-Jacobian}.
%constructs the hypo-Jacobian matrix of a polynomial, a generalization of the Jacobian matrix.
Theorem~\ref{thm:matrix free inverse function theorem} shows that the
 hypo-Jacobian matrix of a free polynomial mapping is simply a matrix form of the 
free derivative; the hypo-Jacobian's invertibility as a matrix encodes 
the invertibility of free polynomials. Indeed, we obtain the following
improvement of Theorem~\ref{thm:PMI is poly}.

\begin{restatable*}{theorem}{THMmatfreeinverse}
 \label{thm:matrix free inverse function theorem}
  The free polynomial mapping $p$ is injective if and only if $\hypj{p}$
 has a multiplicative inverse whose entries are bipartite polynomials.
\end{restatable*}

In other words, a free polynomial mapping $p$ is injective if and only if its hypo-Jacobian matrix is invertible (as a bipartite
polynomial matrix).

The notion of the hypo-Jacobian matrix arises in the study of endomorphisms of the free associative algebra $\C\fax$.
In fact, any such endomorphism has a Jacobian matrix (see \cite{DL82} and \cite{Sch85}) that exactly corresponds with our notion of
the hypo-Jacobian matrix.

%Overall, these results extend certain regularity notions from \cite{KPV17} to the class of hyporationals.

% Theorem~\ref{thm:hyporat domain poly} - and say it extends results of 
%\cite{KPV17} Klep, Pascoe, Volcic. And its corollary. Theorem used
% in proofs of ?? which of results stated in this intro??

\subsection{Reader's guide}
Section~\ref{sec:preliminaries} introduces definitions and notation from formal power series and free analysis that
are repeatedly used throughout the paper.

The Jacobian matrix of a formal power series is defined in Section~\ref{sec:JacMats and FAF} and
	it serves as one of the central objects of study.
%The Jacobian matrix is appropriately named since it obeys the chain rule and 
Invertibility of the Jacobian matrix
	is necessary and sufficient for a formal power series mapping  to have a compositional inverse (Proposition~\ref{prop:comp inv}).
 In subsection \ref{ssec:AI and CI} we 
borrow  ideas from enumerative combinatorics - namely the construction of an algebraic
	formal power series by iterating the composition of a set of polynomials - and
 exploit the chain rule for the Jacobian matrix
  to iteratively construct these compositional inverses.
%a formal power series mapping's 	compositional  inverse.
%The construction of the compositional inverse borrows ideas from enumerative combinatorics, namely the construction of an algebraic
%	formal power series by iterating the composition of a set of polynomials.
Subsection \ref{ssec:FAF and local inv} extends results about Jacobian matrices to free analytic functions. 
 These results are then combined with  a noncommutative Nullstellensatz -- due to man \cite{HM04} -- 
  to prove the key intermediate result Theorem~\ref{thm:PMI is poly}: 
  if a free polynomial is injective,  then its Jacobian matrix has a polynomial matrix inverse.

Subsection \ref{ssec:Polynomial Criteria} explores conditions that guarantee a free polynomial has a free polynomial inverse.
While in subsection \ref{ssec:scions} we recall the free derivative as defined in \cite{Pas14}
 and investigate its properties.
For a fixed free polynomial $p$, we define the function $F:(\x,\y)\mapsto (Dp(\y)[\x],\y)$ and observe that Pascoe's solution \cite{Pas14} to 
	the free Jacobian conjecture can be interpreted as saying,  $p$ is bijective if and only if $F$ is bijective.
Setting $G$ equal to the (free) inverse of $F$, Lemma~\ref{lem:F poly inv} shows that there is a $\z$-affine linear 
	mapping $\bbG$ such that the first $g$ entries of $G$ are the solution to $\bbG$.
Understanding the inverse function $G$ is what motivates Section~\ref{sec:Hyporationals}.

Section~\ref{sec:vals and rats} seemingly departs from the previous discussion and establishes facts about
noncommutative rational functions and rational degree maps needed in the following section.
The main result of this section is Proposition~\ref{prop:r deg bdd}.
 It  shows that evaluating a nc rational function $r$ on matrices produces a matrix whose
entries behave much like the abelianization of $r$.

Section~\ref{sec:Hyporationals} introduces the hyporationals, a generalization of rational formal power series.
We proceed to show that	the free derivative of an injective polynomial has a hyporational inverse.
If $s$ is a hyporational series that is not rational, then we cannot apply results from realization theory.
%Proposition~\ref{prop:rational mat rep} precisely delineates the difference between rational series and hyporational series;
%	a formal power series $r$ is rational if and only if there exists a polynomial mapping $\mfa$ and a polynomial matrix $\mathscr{A}$
%	such that
%	\[
%		r(\x) = \left(\mfa(\x)\left(I-\mathscr{A}(\x)\right)^{-1}\right)_1.
%	\]
%Thus, if $s$ is a hyporational series that is not rational, then no such matrix and polynomial exist.
However, $s[n] = s\lvert_{M_n(\C)^g}$ is a commutative rational function for each $n$, hinting that it may be possible to extend 
	regularity results from nc rational functions to hyporational functions.
Proposition~\ref{prop:hyporat has hypomat} does so by constructing the \dfn{hypomatrix representation} of hyporational function; 
a matrix over $\C\fax\otimes \C\fralg{\bby}$, the algebra of bipartite polynomials, that imitates the realization theory of 
nc rational functions.
The algebra $\C\fax\otimes \C\fralg{\bby}$ is contained in a skew field $\C\skf{\bbx\lra\bby}$, and the hypomatrix representation
is invertible as a matrix over $\C\skf{\bbx\lra\bby}$.
Thus, we may use the results of Section~\ref{sec:vals and rats} to analyze hyporational functions.

By applying Proposition~\ref{prop:r deg bdd} we prove Theorem~\ref{thm:hyporat domain poly}: a hyporational function
	with no domain exceptions is in fact a polynomial, a result established in \cite{KlVo17} and \cite{KPV17} for 
  free rational functions. 
A straightforward consequence is Corollary~\ref{cor:F poly inv}. It  strengthens Pascoe's resolution of the  Free Jacobian Conjecture
  by asserting: 
	a free polynomial $p$ is injective if and only if $(\h,\x)\mapsto (Dp(\x)[\h],\x)$ has a polynomial inverse.
 This corollary is both an ingredient in, and  immediate consequence of, Theorem~\ref{thm:Free Groth Thm} assuming 
 bijectivity of $F(\x,\y)= (Dp(\y)[\x],\y)$. 

In subsection \ref{ssec:bij crit} we introduce $\hypj{p}$, the hypo-Jacobian matrix of the free polynomial $p$.
Using Corollary~\ref{cor:F poly inv} we prove Theorem~\ref{thm:matrix free inverse function theorem}:
a free polynomial is injective if and only if its hypo-Jacobian matrix is invertible as a matrix of bipartite polynomials.
Connecting Theorem~\ref{thm:matrix free inverse function theorem} to results in \cite{DL82} and \cite{Sch85}
proves the Free Grothendieck Theorem, Theorem~\ref{thm:Free Groth Thm}.

Lastly, in Section~\ref{sec:compute} we discuss computational aspects of computing the free inverse $q$ of a given free polynomial $p$.
If $q$ is a free polynomial then Theorem~\ref{thm:deg p bound} provides an upper bound for the degree of $q$ 
	depending only on the number of variables and the degree of $p$,
 leading to 
 an algorithmic test for whether $p$ has a polynomial inverse, Lemma~\ref{lem:mfq max bound}.

\subsection*{Acknowledgements} The author would like to thank his advisor Scott McCullough for reviewing the manuscript and for
the many valuable discussions along the way. The author would also like to thank Igor Klep, Jurij Vol{\v c}i{\v c}, James Pascoe
 and {\v S}pela {\v S}penko 
for many helpful conversations and for pointing out the existence of several algebraic results.

\section{Preliminaries}
	\label{sec:preliminaries}
Let $\mathcal{A}$ be any $\C$-algebra.
We denote the $n\times n$ matrix algebra with entries in $\mathcal{A}$ by $M_n(\mathcal{A})$.
Let $\bbx = \set{x_1,\dots,x_g}$ be a set of noncommuting indeterminates.
The set of finite sequences of elements of $\bbx$ is denoted by $\fralg{\bbx}$.
The empty sequence is the identity element of $\fralg{\bbx}$ and is denoted by $\ew$.

An element of $\bbx$ is called a \dfn{letter}\index{word!letter}, an element of $\fax$ is called a \dfn{word} and the
{\bf length}\index{word!length} of a word $w = x_{i_1}\dots x_{i_m}$ is $m$, denoted by $\abs{w}$.
We denote the \dfn{algebra of free formal power series} with coefficients in $\cA$ by $\cA\fpsx$ and 
if $f\in\cA\fpsx$ then
\[
	f = \sum_{w\in\fax} c_w w,
\]
where each $c_w\in\cA$.

We say $p\in \cA\fpsx$ is a \dfn{polynomial} if all but a finite number of the
coefficients of $p$ are zero.
The set of all polynomials, denoted $\cA\fax$, is the familiar free algebra on $g$ noncommuting indeterminates.
It is a subalgebra of $\cA\fpsx$.
We denote the formal power series with no constant term by $\cA\fpsx_+$ and the formal polynomials 
with no constant term by $\cA\fax_+$.

Suppose $\alpha = \sum_w a_w w$ and $\beta = \sum_w b_w w$.
Define $\omega: \cA\fpsx\times\cA\fpsx\to \N\cup\set{\infty}$ by
\[
	\omega(\alpha,\beta) = \inf\set{n\in\N:\exists w\in\fax,\abs{w}=n \text{ and } a_w\neq b_w}.
\]
The function $d:\cA\fpsx\times \cA\fpsx\to \R$ given by $d(\alpha,\beta) = 2^{-\omega(\alpha,\beta)}$
is a metric on $\cA\fpsx$.
Furthermore, $\cA\fpsx$ is complete and $\cA\fax$ is dense in $\cA\fpsx$.
The metric topology above is equivalent to the $(\bbx)$-adic topology.

Formal power series may be generalized further to free products of
unital $\C$-algebras.
An easy example of such a power series is a polynomial $p\in \C\fralg{\bbx\cup \bbz}$, which can instead be taken
as a polynomial in the free product of $\C\fralg{\bbx}$ with $\C\fralg{\bbz}$.
The free product of $\C\fralg{\bbx}$ and $\C\fralg{\bbz}$ is the set of all words $\alpha^1\beta_1\dots\alpha^k\beta_k$,
where $\alpha^i\in \fralg{\bbx}$, $\beta_i\in\fralg{\bbz}$ are nonempty words.
A much more detailed exposition can be found in \cite{Vol15}.

\begin{definition}
	\label{def:rational series}
	If $S\in \C\fpsx$ and $S$ has a nonzero constant term $\rho$, then $S^{-1}$ the multiplicative inverse of $S$,
	exists and is given by
	\[
		S^{-1} = \frac{1}{\rho}\sum_{n\geq 0} \left(1-\frac{S}{\rho}\right)^n.
	\]

	Let $\C_{\textup{rat}}\fpsx$ denote the \dfn{algebra of rational series}; the smallest subalgebra of 
	$\C\fpsx$ containing $\C\fax$ such that if $S\in \C_{\textup{rat}}\fpsx$ and $S^{-1}$ exists, 
	then $S^{-1}\in \C_{\textup{rat}}\fpsx$.
\end{definition}

%
%    ######  ##     ## ########   ######  ########  ######  ######## ####  #######  ##    ## 
%   ##    ## ##     ## ##     ## ##    ## ##       ##    ##    ##     ##  ##     ## ###   ## 
%   ##       ##     ## ##     ## ##       ##       ##          ##     ##  ##     ## ####  ## 
%    ######  ##     ## ########   ######  ######   ##          ##     ##  ##     ## ## ## ## 
%         ## ##     ## ##     ##       ## ##       ##          ##     ##  ##     ## ##  #### 
%   ##    ## ##     ## ##     ## ##    ## ##       ##    ##    ##     ##  ##     ## ##   ### 
%    ######   #######  ########   ######  ########  ######     ##    ####  #######  ##    ##
%
\subsection{Free analysis}
	\label{ssec:Free Bird}
We give basic definitions and a few results in free analysis that will be used throughout the paper.
	
A free polynomial is a noncommutative polynomial evaluated on tuples of matrices that preserves the structure of
free sets.
A \dfn{free set} $\Gamma = (\Gamma[n])_{n=1}^\infty\subset M(\C)^g = (M_n(\C)^g)_{n=1}^\infty$ 
is a graded set of tuples of matrices that is closed under direct sums and conjugation by similarities.
That is, if $X\in\Gamma[n]$, $Y\in \Gamma[m]$ and $S\in \textup{GL}_n(\C)$ then
\begin{enumerate}[label=(\roman*)]
	\item $X\oplus Y \in \Gamma[n+m]$;
	\item $S^{-1}XS\in \Gamma[n]$,
\end{enumerate}
where
\[
	X\oplus Y = (X_1,\dots,X_g)\oplus (Y_1,\dots,Y_g) = (X_1\oplus Y_1,\dots,X_g\oplus Y_g)
\]
and
\[
	S^{-1}XS = (S^{-1}X_1S,\dots,S^{-1}X_gS).
\]
%For positive integers $n$, let $M_n(\C)^g$ denote the set of $g$-tuples of $n\times n$ matrices with complex
%entries.
%Let $M(\C)^g$ denote the sequence $(M_n(\C)^g)_{n=1}^\infty$.
%It is quite natural to think of $M(\C)^g$ as the graded union of the $M_n(\C)^g$.
%More gnerally, a {\bf subset}\index{free set!subset} $U$ of $M(\C)^g$ is a sequence
%$(U[n])_{n=1}^\infty$ where $U[n]\subset M_n(\C)^g$.
%The subset $U$ is a \dfn{free set} if it is closed under direct sums and similarity, i.e. if $X\in U[k]$ and $Y\in U[\ell]$, then
%\[
%	X\oplus Y 
%	= \left(\bpm X_1 & 0 \\ 0 & Y_1 \epm,\dots, \bpm X_g & 0 \\ 0 & Y_g \epm\right)
%	\in U[k+\ell],
%\]
%and if $S\in M_k(\C)$ is invertible, then
%\[
%	S^{-1}XS = (S^{-1}X_1S,\dots, S^{-1}X_gS)\in U[k],
%\]
Let $U\subset M(\C)^g$ be a free set.
A \dfn{free map} (or \dfn{free function}) $f$ from $U$ into $M(\C)$ is a sequence of functions 
$f[n]:U[n]\to M_n(\C)$ that respects the free structure of $U$;
$f(X\oplus Y) = f(X)\oplus f(Y)$ and $f(S^{-1}XS) = S^{-1}f(X)S$.
The notion of a free map extends easily to vector-valued functions $f:U\to M(\C)^h$, 
matrix-valued functions $f:U\to M_d(M(\C)^g)$ and even operator-valued functions.
If $f_i:U\to M(\C)$ is a free map for $1\leq i\leq h$ then we say the tuple $f=(f_1,\dots,f_h)$ 
is a \dfn{free mapping} and write $f:U\to M(\C)^h$.

Suppose $U\subset M(\C)^g$ is a free set and each $U[n]$ is open (as a subset of $M_n(\C)^g$).
A free map $f:U\to M(\C)$ is \dfn{continuous} if each $f[n]$ is continuous, and is ({\bf free}) {\bf analytic} \index{free analytic}
if each $f[n]$ is analytic.
As shown in \cite{K-VV14}, a free function that is continuous is also free analytic (see also \cite{HKM11}).

As one would hope, there are indeed connections between free analytic functions and formal power series.
In fact, a formal power series with a positive radius of convergence determines a free analytic function and
with a small degree of local boundedness we get the converse (see \cite{HKM12}).

Given a positive integer $d$ let $f\in M_d(\C)^h\fpsx$, that is
\[
	f = \sum_{m=0}^\infty \sum_{\substack{w\in\fax \\ \abs{w}=m}} c_w w,
\]
where $c_w\in M_d(\C)^h$.
For $X\in M_n(\C)^g$ and $w = x_{i_1}\dots x_{i_k}\in \fralg{\bbx}$, we say $w(X) = X^w = X_{i_1}\dots X_{i_k}$.
We define the evaluation of $f$ at $X$ by
\[
	f(X) = \sum_{m=0}^\infty \sum_{\substack{w\in\fax \\ \abs{w}=m}} c_w\otimes X^w,
\]
provided this series converges.

\section{Jacobian matrices and free analytic functions}
	\label{sec:JacMats and FAF}

Fix $g\in \Z^+$ and set $\x = (x_1,\dots,x_g)\in \fax^g$, with $\x$ considered as a row vector.
For $h\in\Z^+$ and $1\leq i\leq h$, let $\phi_i\in\C\fpsx$ and $\phi = (\phi_1,\dots,\phi_h)\in (\C\fpsx)^h$.
Alternatively, we can view $\phi$ as an element of $\C^h\fpsx$, the set of formal power series with coefficients in $\C^h$.

Let $h,k\in\Z^+$.
Suppose $\bby = \set{y_1,\dots,y_h}$ and $\bbz = \set{z_1,\dots, z_k}$ are sets of freely noncommuting indeterminates
and suppose $\phi$ has no constant term, that is, $\phi\in(\C\fpsx_+)^h$.
In this case, we may view $\phi$ from a much more algebraic perspective;
$\phi:\C\fps{\bby}\to \C\fpsx$ is a continuous homomorphism defined by $y_i\mapsto \phi_i$.

%
%    ######  ##     ## ########   ######  ########  ######  ######## ####  #######  ##    ## 
%   ##    ## ##     ## ##     ## ##    ## ##       ##    ##    ##     ##  ##     ## ###   ## 
%   ##       ##     ## ##     ## ##       ##       ##          ##     ##  ##     ## ####  ## 
%    ######  ##     ## ########   ######  ######   ##          ##     ##  ##     ## ## ## ## 
%         ## ##     ## ##     ##       ## ##       ##          ##     ##  ##     ## ##  #### 
%   ##    ## ##     ## ##     ## ##    ## ##       ##    ##    ##     ##  ##     ## ##   ### 
%    ######   #######  ########   ######  ########  ######     ##    ####  #######  ##    ##
%
\subsection{The Jacobian matrix of a formal power series}
	\label{ssec:jac mat of FPS}
	
We define the (left) noncommutative Jacobian matrix of a formal power series, a central object of study throughout the paper.
A treatment of the noncommutative Jacobian matrix can be found in \cite{Reu92}.

\begin{definition}
	\label{def:Jacobian matrix}
	Let $f\in \C\fpsx$ with $f = \sum_{w\in\fax} f_w w$.
	For $1\leq i\leq g$, define $S_{x_i}^*:\C\fpsx\to\C\fpsx$ by
	\[
		S_{x_i}^*f = \sum_{w\in\fax} f_{x_i w} w.
	\]
	In other words, $S_{x_i}^*$ is the adjoint of the operator of left multiplication by $x_i$.
	
	Let $h\in\mathbb{Z}^+$ and take $\phi\in (\C\fpsx)^h$, seen as a row vector of formal power series.
	The $g\times h$ matrix over $\C\fpsx$ defined by
	\[
		J_\phi = (S_{x_j}^* \phi_i)_{i,j=1}^{g,h}
	\]
	is the {\bf (left)} \dfn{Jacobian matrix} of $\phi$. 
	In particular, if $\phi$ has constant term $\phi_\ew = (\phi_\ew^1\ew,\dots, \phi_\ew^h\ew)$, then
	\[
		\phi = \phi_\ew + \x J_\phi,
	\]
	where $\x J_\phi$ is the standard product of a row vector and a matrix.
	This representation of $\phi$ is unique.
\end{definition}

\begin{remark}
	Let $\phi\in (\C\fpsx_+)^h$ and define the homomorphism $\alpha:\C\fps{\bby}\to\C\fpsx$ by $\alpha(y_i) = \phi_i$.
	Defining $J_\alpha = J_\phi$ yields the Jacobian matrix encountered in \cite{Reu92}.
\end{remark}

It is evident that every formal power series has a Jacobian matrix and if $\alpha,\beta\in(\C\fpsx)^h$
have the same Jacobian matrix then $\alpha-\beta\in \C^h$.

\begin{remark}
	If $\phi:\C\fps{\bby}\to\C\fpsx$ and $\psi:\C\fps{\bbz}\to\C\fps{\bby}$ are continuous homomorphisms, then
	certainly $\phi\circ \psi : \C\fps{\bbz}\to\C\fpsx$ is a continuous homomorphism.
	As tuples of formal power series this says that $\psi(\phi(\x))\in (\C\fpsx_+)^k$.
	This aligns with the fact that $\psi(\phi(\x))$ is defined as long as $\phi$ has a zero constant term.
	
	For any $A\in M_{k\times n}(\C\fps{\bby})$ and $\phi\in (\C\fpsx)^h$, $A = A(\y)$ and $A(\phi(\x)) \in M_{k\times n}(\C\fpsx)$,
	where $A(\phi(\x))$ is the result of applying the homomorphism $y_i\mapsto \phi_i$ to each entry of $A$.
	Thus, if $B\in M_{n\times m}(\C\fps{\bby})$ then $A(\phi(\x))B(\phi(\x)) = AB(\phi(\x))\in M_{k\times m}(\C\fpsx)$.
	In particular, if $C\in M_n(\C\fps{\bby})$ is an invertible matrix then $C^{-1} = C(\y)^{-1}$ and 
	$C^{-1}(\phi(\x)) = C(\phi(\x))^{-1}$.
\end{remark}

\begin{proposition}
	\label{prop:Jacobian chain rule}
	If $\phi\in (\C\fpsx_+)^h$ and $\psi\in (\C\fps{\bby}_+)^k$ then
	\[
		J_{\psi\circ \phi}(\x) = J_{\phi}(\x)J_{\psi}(\phi(\x))\in M_{g\times k}(\C\fpsx).
	\]
\end{proposition}

\begin{proof}
	Observe $\psi\circ \phi\in (\C\fpsx)^k$.
	Define $\alpha:\C\fps{\bby}\to \C\fpsx$ and $\beta:\C\fps{\bbz}\to\C\fps{\bby}$ by
	$\alpha(y_i) = \phi_i$ and $\beta(z_j) = \psi_j$.
	Thus, $\alpha\circ \beta: \C\fps{\bbz}\to \C\fpsx$ with $\alpha\circ \beta( z_i)= (\psi\circ \phi)_i$.
	By Proposition~1.2 in \cite{Reu92}, $J_{\psi\circ \phi} = J_\phi(\x) J_\psi(\phi(\x)) \in M_{g\times k}(\C\fpsx)$.
\end{proof}

%\begin{restatable*}{corollary}{CORinvmatreps}
\begin{corollary}
		\label{cor:inv mat reps}
		Suppose $p,q\in (\C\fpsx_+)^g$ have Jacobian matrices $J_p$ and $J_q$, respectively.
		The series $p$ and $q$ are compositional inverses if and only if $J_p(q(\x)) = J_q(\x)^{-1}$
		and $J_q(p(\x)) = J_p(\x)^{-1}$.
%	\end{restatable*}
\end{corollary}
	
%\CORinvmatreps

\begin{proof}
	Suppose $p$ and $q$ are compositional inverses.
	Hence $p(q(\x)) = \x$ and $q(p(\x)) = \x$.
	Applying Proposition~\ref{prop:Jacobian chain rule},
	\[
		J_q(\x)J_p(q(\x)) = J_{p\circ q}(\x) = I_g
	\]
	and
	\[
		J_p(\x)J_q(p(\x)) = J_{q\circ p}(\x) = I_g.
	\]
	Thus $J_p(q(\x)) = J_q(\x)^{-1}$ and $J_q(p(\x)) = J_p(\x)^{-1}$.
	
	Now suppose $J_p(q(\x)) = J_q(\x)^{-1}$ and $J_q(p(\x)) = J_p(\x)^{-1}$.
	Observe
	\[
		p(q(\x)) = \x J_{p\circ q}(\x) = \x J_q(\x) J_p(q(\x)) = \x I_g = \x,
	\]
	and
	\[
		q(p(\x)) = \x J_{q\circ p}(\x) = \x J_p(\x) J_q(p(\x)) = \x I_g = \x.
	\]
	Therefore, $p$ and $q$ are compositional inverses.
\end{proof}

The invertibility of the Jacobian matrix is reminiscent of the inverse function theorem.
Indeed, if $p\in (\C\fpsx_+)^g$ is a free function, then $p$ is locally invertible at $0$
if and only if $J_p$ is invertible at $0$.

It should be noted that $J_{\alpha^{-1}}(\x) = J_{\alpha}^{-1}(\alpha^{-1}(\x))$, hence we cannot use $J_\alpha$
to directly compute $\alpha^{-1}$ without already knowing the explicit form of $\alpha^{-1}$.
However, $J_\alpha^{-1}$ is a local approximation of $J_{\alpha^{-1}}$, implying we may be able to construct $\alpha^{-1}$
from successive approximations.
This leads us directly to subsection \ref{ssec:AI and CI}.

%
%    ######  ##     ## ########   ######  ########  ######  ######## ####  #######  ##    ## 
%   ##    ## ##     ## ##     ## ##    ## ##       ##    ##    ##     ##  ##     ## ###   ## 
%   ##       ##     ## ##     ## ##       ##       ##          ##     ##  ##     ## ####  ## 
%    ######  ##     ## ########   ######  ######   ##          ##     ##  ##     ## ## ## ## 
%         ## ##     ## ##     ##       ## ##       ##          ##     ##  ##     ## ##  #### 
%   ##    ## ##     ## ##     ## ##    ## ##       ##    ##    ##     ##  ##     ## ##   ### 
%    ######   #######  ########   ######  ########  ######     ##    ####  #######  ##    ##
%
\subsection{Auxiliary inverses and compositional inverses}
	\label{ssec:AI and CI}
%Corollary~\ref{cor:inv mat reps} tells us that the automorphisms of $\C\fpsx$ are exactly those with invertible Jacobian matrices.
The main result in this subsection, Proposition~\ref{prop:comp inv}, tells us that if $p\in (\C\fpsx)^g$, such that $J_p$ is invertible, 
then $p^{-1}$ is the limit of a sequence of polynomials constructed from $J_p^{-1}$.

\begin{definition}
	\label{def:num of z terms}
	Suppose $h\in\Z^+$ and $\bbz = \set{z_1,\dots,z_h}$ is a set of freely noncommuting indeterminates.
	For any $w\in\fralg{\bbx\cup\bbz}$ define $\abs{w}_z$ to be the number of $z$-terms appearing in $w$ and define
	$\abs{w}_x$ to be the number of $x$-terms appearing in $w$.
	In particular, $\abs{w} = \abs{w}_x+\abs{w}_z$.
	
	Let $\ell\in \Z^+$ and $\alpha\in \C^\ell\fps{\bbx\cup\bbz}_+\cong (\C\fps{\bbx\cup\bbz}_+)^\ell$ with 
	$\alpha = \sum_{w\in \fralg{\bbx\cup\bbz}} a_w w$.
	Define
	\begin{equation}
		\label{eq:bbd}
		\bbd_z(\alpha) = \inf\set{\abs{w}:\abs{w}_z>0 \text{ and } a_w\neq 0}.
	\end{equation}
	Note if $\alpha$ has no $z$-terms then $\bbd_z(\alpha) = \infty$.
\end{definition}

We will consistently write a formal power series $\alpha\in(\C\fps{\bbx\cup\bbz})^\ell$ as $\alpha(\x)[\z]$ rather than
$\alpha(\x,\z)$.
This convention is simply a preference based on aligning our notation with the notation we use for free derivatives.
Thus, if $\beta\in(\C\fps{\bbx\cup\bbz}_+)^h$ then $\alpha(\x)[\beta(\x)[\z]]$ written another way is $\alpha(\x,\beta(\x,\z))$.

\begin{lemma}
	\label{lem:bbd props}
	Let $\alpha,\beta\in \C^\ell\fps{\bbx\cup\bbz}_+$ and $\gamma\in\C^h\fps{\bbx\cup\bbz}$ with 
	$\alpha = \sum_{w\in \fralg{\bbx\cup\bbz}} a_w w$, $\beta = \sum_{w\in \fralg{\bbx\cup\bbz}} b_w w$,
	and $\gamma = \sum_{w\in \fralg{\bbx\cup\bbz}} c_w w$.
	We have the following,
	\begin{enumerate}[label=(\roman*)]
		\item \label{it:bbd props item 1} $\bbd_z(\alpha + \beta) \geq \min\set{\bbd_z(\alpha),\bbd_z(\beta)}$;
		\item \label{it:bbd props item 2} $\bbd_z(\alpha\beta)\geq \min\set{\bbd_z(\alpha),\bbd_z(\beta)}$ and in particular, 
		if $\bbd_z(\alpha)$ and $\bbd_z(\beta)$ are not both infinite, then
		$\bbd_z(\alpha\beta) > \min\set{\bbd_z(\alpha),\bbd_z(\beta)}$;
		\item \label{it:bbd props item 3} if $\bbd_z(\alpha)>1$ and $\bbd_z(\gamma)<\infty$ then $\bbd_z(\gamma)<\bbd_z(\alpha(\x)[\gamma(\x)[\z]])$.
	\end{enumerate}
\end{lemma}

\begin{proof}
	Since $a_w+b_w$ is nonzero only if at least one of $a_w$ or $b_w$ is nonzero, we automatically have
	$\bbd_z(\alpha + \beta) \geq \min\set{\bbd_z(\alpha),\bbd_z(\beta)}$.
	Thus we have proven $\ref*{it:bbd props item 1}$.
	
	To prove \ref*{it:bbd props item 2}, first suppose $\bbd_z(\alpha),\bbd_z(\beta) = \infty$, i.e. neither has a $z$-term.
	It follows that their product, $\alpha\beta$ has no $z$-terms and thus $\bbd_z(\alpha\beta)=\infty$.
	
	Now suppose $\bbd_z(\alpha)$ is finite.
	If $\bbd_z(\alpha\beta)$ is infinite then $\bbd_z(\alpha\beta)>\bbd_z(\alpha)\geq \min\set{\bbd_z(\alpha),\bbd_z(\beta)}$ and we are done.
	Finally, suppose $\bbd_z(\alpha\beta)<\infty$ and
	let $\alpha\beta = \sum_{w\in \fralg{\bbx\cup\bbz}} c_w w$.
	Let $w$ be a monomial with $\abs{w}_z>0$, $\abs{w}=\bbd_z(\alpha\beta)$ and $c_w\neq 0$.
	There exist monomials $u,v$ such that $a_u b_v \neq 0$ and $uv = w$.
	Recall $\alpha$ and $\beta$ have no constant terms so $\abs{u},\abs{v}>0$.
	Since $\abs{w}_z>0$ we may assume $\abs{u}_z>0$, hence $\abs{u}\geq \bbd_z(\alpha)$ and
	\begin{align*}
		\bbd_z(\alpha\beta) &= \abs{w} = \abs{u}+\abs{v}\geq \bbd_z(\alpha)+\abs{v} 
			> \bbd_z(\alpha)\geq \min\set{\bbd_z(\alpha),\bbd_z(\beta)}.
	\end{align*}
	Thus, $\bbd_z(\alpha\beta)>\min\set{\bbd_z(\alpha),\bbd_z(\beta)}$ and item (ii) is done.
		
	To prove item \ref*{it:bbd props item 3}, suppose $\bbd_z(\alpha)>1$ and $\bbd_z(\gamma)<\infty$.
	Set 
	\[
		W = \set{w\in\fralg{\bbx\cup\bbz}: \abs{w}_z>0, a_w\neq 0}
	\]
	and note if $W$ is empty then 
	$\bbd_z(\alpha(\x)[\gamma(\x)[\z]]) = \infty>\bbd_z(\gamma)$.
	Suppose $W$ is nonempty.
	Writing
	\[
		\alpha(\x)[\gamma(\x)] = \sum_{w\in \fralg{\bbx\cup\bbz}} a_w w(\x)[\gamma(\x)[\z]]
	\]
	and applying item \ref*{it:bbd props item 1} yields
	\[
		\bbd_z(\alpha(\x)[\gamma(\x)[\z]]) \geq \min\set{\bbd_z(a_w w(\x)[\gamma(\x)[\z]]):w\in W}.
	\]
	Suppose there is a $w\in W$ such that $\bbd_z(a_w w(\x)[\gamma(\x)[\z]])<\infty$ 
	(if not, then $\bbd_z(\alpha(\x)[\gamma(\x)[\z]]) = \infty$ and we are done).
	We know $\abs{w}\geq 2$ since $w\in W$ and $\bbd_z(\alpha)\geq 2$.
	Hence $w(\x)[\gamma(\x)[\z]]$ is a finite product of $x_i$ and $\gamma_i(\x)[\z]$ terms.
	Thus, for $w\in W$,
		$\bbd_z(a_w w(\x)[\gamma(\x)[\z]]) > \bbd_z(\gamma)$,
	and therefore $\bbd_z(\alpha(\x)[\gamma(\x)[\z]]) > \bbd_z(\gamma)$.
\end{proof}

\begin{lemma}
	\label{lem:bba n+m}
	Suppose $\bba\in(\C\fps{\bbx\cup\bbz}_+)^h$ and define
	\[
		\bba^{\circ k}(\x)[\z] = \bba(\x)[\bba^{\circ k-1}(\x)[\z]]
	\]
	for $k>1$ and $\bba^{\circ 1}(\x)[\z] = \bba(\x)[\z]$.
	If $n,m\in\Z^+$ then
	\[
		\bba^{\circ n}(\x)[\bba^{\circ m}(\x)[\z]]=\bba^{\circ m}(\x)[\bba^{\circ n}(\x)[\z]]
		=\bba^{\circ n+m}(\x)[\z].
	\]
\end{lemma}

\begin{proof}
	We first prove
	\begin{equation}
		\label{eq:bba n+1}
		\bba^{\circ n+1}(\x)[\z] = \bba(\x)[\bba^{\circ n}(\x)[\z]] = \bba^{\circ n}(\x)[\bba(\x)[\z]]
	\end{equation}
	via induction.
	The base case is from the definition, so suppose \eqref{eq:bba n+1} holds for $n$ and consider
	\begin{align*}
		\bba^{\circ n+2}(\x)[\z] &= \bba(\x)[\bba^{\circ n+1}(\x)[\z]] = \bba(\x)[\bba^{\circ n}(\x)[\bba(\x)[\z]]]\\
		&= \bba^{\circ n+1}(\x)[\bba(\x)[\z]].
	\end{align*}
	Thus \eqref{eq:bba n+1} holds in general.
	
	Now, take any $n,m\in\Z^+$ and consider $\bba^{\circ n}(\x)[\bba^{\circ m}(\x)[\z]]$.
	Applying \eqref{eq:bba n+1} $m$ times yields
	$
	\bba^{\circ n}(\x)[\bba^{\circ m}(\x)[\z]]=\bba^{\circ n+m}(\x)[\z],
	$
	while applying \eqref{eq:bba n+1} $n$ times gives
	$
	\bba^{\circ n+m}(\x)[\z] = \bba^{\circ m}(\x)[\bba^{\circ n}(\x)[\z]].
	$
\end{proof}

Suppose $\bba\in (\C\fps{\bbx\cup\bbz}_+)^h$.
For each $k\geq1$ set $\bbd^k_a = \bbd_z(\bba^{\circ k})$ and $\bba^{\circ k}(\x)[\z] = \sum_{w\in\langle{\bbx\cup\bbz}\rangle} c_w^k w$.
We define
\begin{equation}
	\label{eq:bba k form}
	\mfa^k(\x) = \sum_{\substack{w\in\langle\bbx\rangle\\ \abs{w}<\bbd^k_a}} c_w^k w \quad \text{ and } \quad
	\alpha^k(\x)[\z] = \sum_{\substack{w\in\langle\bbx\cup\bbz\rangle\\ \abs{w}\geq\bbd^k_a}} c_w^k w,
\end{equation}
and observe 
\begin{equation}
	\label{eq:bba k affine form}
	\bba^{\circ k}(\x)[\z] = \mfa^k(\x) + \alpha^k(\x)[\z].
\end{equation}

\begin{lemma}
	\label{lem:ak and bbd}
	Suppose $\bba\in(\C\fps{\bbx\cup\bbz}_+)^h$ and $\bbd_z(\bba)>1$.
	The sequences $(\mfa^k)$ and $\bbd^k_a$ have the following properties.
	\begin{enumerate}[label=(\roman*)]
		\item \label{it::ak and bbd item 1} $\bbd^k_a$ is either strictly increasing with $k$, or there is an $N$ such that if $k<N$ then
		$\bbd^{k}_a< \bbd^{k+1}_a$ and if $k\geq N$ then $\bbd^k_a = \infty$;
		\item \label{it::ak and bbd item 2} $\bbd^k_a>k$ for all $k$;
		\item \label{it::ak and bbd item 3} If $n\geq k$ then the coefficients of $\bba^{\circ k}$ and $\bba^{\circ n}$
		agree on monomials of length less than $\bbd^k_a$ and, in particular, the coefficients of $\mfa^k$ and $\mfa^n$ 
		agree on monomials of length less than $\bbd^k_a$;
		\item \label{it::ak and bbd item 4} $(\mfa^k)$ is a convergent sequence (in the topology of $(\C\fpsx)^g$) and letting $a = \lim_{k\to\infty}\mfa^k$
		we have $a(\x) = \bba(\x)[a(\x)]$.
		Moreover, $a$ is the unique function such that $a(\x) = \bba(\x)[a(\x)]$.
	\end{enumerate}
\end{lemma}

\begin{proof}
	If $\bbd^{k+1}_a<\infty$ then there is a monomial $w$ appearing in $\bba^{\circ k+1}$ with $\abs{w}_z>0$.
	However, since $\bba^{\circ k+1}(\x)[\z] = \bba(\x)[\bba^{\circ k}(\x)[\z]]$, Lemma~\ref{lem:bbd props}(iii) tells us exactly $\bbd_a^{k+1}>\bbd_z^k$.
	
	Suppose $\bbd^n_a = \infty$ for some $n$ and set $N = \min\set{k:\bbd^k_a = \infty}\geq 1$.
	We note $\bbd^{N-1}_a<\infty = \bbd^N_a$ and if $k<N-1$ then $\bbd^k_a<\bbd_a^{k+1}<\infty$.
	For any $k\geq N$,
	\[
		\bba^{\circ k}(\x)[\z] = \bba^{\circ N}(\x)[\bba^{\circ k-N}(\x)[\z]] 
	\]
	by Lemma~\ref{lem:bba n+m}.
	Since there are no $z$-terms appearing in $\bba^{\circ N}$,
	\[
		\bba^{\circ k}(\x)[\z]=\bba^{\circ N}(\x)[\bba^{\circ k-N}(\x)[\z]] = \bba^{\circ N}(\x)[\z]
	\]
	Hence $\bbd^{k}_a = \bbd^{N}_a = \infty$ and item (i) is proved.
	
	By item (i) we know $\bbd^{k+1}_a\geq \bbd^k_a+1$ for all $k$.
	Since $\bbd_a^1>1$ by hypothesis, we see that $\bbd_a^k\geq \bbd_a^1+k-1>k$, thus item (ii) is proved.

	First, recall from \eqref{eq:bba k affine form} that $\bba^{\circ k}(\x)[\z] = \mfa^k(\x) + \alpha^k(\x)[\z]$.
	If $\alpha^k = 0$ then $\bbd^k_a = \infty$ and $\bba^{\circ k}(\x)[\bba(\x)[\z]] = \bba^{\circ k}(\x)[\z]$, hence
	the coefficients of $\mfa^k$ and $\mfa^{k+1}$ agree up to $\bbd^k_a$, i.e. $\mfa^k=\mfa^{k+1}$.

	If $\alpha^k\neq 0$ then the minimum length of a monomial appearing in $\alpha^k(\x)[\bba(\x)[\z]]$ is at least $\bbd^k_a$ 
	since the minimum length of a monomial appearing in $\alpha^k(\x)[\z]$ is $\bbd^k_a$.
	Hence, the coefficients of $\bba^{\circ k+1}(\x)[\z]$ and $\bba^{\circ k}(\x)[\z]$ agree on monomials of 
	length less than $\bbd^k_a$, and in particular, the coefficients of $\mfa^k$ and $\mfa^{k+1}$ 
	agree on monomials of length less than $\bbd^k_a$.
	Hence, with iteration, if $n\geq k$ then the coefficients of $\bba^{\circ k}$, $\bba^{\circ n}$, 
	$\mfa^k$ and $\mfa^{n}$ agree on monomials of length less than $\bbd^k_a$.
	Thus, item (iii) is proved.

	Finally, to prove item (iv), we observe $d(\mfa^n,\mfa^m)\leq2^{-\min\set{n,m}}$ (recall $d$ is the metric on formal power series), 
	hence $(\mfa^k)$ is a Cauchy sequence and thus converges.
	Set $a = \lim_{k\to\infty} \mfa^k$.
	
	Let $n\in\Z^+$ be given and note $\bbd^n_a>n$, by item (ii).
	By item (iii), the coefficients of $a$, $\bba^{\circ n}$ and $\mfa^n$ agree on monomials of length less than $\bbd^n_a$.
	Hence, the coefficients of $\mfa^n(\x)$, $\bba(\x)[\mfa^n(\x)]$ and $\bba^{\circ n}(\x)[\z]$
	all agree on monomials of length less than $n$.
	Consequently, the coefficients of $a(\x)$ and $\bba(\x)[a(\x)]$ must agree on all monomials of length less than $n$.
	Thus $a(\x) = \bba(\x)[a(\x)]$.
	
	If $\hat{a}$ is any formal power series mapping such that $\bba(\x)[\hat{a}(\x)] = \hat{a}(\x)$ then 
	$\bba^{\circ n}(\x)[\hat{a}(\x)] = \hat{a}(\x)$ for all $n\geq 1$.
	However, this implies that the coefficients of $a$ and $\hat{a}$ agree on monomials of length less than $n$, for all $n$.
	Thus, $\hat{a} = a$.
\end{proof}

In order to connect items \ref*{it::ak and bbd item 3} and \ref*{it::ak and bbd item 4}
	to other ideas from analysis we define a partial ordering on $\C\fpsx$.
If $\alpha = \sum_{w} a_w w$ and $\beta = \sum_w b_w w$, then we say $\alpha \leq \beta$ if $a_w = 0$ whenever $b_w = 0$
and $a_w = b_w$ whenever $a_w\neq 0$.
Thus, Lemma \ref{lem:ak and bbd} says $(\mfa^n)$ is an increasing sequence of polynomials with $a$ as its unique limit.

Under the correct reformulation, Lemma \ref{lem:ak and bbd} is actually an implicit function theorem.
In Section \ref{sec:Reduce} we fully define the free derivative of a formal power series,
allowing us to easily state and prove Theorem \ref{thm:fps implicit func thm}, the implicit function theorem for free formal power series.

Although the definitions and results in Lemma \ref{lem:ak and bbd} are valid when $g\neq h$, 
	when applying these ideas to Jacobian matrices we often assume $g=h$.

\begin{definition}
	\label{def:aux inv}
	Suppose $p\in (\C\fpsx_+)^g$ has a Jacobian matrix $J_p\in M_g(\C\fpsx)$ such that $J_p^{-1}\in M_g(\C\fpsx)$.
	Define the \dfn{auxiliary inverse} of $p$ to be $\bbq(\x)[\z] = \x J_p^{-1}(\z)\in(\C\fps{\bbx\cup\bbz}_+)^g$ and
	recursively define the $\bm{k^{th}}$ {\bf auxiliary inverse} by
	\[
		\bbq^{\circ k}(\x)[\z] 
		= \bbq\left(\x\right)\left[\bbq^{\circ k-1}(\x)[\z]\right]
	\]
	where $\bbq^{\circ 1}(\x)[\z] = \bbq(\x)[\z]$.
\end{definition}

%\begin{lemma}
%	\label{lem:bbq n+m}
%	If $n,m\in\Z^+$ then
%	\[
%		\bbq^{\circ n}(\x)[\bbq^{\circ m}(\x)[\z]]=\bbq^{\circ m}(\x)[\bbq^{\circ n}(\x)[\z]]
%			=\bbq^{\circ n+m}(\x)[\z].
%	\]
%\end{lemma}
%
%\begin{proof}
%	We first prove
%	\begin{equation}
%		\label{eq:bbq n+1}
%		\bbq^{\circ n+1}(\x)[\z] = \bbq(\x)[\bbq^{\circ n}(\x)[\z]] = \bbq^{\circ n}(\x)[\bbq(\x)[\z]]
%	\end{equation}
%	via induction.
%	The base case is from the definition, so suppose \eqref{eq:bbq n+1} holds for $n$ and consider
%	\begin{align*}
%		\bbq^{\circ n+2}(\x)[\z] &= \bbq(\x)[\bbq^{\circ n+1}(\x)[\z]] = \bbq(\x)[\bbq^{\circ n}(\x)[\bbq(\x)[\z]]]\\
%		&= \bbq^{\circ n+1}(\x)[\bbq(\x)[\z]].
%	\end{align*}
%	Thus \eqref{eq:bbq n+1} holds in general.
%	
%	Now, take any $n,m\in\Z^+$ and consider $\bbq^{\circ n}(\x)[\bbq^{\circ m}(\x)[\z]]$.
%	Apply \eqref{eq:bbq n+1} $m$ times yields
%	\[
%		\bbq^{\circ n}(\x)[\bbq^{\circ m}(\x)[\z]]=\bbq^{\circ n+m}(\x)[\z],
%	\]
%	while applying \eqref{eq:bbq n+1} $n$ times gives
%	\[
%		\bbq^{\circ n+m}(\x)[\z] = \bbq^{\circ m}(\x)[\bbq^{\circ n}(\x)[\z]].
%	\]
%\end{proof}

The indeterminates $z_1,\dots,z_g$, in $\bbq(\x)[\z]$ are `targets' for composition of $\bbq$ with itself.
As such it is good to understand how the $z$ terms behave under the successive compositions.
We imitate the setup of Lemma~\ref{lem:bbd props}.
%in particular we would like to know the minimum length of any word containing a $z$-term.
For any $k\geq 1$ we write
\[
	\bbq^{\circ k}(\x)[\z] = \sum_{w\in \fralg{\bbx\cup\bbz}} \rho^k_w w
\]
where $\rho^k_w \in \C^g$, and for shorthand purposes we set $\bbd_q^k = \bbd_z(\bbq^{\circ k})$.
We split $\bbq^{\circ k}$ into terms with degree less than $\bbd^k_q$ 
and those with degree greater than or equal $\bbd^k_q$;
\[
	\mfq^k(\x) = \sum_{\substack{w\in\fax \\ \abs{w}<\bbd^k_q}} \rho^k_w w
	\quad \text{ and } \quad
	r^k(\x)[\z] = \sum_{\substack{w\in\fralg{\bbx\cup\bbz} \\ \abs{w}\geq\bbd^k_q}} \rho^k_w w.
\]
Thus
\begin{equation}
	\label{eq:bbq k form}
	\bbq^{\circ k}(\x)[\z] = \mathfrak{q}^k(\x) + r^k(\x)[\z].
\end{equation}
Since the minimum length of any monomial appearing in $r^k(\x)[\bbq(\x)[\z]]$ is greater than $\bbd^k_q$,
we have $\rho^k_w = \rho^{k+1}_w$ for all $\abs{w}<\bbd^k_q$.

\begin{remark}
	\label{rem:bbq ok}
	Since the auxiliary inverse, $\bbq(\x)[\z] = \x J_p^{-1}(\z)$, we note that $\bbd^1_q = \bbd_z(\bbq)>1$.
	Hence, Lemmas \ref{lem:bba n+m} and \ref{lem:ak and bbd} apply to $\bbq$.
\end{remark}

	% \begin{restatable*}{proposition}{PROPcompinv}
%	\begin{proposition}	\label{prop:comp inv}
%		Suppose $p\in (\C\fpsx_+)^g$ and $J_p$ is the Jacobian matrix of $p$.
%		If $J_p$ is invertible in $M_g(\C\fpsx)$ then there exists $q\in (\C\fpsx_+)^g$
%		such that $p$ and $q$ are compositional inverses.
%		In particular, if $\bbq(\x)[\z] = \x J_p^{-1}(\z)$ is the auxiliary inverse of $p$, then
%		$q(\x) = \bbq(\x)[q(\x)]$.
%\end{proposition}
	% \end{restatable*}
	
\PROPcompinv

\begin{proof}
	By Corollary~\ref{cor:inv mat reps} we know that $p$ and $q$ are compositional inverses if and only if
	$J_q(\x) = J_p^{-1}(q(\x))$ and $J_p(\x) = J_q^{-1}(p(\x))$.

	Lemma~\ref{lem:ak and bbd} implies there exists a unique $q\in (\C\fpsx_+)^g$ such that $\bbq(\x)[q(\x)] = q(\x)$,
	where $\bbq(\x)[\z] = \x J_p^{-1}(\z)$ is the auxiliary inverse of $p$.
	Since $q\in(\C\fpsx_+)^g$, we see that $J_p^{-1}(q(\x))\in M_g(\C\fpsx)$ is defined
	and $J_p(q(\x))$ and $J_p^{-1}(q(\x))$ are inverses.
	Hence $q(\x) = \bbq(\x)[q(\x)] = \x J_p^{-1}(q(\x))$ and
	\[
		p(q(\x))=q(\x)J_p(q(\x))=\x J_p^{-1}(q(\x))J_p(q(\x))=\x I_g = \x.
	\]
	
	Next, $q\in\C\fpsx_+^g$ also has an auxiliary inverse, $\bbp(\x)[\z] = \x J_p(q(\z))$.
	Applying Lemma~\ref{lem:ak and bbd} and the same argument as above we know there is a $\tilde{p}\in\C\fpsx_+^g$ 
	such that $\x = q(\tilde{p}(\x))$ and
	\[
		\tilde{p}(\x) = \bbp(\x)[\tilde{p}(\x)] = \x J_p(q(\tilde{p}(\x))).
	\]
	However, since $q(\tilde{p}(\x)) = \x$,
	\[
		\tilde{p}(\x) = \x J_p(q(\tilde{p}(\x))) = \x J_p(\x) = p(\x).
	\]
	Thus $q(p(\x)) = \x$.
	Therefore, $q(\x) = \bbq(\x)[q(\x)]$ and $p$ and $q$ are compositional inverses.
\end{proof}

We note that Proposition~\ref{prop:comp inv} does not require that $p$ corresponds to a bijective free analytic map.
However, $J_p(0)$ and $J_p^{-1}(0)$ both exist, thus with an application of the free inverse function theorem (Theorem~5 in \cite{Pas14})
we get that $p$ is locally invertible on some open free set containing the origin.

We now have conditions guaranteeing a formal power series has a compositional inverse and in fact, we have a way
to calculate the inverse, or at least to approximate it.

\begin{definition}
	\label{def:prop alg poly}
	We once again suppose $\bbz = \set{z_1,\dots,z_h}$, where $h$ is not necessarily equal to $g$.
	Suppose $\alpha\in (\C\fralg{\bbx\cup\bbz})^h$.
	We say $\alpha$ is a \dfn{proper algebraic polynomial}\footnote %FOOTNOTE
	{
		This definition differs from the established terminology often seen in enumerative combinatorics 
		and automata theory.
		In those contexts we say a system of equations $\alpha(\x)[\z] = \z$ is a 
		{\bf proper algebraic system} if $\alpha$ is a proper algebraic polynomial.
	}
	if $\alpha$ has no constant terms and $\bbd_z(\alpha)>1$.
	
	We say $\beta\in (\C\fpsx_+)^h$ is a \index{proper algebraic polynomial!solution}{\bf solution} to the proper algebraic polynomial 
	if $\alpha(\x)[\beta(\x)] = \beta(\x)$.
	Each $\beta^i$ is called a \index{proper algebraic polynomial!component} {\bf component} of the solution.
	
	By either Lemma \ref{lem:ak and bbd} or Theorem~6.6.3 in \cite{Sta11}, every proper algebraic polynomial has a unique solution.
	Let $\gamma\in\C\fpsx$ with constant term $c$.
	We say $\gamma$ is \dfn{algebraic} if $\gamma - c$ is a component of the solution to some proper algebraic polynomial.
\end{definition}

\begin{proposition}
	\label{prop:q alg}
	Suppose $p\in (\C\fax_+)^g$ and $J_p\in M_g(\C\fax)$ is the Jacobian matrix of $p$.
	If $J_p^{-1}\in M_g(\C\fax)$, then the compositional inverse of $p$ is algebraic.
\end{proposition}

\begin{proof}
	Recall $\bbq$, the auxiliary inverse of $p$, is given by $\bbq(\x)[\z] = \x J_p^{-1}(\z)$.
	Proposition~\ref{prop:comp inv} tells us there is a $q\in(\C\fpsx)^g$ such that $q$ and $p$ are
	compositional inverses and $\bbq(\x)[q(\x)] = q(\x)$.
	Observe $\bbq(\x)[\z]$ has no constant terms and $\bbd_z(\bbq)>1$.
	Thus $\bbq(\x)[\z]$ is a proper algebraic polynomial and $q$ is the unique 
	algebraic function satisfying $\bbq(\x)[q(\x)] = q(\x)$.
\end{proof}

We know every polynomial mapping is a rational mapping and Example 6.6.5 in \cite{Sta11} shows every 
rational mapping is an algebraic mapping.
Unfortunately, this does not help us prove a bijective free polynomial has a free polynomial inverse
(at least not directly).

If $p$ is not bijective then it may still have a compositional inverse that is algebraic.
The auxiliary inverse can be a polynomial even if $p$ is not injective, as Example \ref{ex:classic alg} shows.
In the case where $p$ is not injective but $\bbq$ is still a polynomial, we get a unique algebraic function $q$
so that $p(q(X)) = X$ and $q(p(X)) = X$ whenever these compositions are defined.

%
%    ######  ##     ## ########   ######  ########  ######  ######## ####  #######  ##    ## 
%   ##    ## ##     ## ##     ## ##    ## ##       ##    ##    ##     ##  ##     ## ###   ## 
%   ##       ##     ## ##     ## ##       ##       ##          ##     ##  ##     ## ####  ## 
%    ######  ##     ## ########   ######  ######   ##          ##     ##  ##     ## ## ## ## 
%         ## ##     ## ##     ##       ## ##       ##          ##     ##  ##     ## ##  #### 
%   ##    ## ##     ## ##     ## ##    ## ##       ##    ##    ##     ##  ##     ## ##   ### 
%    ######   #######  ########   ######  ########  ######     ##    ####  #######  ##    ##
%
\subsection{Invertibility of the Jacobian matrix}
	\label{ssec:FAF and local inv}
	In this section we establish the following result about bijective free polynomials:
	\THMpolymatinv*
	
	By using the Free Grothendieck theorem, we have that every injective free polynomial has a free polynomial inverse.
	Hence, Theorem~\ref{thm:PMI is poly} is an unsurprising consequence of the chain rule.
	However, Theorem~\ref{thm:PMI is poly} is critical for the proof of the Free Grothendieck theorem and we cannot forgo
	its exposition.

\begin{lemma}
	\label{lem:q agrees with poly}
	If $p:M(\C)^g\to M(\C)^g$ is a bijective free polynomial and $q$ is the inverse of $p$ then for each $n$, 
	there exists a free polynomial $r_n$ such that $q(X) = r_n(X)$ for all $X\in M_n(\C)^g$.
\end{lemma}

\begin{proof}
	This is part of Theorem~\ref{thm:Free Jacobian Conjecture}(iv) 
	and a proof can be found in \cite{Pas14}, however, for the reader's convenience we present 
	a more detailed argument showing $q$ agrees with a free polynomial on each $M_n(\C)^g$.
	
	Let $\pi:M_n(\C)^g\to \C^{g n^2}$ be the canonical isomorphism.
	Since $p$ is bijective, $p[n]:M_n(\C)^g\to M_n(\C)^g$ is bijective and we may view $p[n]$ as a polynomial
		in $g n^2$ variables.
	That is, $\pi\circ p\circ \pi^{-1}$ is a bijective (commutative) polynomial, hence by the classical Grothendieck theorem, 
	$\pi\circ p\circ \pi^{-1}$ has a (commutative) polynomial inverse $\hat{q}:\C^{g n^2}\to \C^{g n^2}$.
	
	Since $p$ is a bijective free polynomial, $q$ is free analytic by Theorem 3.1 in \cite{HKM11}, hence 
			$q[n]:M_n(\C)^g\to M_n(\C)^g$ is analytic and there is a power series, $R = \sum_{m=0}^\infty \sum_{\abs{w} = m} r_w w$ 
			such that $R$ converges on $M_n(\C)^g$ and $R(X) = q[n](X)$ for all $X\in M_n(\C)^g$.
	In particular, $(\pi^{-1}\circ \hat{q}\circ \pi)(X) = R(X) = q[n](X)$ for all $X\in M_n(\C)^g$, hence $\deg(\hat{q}) = \deg(q[n])$.
	Set $\hat{R} = \sum_{m=0}^{\deg(\hat{q})} \sum_{\abs{w}=m} r_w w$ and note $\hat{R}(X) = q[n](X)$ for all $X\in M_n(\C)^g$.
	Since $\hat{R}$ is a free polynomial, we conclude $q$ agrees with a free polynomial on $M_n(\C)^g$.	
%	Since $p$ is bijective, $p[n]:M_n(\C)^g\to M_n(\C)^g$ is bijective and we may view $p[n]$ as a polynomial
%		in $g\hspace{0.1em}n^2$ variables, i.e. there exists an isomorphism $\pi:M_n(\C)^g\to \C^{g\hspace{0.05em}n^2}$ and polynomial
%		$\hat{p}:\C^{g\hspace{0.05em}n^2}\to\C^{g\hspace{0.05em}n^2}$ with $\hat{p}(X) = (\pi\circ p[n]\circ\pi^{-1})(X)$, 
%		for all $X\in M_n(\C)^g$.
%	Moreover, $\hat{p}$ is bijective and its inverse, $\hat{q} = \pi\circ q[n]\circ \pi^{-1}$, is a polynomial by the classical
%		Grothendieck theorem.
%	Note $\deg(\hat{q}) = \deg(q)$
%	
%	Since $p$ is a bijective free polynomial, $q$ is free analytic by Theorem 3.1 in \cite{HKM11}, hence 
%		$q[n]:M_n(\C)^g\to M_n(\C)^g$ is analytic and there is a power series, $R = \sum_{m=0}^\infty \sum_{\abs{w} = m} r_w w$ 
%		such that $R$ converges on $M_n(\C)^g$ and $R(X) = q[n](X)$ for all $X\in M_n(\C)^g$.
%	Hence, $R(X) = q[n](X) = (\pi^{-1}\circ\hat{q}\circ\pi)(X)$ and $\hat{q}$ is a polynomial.
%	For any $X\in M_n(\C)^g$, $R(tX) = \sum_{m=0}^\infty t^m\sum_{\abs{w} = m} r_w w(X)$.
%	Since $(\pi^{-1}\circ\hat{q}\circ\pi)(tX)$ is a polynomial in $t$, with a $t$-degree of at most $\deg(\hat{q})$,
%	we get that $\sum_{\abs{w} = m} r_w w(X) = 0$ for $m>\deg(\hat{q})$.
%	Thus, if we set $\hat{R} = \sum_{m=0}^{\deg(\hat{q})} \sum_{\abs{w} = m} r_w w$ then $\hat{R}(X) = R(X) = q[n](X)$ for all 
%		$X\in M_n(\C)^g$.
%	Lastly, $\hat{R}$ is a free polynomial, therefore $q$ agrees with a free polynomial on $M_n(\C)^g$.
\end{proof}

\begin{remark}
	\label{rem:poly inv PMI}
	If $p$ is a bijective free polynomial with a free polynomial inverse $q$, then both $J_p$ and $J_q$ are polynomial matrices and
	$J_q(p(\x))$ also is a polynomial matrix.
	Observe $I_g = J_{q\circ p}(\x) = J_p(\x) J_q(p(\x))$, thus $J_p^{-1}\in M_g(\C\fax)$ since $J_p(\x)^{-1} = J_q(p(\x))$.
\end{remark}

Remark~\ref{rem:poly inv PMI} is an expected consequence of the Jacobian matrix satisfying the chain rule and 
Corollary 1.4 in \cite{Reu92} offers a slightly different proof.
Certainly if $p$ is invertible then it is bijective, however Example \ref{ex:classic alg} shows that $J_p,J_p^{-1}\in M_g(\C\fax)$
is not sufficient for $p^{-1}$ to be a polynomial.
In that sense there is no Jacobian conjecture for the noncommutative Jacobian matrix.

On the other hand, Theorem~\ref{thm:PMI is poly} profits from a noncommutative Nullstellensatz in \cite{HM04} to prove the Jacobian matrix 
	of an injective free polynomial is invertible over $M_g(\C\fax)$.
Before proving the theorem, we first state the noncommutative Nullstellensatz (proved by George Bergman), Theorem 6.3 in \cite{HM04}.

\begin{thmx}
	\label{thm:NC Nullstell}
	Let $\mathcal{P}\subset \C\fax$ be finite and let $s\in \C\fax$.
	Let $d$ denote the maximum of the $\deg(s)$ and $\set{\deg(p):p\in\mathcal{P}}$.
	There exists a complex Hilbert space $\mathcal{H}$ of dimension $\sum_{j=0}^d g^j$, such that, if
	\[
		s(X)v = 0
	\]
	whenever $X = (X_1,\dots,X_g)\in \B(\mathcal{H})^g$, $v\in \mathcal{H}$, and
	\[
		p(X)v = 0 \text{  for all  } p\in\mathcal{P},
	\]
	then $s$ is in the left ideal generated by $\mathcal{P}$.
\end{thmx}

\THMpolymatinv

\begin{proof}
	Let $d = \max\set{\deg(p^{j})}$ and set $N=g^{d+1}$.
	Lemma~\ref{lem:q agrees with poly} tells us $q[N]$ agrees with a free polynomial.
	Suppose $s = (s^1,\dots,s^g)$ is a free polynomial such that $q(X) = s(X)$ for all $X\in M_N(\C)^g$.
	In particular, each $s^j$ has no constant term.
	For each $1\leq j\leq g$ write
	\[
		s^j = \sum_{m=1}^{\deg(s)} \sum_{\abs{w}=m} \sigma^j_w w
	\]
	and observe
	\[
		X_j = s^j(p(X)) = \sum_{m=1}^{\deg(s)} \sum_{\abs{w}=m} \sigma^j_w w(p(X)).
	\]
	Take $X = (X_1,\dots,X_g)\in M_N(\C)^g$ and $v\in \C^N$ such that $v^Tp^{j}(X) = 0$ for all $1\leq j\leq g$.
	Hence $v^Tw(p(X)) =0$ for all $\abs{w}>0$ since $v^Tp^{j}(X) = 0$ for each $1\leq j\leq g$.
	Thus,
	\[
		v^TX_j = v^Ts^j(p(X)) = \sum_{m=1}^{\deg(s)} \sum_{\abs{w}=m} \sigma^j_w v^Tw(p(X)) = 0.
	\]
	By Theorem~\ref{thm:NC Nullstell}, $x_j$ is contained in the right ideal generated by $p^1,\dots, p^{g}$, that is,
	there exist polynomials $\mu_{i,j}$ such that $x_j = \sum_{i=1}^g p^i(\x)\mu_{i,j}(\x)$.
	Let $R = (\mu_{i,j})_{i,j=1}^g\in M_g(\C\fax)$ and observe
	\[
		\x = p(\x)R(\x) = \x J_p(\x)R(\x).
	\]
%	and
%	\[
%		p(\x) = \x J_p(\x) = p(\x) R(\x) J_p(\x).
%	\]
	Therefore $J_p^{-1} = R\in M_g(\C\fax)$.
\end{proof}

%Combining Theorem~\ref{thm:PMI is poly} with Proposition~\ref{prop:q alg} we get the following corollary.
%
%\CORbijpolyalginv
%
%However, we cannot conclude that $q$ is a polynomial as there may be some issue that arises when composing
%$\bbq$ with itself repeatedly.

\begin{example}
	\label{ex:nilpotent sad}
	Let
	\[
		N = \bpm -1 & 1 \\ -1 & 1 \epm
	\]
	and set
	\begin{align*}
		p(\x) &= \x (I_2-N x_1)\\
		&= \bpm x_1, & x_2 \epm \bpm 1+x_1 & -x_1\\ x_1 & 1-x_1 \epm \\
		&= \bpm x_1 + x_1^2 + x_2x_1, & x_2 -x_1^2 - x_2x_1 \epm.
	\end{align*}
	Observe, $J_p=I_2 - N x_1$, and that $N^2 = 0$.
	Hence, $J_p^{-1} = I_2 + N x_1$ and
	\begin{align*}
		\bbq(\x)[\z] &= \x (I_2 + N z_1)\\
		&= \bpm x_1 - x_1z_1 - x_2z_1, &  x_2 +x_1z_1 + x_2z_1\epm\\
		&= \bpm x_1 - (x_1+x_2)z_1, &  x_2 +(x_1 + x_2)z_1\epm.
	\end{align*}
	However, $p$ is not even injective on $\C^2$ since $p(-1/2,-1/2) = (0,1) = p(0,1)$.
	
	Note,
	\[
		\bpm 0 & 1\\ 0& 0 \epm = \bpm 0 & 1\\-1 & 1\epm \bpm -1 & 1 \\ -1 & 1\epm \bpm 1 & -1 \\ 1 & 0 \epm,
	\]
	that is, $N$ is similar to a strictly upper triangular nilpotent matrix.
	Thus, conjugation of a Jacobian matrix by a similarity does not preserve the desirable properties of
	the Jacobian matrix.
\end{example}

In some sense, the noncommutative Jacobian matrix attempts to linearize polynomial mappings so that a reasonable structure
is preserved via composition.
In fact, if $p$ is a formal power series mapping, then $J_p$ is invertible if and only if $p$ is locally invertible at $0$,
a statement reminiscent of the inverse function theorem.
Hence, $J_p^{-1}$ is a linear approximation of $p^{-1}$ at $0$, explaining why we can iteratively construct $p^{-1}$ from 
$J_p^{-1}$.
However, Example \ref{ex:classic alg} shows how the Jacobian matrix can fail to witness the non-injectivity of a polynomial.

In subsection \ref{ssec:bij crit} we construct the hypo-Jacobian matrix of a free polynomial, a matrix whose invertibility exactly
captures the injectivity or non-injectivity of the free polynomial.

\begin{example}
	\label{ex:classic alg}
	This example is investigated in \cite{Reu92} and it shows that there is no Jacobian conjecture with the noncommutative Jacobian matrix.
	Let $p(\x) = (x_1,x_2 - x_1 x_2 x_1)$, and observe
	\[
		J_p(\x) = \bpm 1 & -x_2 x_1 \\ 0 & 1 \epm, \quad \quad J_p^{-1}(\x) = \bpm 1 & x_2 x_1 \\ 0 & 1 \epm.
	\]
	Note $p$ is not bijective and $\bbq(\x)[\z] = (x_1,x_2+x_1z_2z_1)$.
	It is straightforward to verify that 
	$\bbq^{\circ k}(\x)[\z] = (x_1,x_2+x_1^{k}z_2z_1x_1^{k-1}+\sum_{j=1}^{k-1} x_1^j x_2 x_1^j)$ and
	$q(\x) = (x_1,\sum_{j=0}^{\infty} x_1^j x_2 x_1^j)$.
	Thus $q$ is certainly not a polynomial.
	
	Recall $\bbq^{\circ k}(\x)[\z] = \mfq^k(\x) + r^n(\x)[\z]$.
	For this example, $\mfq^k(\x) = (x_1,\sum_{j=0}^{k-1} x_1^j x_2 x_1^j)$ while $r^k(\x)[\z] = (0,x_1^k z_2 x_1^k)$
	and $\bbd^k_q = 2k+1$.
	In particular, $\deg(\mfq^k) = 2k-1$ is strictly increasing with $k$, immediately discounting $q$ from being
	a polynomial, however $q$ is an algebraic function.
\end{example}

	In fact, if $p$ is any free polynomial whose auxiliary inverse, $\bbq$, is a polynomial then the only way for
	the $p$ to have a non-polynomial inverse is if the situation above occurs, 
	that is, $\deg(\mfq^k)$ is a strictly increasing sequence.
	Section~\ref{sec:Reduce} deals with exactly this.

\section{Free derivatives and the linearization of the auxiliary inverse}
	\label{sec:Reduce}
In this section we establish conditions that guarantee $q$, the compositional inverse of $p$, is a polynomial.
We use Theorem~\ref{thm:Free Jacobian Conjecture} to linearize $\bbq$, the auxiliary inverse of $p$, in terms of $z_1,\dots,z_g$.
This linearization has the caveat that we introduce $g$-`dummy' variables.

%
%    ######  ##     ## ########   ######  ########  ######  ######## ####  #######  ##    ## 
%   ##    ## ##     ## ##     ## ##    ## ##       ##    ##    ##     ##  ##     ## ###   ## 
%   ##       ##     ## ##     ## ##       ##       ##          ##     ##  ##     ## ####  ## 
%    ######  ##     ## ########   ######  ######   ##          ##     ##  ##     ## ## ## ## 
%         ## ##     ## ##     ##       ## ##       ##          ##     ##  ##     ## ##  #### 
%   ##    ## ##     ## ##     ## ##    ## ##       ##    ##    ##     ##  ##     ## ##   ### 
%    ######   #######  ########   ######  ########  ######     ##    ####  #######  ##    ##
%
\subsection{Polynomial criteria}
	\label{ssec:Polynomial Criteria}
We begin by recalling a few facts about auxiliary inverses.
By Theorem~\ref{thm:PMI is poly} we know if $p$ is a bijective free polynomial with $p(0)=0$ 
then $J_p$ and $J_p^{-1}$ are matrices of free polynomials.
Recall from Definition \ref{def:aux inv} that $\bbq^k$, the $k^{\text{th}}$ auxiliary inverse of $p$, is given by
$\bbq^{\circ k}(\x)[\z] = \mfq^k(\x)+r^k(\x)[\z]$, where $\deg(\mfq^k)< \bbd^k_q$.
Furthermore, by Lemma~\ref{lem:ak and bbd} we know $\bbd^k_q\nearrow\infty$.
Proposition~\ref{prop:q alg} tells us that $q$, the inverse of $p$, is the unique solution of 
$\bbq(\x)[q(\x)]=q(\x)$.
Since $q = \lim_{k\to\infty} \mfq^k$, if $q$ were actually a free polynomial then we would expect a large degree
gap to appear in the monomials of $\bbq^{\circ k}$.
This is precisely what Lemma~\ref{lem:a gaps} deals with.

\begin{lemma}
	\label{lem:a gaps}
	Suppose $b\in (\C\fax_+)^g$ and $a\in(\C\fpsx_+)^g$ are compositional inverses,
	$\bba\in(\C\fralg{\bbx\cup\bbz}_+)^g$ is a proper algebraic polynomial, and
	$\bba(\x)[a(\x)] = a(\x)$.
	Let $\bba^{\circ k}(\x)[\z] =  \mfa^k(\x)+\alpha^k(\x)[\z]$ as in \eqref{eq:bba k form} and 
	$\bbd^k_a = \bbd_z(\bba^{\circ k})$ as in \eqref{eq:bbd}.
	The following are equivalent;
	\begin{enumerate}[label=(\roman*)]
		\item $a$ is a polynomial; \label{it:a gap item 1}
		\item $a = \mfa^m$ for some $m\in\Z^+$; \label{it:a gap item 2}
		\item $\bbd_a^N>\deg(\mfa^N)\deg(b)$ for some $N\in\Z^+$. \label{it:a gap item 3}
	\end{enumerate}
\end{lemma}

\begin{proof}
	Recall from Lemma~\ref{lem:ak and bbd} that $\lim_{k\to\infty} \mfa^k = a$, $\bba(\x)[a(\x)] = a(\x)$
	and $\bbd^k_a$ is either always strictly increasing or is strictly increasing until it becomes constant at infinity.
	We note $\bba^{\circ k+1}(\x)[a(\x)] = \bba^{\circ k}(\x)[\bba(\x)[a(\x)]] = \bba^{\circ k}(\x)[a(\x)]$, implying 
	\begin{equation}
		\label{eq:bbak of a is a}
		\bba^{\circ k}(\x)[a(\x)] = a(\x)
	\end{equation}
	for all $k\geq 1$.
	Next, composing with $b(\x)$ yields $\bba^{\circ k}(b(\x))[a(b(\x))] = a(b(\x)) = \x$.
	
	$\ref*{it:a gap item 1}\Rightarrow \ref*{it:a gap item 2}$. Suppose $a$ is a polynomial.
	By Item (ii) in Lemma~\ref{lem:ak and bbd}, if $k\geq \deg(a)$, then $\mathbbm{d}^k_a>\deg(a)$.
	Hence 
	\[
		a(\x) = \bba^{\circ k}(\x)[a(\x)] = \mfa^k(\x) + \alpha^k(\x)[a(\x)].
	\]
	However, by the definition of $\alpha^k$, the minimum possible length of any word appearing in 
	$\alpha^k$ is $\bbd_a^k>\deg(a)$.
	Thus 
	$\alpha^k(\x)[a(\x)] = 0$ and $\mfa^k = a$.
	
	$\ref*{it:a gap item 2}\Rightarrow \ref*{it:a gap item 3}$. Suppose $a = \mfa^m$ for some $m\in\Z^+$.
	If $n\geq m$ then items (iii) and (iv) in Lemma~\ref{lem:ak and bbd} imply $\mfa^n-\mfa^m$ contains no monomials of length
	less than or equal to $\deg(\mfa^m)$.
	However, $\mfa^m = a$, hence we must have $\mfa^m=\mfa^n=a$, for all $n\geq m$.
	Thus, the sequence $\deg(\mfa^n)\deg(b)$ is constant for $n\geq m$.
	On the other hand, $\bbd^n_a$ is either always strictly increasing or is strictly increasing until it becomes constant at infinity.
	Therefore, there is some $N$ such that $\bbd^N_a > \deg(a)\deg(b) = \deg(\mfa^N)\deg(b)$.
	
%	If $n\geq m$ then we know by item (iii) in Lemma~\ref{lem:ak and bbd} that the coefficients of
%	$\mfa^m$ and $\mfa^n$ agree on monomials of length less than $\bbd^m_a>\deg(\mfa^m)$.
%	However, we also know that the coefficients of $a$ and $\mfa^n$ agree on monomials of length less than $\bbd^n_a$.
%	Thus, the coefficients of $\mfa^m=a$ and $\mfa^n$ agree on monomials of length less than $\bbd^n_a>\bbd^m_a>\deg(a)$,
%	implying $\mfa^n=\mfa^m=a$ for all $n\geq m$.
%	Choosing $N=\bbd^m_a\deg(b)>\deg(a)\deg(b)$ we get that $\mfa^N = a$ and therefore
%	$\bbd^N_a>N>\deg(a)\deg(b) = \deg(\mfa^N)\deg(b)$.
	
	(iii)$\Rightarrow$(i). Suppose there is some $N$ such that $\bbd_a^N>\deg(\mfa^N)\deg(b)$.
	Substituting $b(\x)$ for $\x$ in \eqref{eq:bbak of a is a},
	\[
		\x=\bba^{\circ N}(b(\x))[\x] = \mfa^N(b(\x)) + \alpha^N(b(\x))[\x].
	\]
	However, 
	\[
		\deg\left(\mfa^N\left(b(\x)\right)\right)\leq\deg\left(\mfa^N\right)\deg(b)<\bbd_q^N,
	\]
	and the minimum degree of any monomial appearing in $\alpha^N(b(\x))[\x]$ is greater than $\bbd^N_q$, 
	implying $\alpha^N(b(\x))[\x] = 0$.
	Thus $\mfa^N(b(\x))=\x$, therefore $a = \mfa^N$ is a polynomial.
\end{proof}

\begin{remark}
	\label{rem:q gaps ok}
	Suppose $p\in (\C\fax_+)^g$ with $J_p,J_p^{-1}\in M_g(\C\fax)$.
	Let $\bbq$ be the auxiliary inverse of $p$ and let $q$ be the compositional inverse of $p$.
	Recall that since $\bbq$ is a polynomial and $\bbd_z(\bbq)>1$, $\bbq(\x)[\z]$ is a proper algebraic polynomial
	and $q$ is the unique algebraic function such that $\bbq(\x)[q(\x)] = q(\x)$.
	Thus Lemma~\ref{lem:a gaps} applies to $\bbq$ and $q$.
\end{remark}

It should be noted that if $a = \mfa^N$ for some $N$ then we cannot conclude $\alpha^N =0$.
Example \ref{ex:twisted comp} describes a bijective free polynomial $p$ with a free polynomial inverse $q$, such that
$\bbq^{\circ k}(\x)[\z]\neq q(\x)$, i.e. $r^k(\x)[\z]\neq 0$ for any $k\geq 1$.

\begin{example}
	\label{ex:twisted comp}
	Let $p^1,p^2\in \C\fralg{x_1,x_2}$ with $p^1 = (x_1,x_2+x_1^2)$, and $p^2 = (x_1+x_2^2,x_2).$
	Both are bijective with $q^1 = (x_1,x_2-x_1^2)$ and $q^2 = (x_1-x_2^2,x_2)$	as their respective inverses.
	Their composition $p = p^1\circ p^2 = (x_1+x_2^2,x_2+(x_1+x_2^2)^2)$ has inverse $q = q^2\circ q^1 = (x_1-(x_2-x_1^2)^2,x_2-x_1^2).$

	Since 
	\[
		p^1 = \bpm x_1 & x_2 \epm
			\left(\begin{smallmatrix} 1 & -x_1\\ 0 & 1 \end{smallmatrix}\right)^{-1} \quad \text{ and } \quad 
		p^2 = \bpm x_1 & x_2 \epm
			\left(\begin{smallmatrix} 1 & 0\\ -x_2 & 1 \end{smallmatrix}\right)^{-1},
	\]
	we have,
	\begin{align*}
		p &= \bpm x_1 & x_2 \epm \left(\begin{smallmatrix} 1 & 0\\ -x_2 & 1 \end{smallmatrix}\right)^{-1} 
			\left(\begin{smallmatrix} 1 & -(p^2)_1\\ 0 & 1 \end{smallmatrix}\right)^{-1} 
			= \bpm x_1 & x_2 \epm \left(\begin{smallmatrix} 1 & x_1+x_2^2 \\ x_2 & 1+x_2(x_1+x_2^2) \end{smallmatrix}\right),
	\end{align*}
	and
	\begin{align*}
		\bbq(\x)[\z] &= \bpm x_1 & x_2 \epm\left(\begin{smallmatrix} 1 & -(z_1+z_2^2)\\ 0 & 1 \end{smallmatrix}\right)
		\left(\begin{smallmatrix} 1 & 0\\ -z_2 & 1 \end{smallmatrix}\right)\\
		&=\left(x_1+x_1(z_1+z_2^2)z_2-x_2z_2  ,  x_2-x_1(z_1+z_2^2)\right).
	\end{align*}
	In this case a gap between $\mfq^k$ and $r^k(\x)[\z]$ forms rather quickly and the true
	inverse is extracted quite easily.
	However each iterate of $\bbq^k$ will have a $z$-term.
\end{example}

%
%    ######  ##     ## ########   ######  ########  ######  ######## ####  #######  ##    ## 
%   ##    ## ##     ## ##     ## ##    ## ##       ##    ##    ##     ##  ##     ## ###   ## 
%   ##       ##     ## ##     ## ##       ##       ##          ##     ##  ##     ## ####  ## 
%    ######  ##     ## ########   ######  ######   ##          ##     ##  ##     ## ## ## ## 
%         ## ##     ## ##     ##       ## ##       ##          ##     ##  ##     ## ##  #### 
%   ##    ## ##     ## ##     ## ##    ## ##       ##    ##    ##     ##  ##     ## ##   ### 
%    ######   #######  ########   ######  ########  ######     ##    ####  #######  ##    ##
%
\subsection{Free derivatives and scions}
	\label{ssec:scions}
We now introduce the formal directional derivative as was done in \cite{Pas14} and similarly 
in \cite{HKM12} and \cite{HKM11}.

\begin{definition}
	\label{def:D def}
	Let $\bby = \set{y_1,\dots,y_g}$ be a set of noncommuting indeterminates distinct from $\bbx$ and let
	$\y = (y_1,\dots,y_g)$ be considered as a row vector.
	We define the \dfn{free derivative} $D:\C\fpsx\to\C\fps{\bbx\cup\bby}$ by its action on monomials and then extend it linearly and continuously.
	Define
	\[
		Dx_i(\x)[\y] = y_i
	\]
	and require
	\begin{enumerate}[label=(\roman*)]
		\item $D(k\eta+\mu)(\x)[\y] = k D\eta(\x)[\y]+D\mu(\x)[\y]$;
		\item $D(\eta\mu)(\x)[\y] = D\eta(\x)[\y]\mu(\x)+\eta(\x) D\mu(\x)[\y]$,
%		\item $D(\eta\circ \nu) (\x)[\y] = D\eta(\nu(\x))[(D\nu_1(\x)[\y],\dots,D\nu_g(\x)[\y])]$,
	\end{enumerate}
	for all formal power series $\eta,\mu\in\C\fpsx$. Consequently, for all $\nu \in (\C\fpsx_+)^g$ we have
	\[
		D(\eta\circ \nu) (\x)[\y] = D\eta(\nu(\x))[(D\nu_1(\x)[\y],\dots,D\nu_g(\x)[\y])].
	\]
\end{definition}

Observe $D\eta(\x)[\y]$ is linear in $\y$, that is, $D\eta(\x)[k\y+\z] = k D\eta(\x)[\y] + D\eta(\x)[\z]$.
The linearity of the free derivative allows us to define $D$ on matrices of formal power series.
If $A\in M_{m\times n}(\C\fpsx)$ then define
$D:M_{m\times n}(\C\fpsx)\to M_{m\times n}(\C\fps{\bbx\cup\bby})$ by
\[
	DA(\x)[\y] = (DA_{i,j}(\x)[\y])_{i,j=1}^{m,n}.
\]
In particular $D$ extends to row vectors in the obvious way.

\begin{remark}
	The derivative in free analysis is defined below, and is almost a pure matrix result.
	Suppose $\mathcal{U}$ is a free domain (hence open) and $\eta:\mathcal{U}\to M(\C)^g$ is an analytic free map.
	For any small enough $H\in M(\C)^g$,
	\begin{equation}
		\label{eq:D free}
		\eta\bpm X & H \\ 0 & X \epm = \bpm \eta(X) & D\eta(X)[H]\\ 0 & \eta(X) \epm.
	\end{equation}
	
	It turns out, there is a strong connection between the free derivative in \eqref{eq:D free} and the formal
	power series derivative in Definition \ref{def:D def} (justifying the redundant use of $D$).
	
	If $\lambda:M(\C)^g\to M(\C)^g$ is a free analytic mapping, $\Lambda\in (\C\fpsx)^g$ is a formal power series that
	converges on $M(\C)^g$, and $\Lambda(X) = \lambda(X)$ for all $X\in M(\C)^g$,
	then $D\Lambda\in(\C\fps{\bbx\cup\bby})^{g}$ converges on $M(\C)^{2g}$ and 
	$D\lambda(X)[Y] = D\Lambda(X)[Y]$ for all $X,Y\in M(\C)^g$.
	
	To see this, let $\Lambda = \sum_{m=0}^\infty \sum_{\abs{w}=m} L_w w$ and
	$\Lambda_N = \sum_{m=0}^N \sum_{\abs{w}=m} L_w w$.
	Since $\Lambda_N\in (\C\fax)^g$, $D\Lambda_N\in (\C\fralg{\bbx\cup\bby})^g$, thus $D\Lambda_N(X)[Y]$ 
	exists for all $(X,Y)\in M(\C)^{2g}$.
	Let $Z = \left(\begin{smallmatrix}X & Y\\ 0 & X\end{smallmatrix}\right)$.
	By Proposition~6 in \cite{Pas14},
	\[
		\Lambda_N(Z) =\Lambda_N\bpm X & Y\\ 0 & X \epm = \bpm \Lambda_N(X) & D\Lambda_N(X)[Y]\\ 0 & \Lambda_N(X)\epm.
	\]
	The sequence of polynomials $(D\Lambda_N)$ converges to $D\Lambda$ in the metric topology on $(\C\fps{\bbx\cup\bby})^g$,
	thus $D\Lambda_N(X)[Y]$ converges to $D\Lambda(X)[Y]$, since $\Lambda_N(Z)$ converges to $\Lambda(Z)$.
	Hence,
	\[
		\Lambda(Z) =\Lambda\bpm X & Y\\ 0 & X \epm = \bpm \Lambda(X) & D\Lambda(X)[Y]\\ 0 & \Lambda(X)\epm,
	\]
	and $D\Lambda(X)[Y] = D\lambda(X)[Y]$.
	
	Thus, to find the derivative of a free analytic function with a formal power series, it is sufficient to find the
	derivative of the formal power series and then evaluate where desired.
\end{remark}

\begin{example}
	We present a few formal power series and their corresponding derivatives.
	If $p(x_1,x_2) = x_1x_2 - x_2x_1$ then $Dp(x_1,x_2)[y_1,y_2] = y_1x_2 + x_1y_2 - y_2x_1 - x_2y_1.$
	Next, if $r(x_1) = (1-x_1)^{-1}$ then $Dr(x_1)[y_1] = (1-x_1)^{-1}y_1(1-x_1)^{-1}$.
	Finally, if $s^1(x_1,x_2) = (x_1,x_2+x_1^2)$ and $s^2(x_1,x_2) = (x_2,x_1+x_2)$ then
	\[
		Ds^1(x_1,x_2)[y_1,y_2] = (y_1,y_2+x_1y_1+y_1x_1), \quad Ds^2(x_1,x_2)[y_1,y_2] = (y_2,y_1+y_2)
	\]
	and
	\begin{align*}
		D(s^1\circ s^2)(x_1,x_2)[y_1,y_2] &= Ds^1(s^2(x_1,x_2))[Ds^2(x_1,x_2)[y_1,y_2]]\\
			&= (y_2,(y_1+y_2)+x_2y_2+y_2x_2).
	\end{align*}
\end{example}

Before proceeding with our investigation of the free derivative, we stop to quickly prove the 
	implicit function theorem for nc formal power series.
For an analytic approach to the implicit function theorem for $M(\C)^g$ see \cite{AgMc16}.

\begin{definition}
	Suppose $\bbz = \set{z_1,\dots,z_h}$ and $f(\x,\z) \in (\C\fps{\bbx\cup\bbz})^h$.
	Define
	\[
		\frac{\partial f}{\partial \z} = \bpm Df_i(\x,\z)[0,e_j] \epm_{i,j=1}^h \in M_h(\C\fps{\bbx\cup\bbz}),
	\]
	where $e_j$ is the standard vector with a $1$ in the $i^{\text{th}}$ position and $0$ elsewhere.
\end{definition}

\THMfreeIFT

\begin{proof}
	Since $f(0,0) = 0$, we see that $f(\x,\z)$ has no constant terms.
	By composing with an appropriate change of variables, we may assume ${\partial f}/{\partial \z}(0,0) = I_h$.
	Hence, the coefficient of each $z_i$ term in $f_j$ is $\delta_{i,j}$, the Kronecker delta.
	Set $\hat{f}(\x)[\z] = \z - f(\x,\z)$ and note $\hat{f}$ satisfies the conditions of Lemma \ref{lem:ak and bbd}.
	Thus, there exists a unique $\mathfrak{g}\in (\C\fpsx)^h$ such that 
	$\mathfrak{g}(0) = 0$ and $\hat{f}(\x)[\mathfrak{g}(\x)] = \mathfrak{g}(\x)$.
	Finally, since $f(\x,\z) = \z - \hat{f}(\x)[\z]$,
	\begin{equation}
		\label{eq:Implicit Func eq}
		f(\x,\mathfrak{g}(\x)) = \mathfrak{g}(\x) - \hat{f}(\x)[\mathfrak{g}(\x)] = \mathfrak{g}(\x) - \mathfrak{g}(\x) = 0,
	\end{equation}
	and the uniqueness of $\mathfrak{g}$ for $\hat{f}$ implies $\mathfrak{g}$ is the unique formal power series 
	satisfying both $\mathfrak{g}(0)=0$ and \eqref{eq:Implicit Func eq}.
\end{proof}

\begin{definition}
	Suppose $p\in (\C\fax)^g$.  
	We define the \dfn{scion} of $p$, $F\in(\C\fralg{\bbx\cup\bby})^{2g}$, by	$F(\x,\y) = (Dp(\y)[\x],\y)$.
	Furthermore, if we view $p$ as a free polynomial from $M(\C)^g$ to $M(\C)^g$, then $F:M(\C)^{2g}\to M(\C)^{2g}$
	is a free polynomial and $F(X,Y) = (Dp(Y)[X],Y)$.
\end{definition}

	Of particular importance is the fact that $F$ is $\x$-linear, that is, $F(\x+\z,\y) = F(\x,\y)+F(\z,\y)$.
	Moreover, $DF(\x,\y)[\z,\w]$ is automatically $\z$-linear.

\begin{proposition}
	\label{prop:p bij F bij}
	Suppose $p$ is a free polynomial.
	If $F$ is the scion of $p$, then $p$ is bijective if and only if $F$ is bijective.
\end{proposition}

\begin{proof}
	Suppose $J_p$ is the Jacobian matrix of $p$.
	Since $p(\x) = \x J_p(\x)$, the Jacobian matrix of $F$ is given by
	\[
		J_F(\x,\y) = \bpm J_p(\y) & 0 \\ DJ_p(\y)[\x] & I \epm.
	\]
	We note that $F(\x,\y) = (\x,\y) J_F(\x,\y)$.
	Let $\h = (h_1,\dots,h_g)$ and $\k=(k_1,\dots,k_g)$ be $g$-tuples of noncommuting indeterminates treated as row vectors.
	Consider
	\begin{align*}
		DF&(\x,\y)[\h,0]\\
		&= (\h J_p(\y)+\x DJ_p(\y)[0]+0 DJ_p(\y)[\x]+\y D(DJ_p(\y)[\x])[\h,0],0)\\
		&= (\h J_p(\y) + 0 + 0 + \y DJ_p(\y)[\h],0)\\
		&= (Dp(\y)[\h],0).
	\end{align*}
	The motivation for why $D(DJ_p(\y)[\x])[\h,0] = DJ_p(\y)[\h]$ hinges on the $\x$-linearity of $DJ_p(\y)[\x]$.
	To clarify this point we demonstrate on a monomial $m(\x,\y) = \alpha(\y)x_i\beta(\y)$:
	\begin{align*}
		Dm(\x,\y)[\h,0] &= D\alpha(\y)[0] x_i \beta(\y) + \alpha(\y)h_i \beta(\y) + \alpha(\y)x_i D\beta(\y)[0]\\ 
		&= m(\h,\y).
	\end{align*}
		
	Recall $F_{g+i}(\x,\y) = y_{g+i}$ for $1\leq i\leq g$, hence 
	$DF_{g+i}(\x,\y)[\h,\k] = k_i$.
	In particular, if $X,Y,H,K\in M_n(\C)^g$ then $DF(X,Y)[H,K] = 0$ implies $K=0$.
	Thus $Dp(Y)[H] = 0$ if and only if $DF(X,Y)[H,K] = 0$. 
	Therefore, an application of Theorem~\ref{thm:Free Jacobian Conjecture} implies $p$ is bijective if and only if $F$ is bijective.
\end{proof}

\begin{lemma}
	\label{lem:F poly inv}
	Suppose $p$ is a bijective free polynomial with no constant term, $F$ is the scion of $p$, and
	$q$ and $G$ are the compositional inverses of $p$ and $F$ respectively.
	Let $\bbG(\x,\y)[\z,\w]$ be the auxiliary inverse of $F$.
	Then,
	\begin{enumerate}[label=(\roman*)]
		\item $\bbG(\x,\y)[\z,\y]$ satisfies the conditions of Lemma~\ref{lem:ak and bbd} with $G(\x,\y)$ as its unique solution;
		\item $\bbG(\x,\y)[\z,\y]$ is affine $\z$-linear;
		\item $G(\x,\y) = (Dq(p(\y))[\x],\y)$;
		\item if $q$ is a free polynomial then $G$ is a free polynomial and
			\[
				\deg(q)\leq \deg(G)\leq \deg(p)\deg(q).
			\]
	\end{enumerate}
\end{lemma}

\begin{proof}
%	Suppose $q$ is the inverse function of $p$, $G$ is the inverse function of $F$ and $\bbG$ is the auxiliary inverse of $F$.
	Since $\bbG$ is the compositional inverse of $F$, it automatically satisfies the conditions of Lemma~\ref{lem:ak and bbd}.
	Hence, $\bbG(\x,\y)[\z,\w]$ has $G(\x,\y)$ as its unique solution.
	
	Recall $F_{g+i}(\x,\y) = y_i$ for $1\leq i\leq g$. 
	It follows that $\bbG_{g+i}(\x,\y)[\z,\w] = y_i$ and thus $G_{g+i}(\x,\y) = y_i$.
	In particular, $\bbG(\x,\y)[\z,\y]$ still satisfies the conditions of Lemma~\ref{lem:ak and bbd}
	and $G(\x,\y)$ is the unique solution of $\bbG(\x,\y)[\z,\y]$.
	Thus, Lemma~\ref{lem:a gaps} applies to $\bbG(\x,\y)[\z,\y]$ and $G(\x,\y)$, justifying our use of 
	$\bbG(\x,\y)[\z,\y]$ in lieu of $\bbG(\x,\y)[\z,\w]$.

	To prove item (ii), let $J_p$ and $J_F$ be the Jacobian matrices of $p$ and $F$, respectively.
	Since both $p$ and $F$ have compositional inverses, Proposition~\ref{prop:comp inv} says $J_p^{-1}$ and $J_F^{-1}$ exist as
	matrices of formal power series.
	Hence, 
	\begin{align*}
		J_F^{-1}(\z,\w) &= \bpm J_p(\w) & 0 \\ DJ_p(\w)[\z] & I \epm^{-1}  = \bpm J_p(\w)^{-1} & 0 \\ -DJ_p(\w)[\z]J_p(\w)^{-1} & I \epm.
	\end{align*}
	Observe $\z$ only appears in a free derivative, so $J_F^{-1}(\z,\w)$ is affine $\z$-linear.
	Thus, $\bbG(\x,\y)[\z,\y] = (\x, \y)J_F(\z,\y)$ must also be affine $\z$-linear.

	For (iii), since $(q\circ p)(\x) = \x$, we have $\x = D(q\circ p)(\y)[\x] = Dq(p(\y))[Dp(\y)[\x]]$.
	Thus,
	\begin{equation}
		\label{eq:DqpF}
		(\x,\y) = (Dq(p(\y))[Dp(\y)[\x]],\y) = (Dq(p(\y))[F(\x,\y)],\y).
	\end{equation}
	Since $G_{g+i}(\x,\y) = y_i$, substituting $G(\x,\y)$ for $(\x,\y)$ into \eqref{eq:DqpF} yields
	\[
		G(\x,\y) = (Dq(p(\y))[F(G(\x,\y))],\y) = (Dq(p(\y))[\x],\y).
	\]

	Lastly, suppose $q$ is a free polynomial.
	It follows that $Dq$ is a free polynomial, hence $G$ is a free polynomial.
	Since $\deg(q)\leq \deg(Dq(p(\y))[\x])\leq \deg(q)\deg(p)$ we conclude
	\begin{align*}
		\deg(q)\leq \deg(G) \leq \deg(q)\deg(p).
		\tag*{$\square$}
	\end{align*}
\end{proof}

In a trivial sense, the degree bounds between $G$ and $q$ are not strict.
If $p(\x) =\x$ then $q(\x) = \x$, $F(\x,\y) = (\x,\y)$ and $G(\x,\y) = (\x,\y)$.
Hence, $\deg(q) = \deg(q)\deg(p) = \deg(G)$.

\begin{remark}
	\label{rem:bbG ok}
	We emphasize a point made in the proof of \ref{lem:F poly inv};
	since $G$ is the solution to both $\bbG(\x,\y)[\z,\y]$ and $\bbG(\x,\y)[\z,\w]$ we may use $\bbG(\x,\y)[\z,\y]$
	rather than $\bbG(\x,\y)[\z,\w]$.
\end{remark}

\begin{example}
	Let 
	\[
		p(\x) = (x_1,x_2-x_1^2) \quad  \text{ and } \quad  F(\x,\y) = (x_1,x_2-x_1y_1-y_1x_1,y_1,y_2).
	\]
	Hence 
	\[
		\bbq(\x)[\z] = (x_1,x_2+x_1z_1) \quad \text{ and } \quad D\bbq(\x,\z)[\h,\k] = (h_1,h_2+h_1z_1+x_1k_1)
	\]
	and $\mathbbm{G}(\x,\y)[\z,\y] = (D\bbq(p(\y),\y)[\x,\z],\y) = (x_1,x_2+x_1y_1+x_1z_1,y_1,y_2)$.
	Note $\bbq^{\circ 2}(\x)[\z] = (x_1,x_2+x_1^2) = q(\x)$ and 
	$\mathbbm{G}^{\circ 2}(\x,\y)[\z,\y] = (x_1,x_2+x_1y_1+x_1x_1,y_1,y_2)$.
	Lastly, $2 = \deg(q) = \deg(G)$, while $\deg(q)\deg(p) = 4$, so $\deg(q) = \deg(G) < \deg(q)\deg(p)$. 
\end{example}

Proposition~\ref{prop:p bij F bij} tells us that a polynomial, $p$, is bijective if and only if its scion, $F$, is bijective.
The scion is $\x$-affine linear, and its inverse function, $G$, is the unique algebraic solution to a proper algebraic
polynomial that is $\z$-affine linear.
We investigate precisely the formal power series that are generated by such $\z$-affine linear proper algebraic polynomials
in Section~\ref{sec:Hyporationals}.

\section{Degree bounds on nc rational maps}
	\label{sec:vals and rats}
In order to prove Theorem~\ref{thm:hyporat domain poly} we require results about how rational functions behave when evaluated
on matrices.
Using rational degrees on nc rational functions, we prove Proposition~\ref{prop:r deg bdd},
a result about the behavior of nc rational functions when they are evaluated on generic matrices.

%
%    ######  ##     ## ########   ######  ########  ######  ######## ####  #######  ##    ## 
%   ##    ## ##     ## ##     ## ##    ## ##       ##    ##    ##     ##  ##     ## ###   ## 
%   ##       ##     ## ##     ## ##       ##       ##          ##     ##  ##     ## ####  ## 
%    ######  ##     ## ########   ######  ######   ##          ##     ##  ##     ## ## ## ## 
%         ## ##     ## ##     ##       ## ##       ##          ##     ##  ##     ## ##  #### 
%   ##    ## ##     ## ##     ## ##    ## ##       ##    ##    ##     ##  ##     ## ##   ### 
%    ######   #######  ########   ######  ########  ######     ##    ####  #######  ##    ##
%
\subsection{Rational degree bounds}
\label{ssec:Vals and rat deg}

In this subsection we introduce topics from noncommutative algebra in order to prove a general principle; 
evaluating a noncommutative rational function $r$ on a tuple of matrices produces a matrix whose entries behave similarly to $r$.
A major obstacle in proving this principle is the fact that noncommutative rational functions 
cannot always be written as a fraction of polynomials.
However, by introducing a commuting indeterminate $t$ we are able to characterize the degree of a nc rational function and its evaluations
on matrices.

\begin{definition}
	\label{def:rat deg def}
	Suppose $U$ is a skew field containing $\C$ and $U[t]$ is the polynomial ring over $U$.
	We define the map $\deg_t:U[t]\to \Z\cup\set{-\infty}$ in the natural way;
	$\deg_t(0) = -\infty$ and if $r = r_0 + r_1 t + \dots + r_m t^m$, $r_m\neq0$, then $\deg_t(r) = m$.
	For any $r,s\in U[t]$,
	\begin{enumerate}[label=(\roman*)]
		\item $\deg_t(rs) = \deg_t(r)+\deg_t(s)$\footnote{
		Item (i) is true as long as $U$ is a domain ($U$ has no zero divisors). 
		If $U$ has zero divisors then $\deg_t(rs)\leq \deg_t(r)+\deg_t(s)$.
		},
		\item $\deg_t(x+y) \leq \max\set{\deg_t(x),\deg_t(y)}$.
	\end{enumerate}
	
	Theorem 2.1.15 in \cite{MR01} tells us $U[t] = U\otimes \C[t]$ is an Ore domain with a classical ring of quotients, $U(t)$.
	Hence, for any $r\in U(t)$ there exist $\alpha,\beta\in U[t]$ 	such that $r = \alpha \beta^{-1}$.
	
	We can uniquely extend $\deg_t$ to $U(t)$ (by Theorem 9.1 in \cite{Cohn95}) such that $\deg_t(r^{-1}) = -\deg_t(r)$, 
		for all $r\neq 0$.
	In particular, if $r = \alpha\beta^{-1}$ then $\deg_t(r) = \deg_t(\alpha) - \deg_t(\beta)$.
	We say $\deg_t:U(t)\to \Z\cup\set{\infty}$ is a \dfn{rational degree map}.
\end{definition}

%\begin{definition}
%	If $R$ is any ring, then a \dfn{valuation}\footnote{
%		In general, a valuation is a map $\nu: R\to \Gamma\cup\set{\infty}$ where $\Gamma$ is any ordered group.
%		We say a group $\Gamma$ with a total ordering $x\geq y$ is an ordered group if $x\geq y$, $\hat{x}\geq \hat{y}$
%		implies $x+\hat{x}\geq y+\hat{y}$.
%		Moreover, $\Gamma$ is augmented with the symbol $\infty$ to form a monoid such that $x+\infty = \infty = \infty + x$,
%		for all $x\in \Gamma\cup\set{\infty}$.
%	}
%	is a map 
%	$\nu:R\to \Z\cup\set{\infty}$ such that
%	\begin{enumerate}
%		\item $\nu(x)\in \Z\cup\set{\infty}$ and $\nu$ assumes at least two values,
%		\item $\nu(xy) = \nu(x)+\nu(y)$,
%		\item $\nu(x+y) \geq \min\set{\nu(x),\nu(y)}$.
%	\end{enumerate}
%	We say a valuation is {\bf proper}\index{valuation!proper} if $\nu^{-1}(\set{\infty}) = \set{0}$.
%\end{definition}

%If $\mathbb{S}$ is a skew field then $-\deg_t$ is a valuation on $\mathbb{S}[t]$ and there exists an extension to 
%$\nu_t$, a valuation on $\mathbb{S}(t)$.
%We expand upon this below.

The main result of subsection \ref{ssec:Vals and rat deg} is as follows.

\begin{restatable}{proposition}{PROPratdegbdd}
	\label{prop:r deg bdd}
	Suppose $\bbw = \set{w_1,\dots,w_h}$ is a collection of freely noncommuting indeterminates, $\w = (w_1,\dots, w_h)$,
	$t$ is a central indeterminate and $N\in \Z^+$.
	Let $r\in \C\skf{\bbw}$ be a nonzero rational function such that $r(t\w) = \alpha(\w)[t]\beta(\w)[t]^{-1}$, 
	where $\alpha,\beta\in \C\skf{\bbw}[t]$.
	Suppose
	\begin{enumerate}[label=(\roman*)]
		\item $\mathbb{S}$ is a field containing $\C$,
		\item $U\subset M_N(\mathbb{S})$ is a skew field generated by $u_1,\dots,u_h\in M_N(\mathbb{S})$, each $u_i\neq 0$.
	\end{enumerate}
	If $\deg_t$ is the rational degree map on $\mathbb{S}(t)$,
	then 
	\[
		\deg_t(r(tu_1,\dots,tu_h)_{i,j})\leq \deg_t(\alpha(\w)[t])\in\Z,
	\]
	for all $1\leq i,j\leq N$, whenever $r(u_1,\dots,u_h)$ is defined.
\end{restatable}

Of particular importance in Proposition~\ref{prop:r deg bdd} is that $\deg_t(\alpha(\w)[t])$ is independent of $N$,
hence it applies quite nicely to free functions.

\begin{remark}
	\label{rem:algebra defs}
	The following definitions will be familiar to an algebraist but perhaps not to an analyst.

	If $R$ is any commutative integral domain, then the field of fractions of $R$ is the 
	smallest field in which $R$ can be embedded.
	Every integral domain has a field of fractions.
	
	Next, a ring $D$ is said to be a noncommutative domain if it has no zero divisors, i.e. if $a,b\in D$ such that $ab = 0$ then
	either $a=0$ or $b=0$.
	If, in addition, every nonzero element of $D$ has a multiplicative inverse then $D$ is said to be a skew field.

	Let $R$ be a noncommutative domain and let $S$ be the set of all the nonzero elements of $R$.
	We say $R$ is a right Ore domain if for every $r\in R$ and $s\in S$, $rS\cap sR\neq \varnothing$.
	If $R$ is a right Ore domain, then there is a unique (up to $R$-isomorphism) skew field $D$ containing $R$ as a subring such that every 
	element of $D$ has the form $rs^{-1}$, for $s,r\in R$ and $s\neq 0$.
	In this case, the skew field $D$ is called the classical ring of quotients of $R$, and it is unique up to isomorphism.
\end{remark}

%\begin{definition}
%	\label{def:val def}
%	Suppose $\mathbb{S}$ is any skew-field containing $\C$ and let $t$ be a \dfn{central indeterminate}, that is,
%	$t$ commutes with every member of $\mathbb{S}$.
%	Let $\alpha\in \mathbb{S}[t]\setminus \set{0}$ with $\alpha[t] = a_0+\dots+a_K t^K$, $a_K\neq 0$.
%	Define $\deg_t:\mathbb{S}[t]\to \Z\cup\set{-\infty}$ by $\deg_t(\alpha[t]) = K$ and $\deg_t(0) = -\infty$.
%	It is routine to verify $-\deg_t$ is a proper valuation on $\mathbb{S}[t]$.
%	
%	Moreover, Theorem 2.1.15 in \cite{MR01} tells us $\mathbb{S}[t] = \mathbb{S}\otimes \C[t]$ is an Ore domain 
%	with a classical ring of quotients, $\mathbb{S}(t)$.
%	By Theorem 9.1 in \cite{Cohn95} we can uniquely extend $-\deg_t$ to a valuation 
%	$\nu_t:\mathbb{S}(t)\to\Z\cup\set{\infty}$.
%	We call such a valuation a \bf{rational degree valuation}\index{valuation!rational degree}.
%\end{definition}

\begin{lemma}
	\label{lem:min val bdd}
	Suppose $U$ is any skew field, $v\in U[t]^n$ is a row vector of polynomials and 
	$M\in M_n(U(t))$ is an $n\times n$ matrix.
	If $\deg_t$ is a rational degree map on $U(t)$, then
	\[
		\deg_t((vM)_i)\leq \max_k \set{\deg_t(v_k)}+\max_k\set{\deg_t(M_{k,i})},
	\]
	for each $1\leq i\leq n$.
	
	Moreover, if $\deg_t(v_i)\leq \delta$ and $\deg_t(M_{i,j})\leq \Delta$ for $1\leq i,j\leq n$ then
	\[
		\max_{1\leq k\leq n}\set{\deg_t\big((vM)_k\big)} \leq \delta+\Delta.
	\]
\end{lemma}

\begin{proof}
	Fix $1\leq i\leq g$. 
	Simply applying the properties of the degree map,
	\begin{align*}
		\deg_t((vM)_i) 
		&= \deg_t\left(\sum_k v_k M_{k,i}\right) \leq \max_k \set{\deg_t(v_k M_{k,i})}\\
		&= \max_k \set{\deg_t(v_k)+\deg_t(M_{k,i})}\\
%		&\leq \max_k \set{\deg_t(v_k)}+\max_k\set{\deg_t(M_{k,i})}\\
		&\leq \max_k\set{\deg_t(v_k)}+\max_k\set{\deg_t(M_{k,i})}.
	\end{align*}
	
	Next, suppose $\deg_t(v_i)\leq \delta$ and $\deg_t(M_{i,j})\leq \Delta$ for all $1\leq i,j\leq n$.
	Thus, $\max_k\set{\deg_t(v_k)}\leq \delta$ and $\max_{i,k}\set{\deg_t(M_{k,i})}\leq \Delta$.
	Finally, since $\deg_t((vM)_i)\leq \delta+\Delta$ for $1\leq i\leq n$, we conclude
	$\max_i\set{\deg_t((vM)_i)}\leq \delta+\Delta$.
\end{proof}

Let $\bbw = \set{w_1,\dots,w_h}$ be a finite collection of freely noncommuting indeterminates and let $\w = (w_1,\dots, w_h)$.
We recall a few facts about the construction of $\C\skf{\bbw}$, the algebra of noncommutative rational functions.
These results and definitions can be found in \cite{KVV16} and \cite{KlVo17}.

Let $\mathcal{R}_\C(\bbw)$ be the set of all noncommutative rational expressions over $\C$, i.e. all possible syntactically valid
combinations of elements in $\C$ and $\bbw$,
arithmetic operations (addition, multiplication, inversion) and parentheses.
For example, $w_1+w_1$, $w_1(w_2-w_1)^{-1}$ and $0^{-1}$ are syntactically valid combinations.
The \dfn{inversion height} of $\rho\in\mathcal{R}_\C(\bbw)$ is the maximum number of nested inverses in $\rho$.

The subset of $M(\C)^h$ at which $\rho$ is defined is denoted $\dom \rho$ and is called the \dfn{domain} of $\rho$.
We say $\rho\in\mathcal{R}_\C(\bbw)$ is \index{nondegenerate rational expression}{\bf {nondegenerate}}
if $\rho(A)$ is defined for some $A\in M(\C)^h$.
If $\rho_1,\rho_2$ are nondegenerate rational expressions then we say $\rho_1 \sim \rho_2$ if and only if  
$\rho_1(A) = \rho_2(A)$ for all  $A\in \dom \rho_1 \cap \dom \rho_2$.
This relation $\sim$ is an equivalence relation on the set of all nondegenerate rational expression in $\mathcal{R}_\C(\bbw)$.
We define $\C\skf{\bbw}$, the skew field of \dfn{noncommutative rational functions}, to be the set of equivalence classes of 
nondegenerate expression with respect to $\sim$.
If $r\in\C\skf{\bbw}$, then the domain of $r$, denoted $\dom r$, is defined as the union of the domains of all representatives of $r$
and if $A\in \dom r$ then $r(A) = \rho(A)$ for any representative $\rho\in \mathcal{R}_\C(\bbw)$ such that $A\in \dom \rho$.

\begin{remark}
	\label{rem:reg at 0 are rats}
	Both $\C\fralg{\bbw}$ and $\C_{\textup{rat}}\fps{\bbw}$ embed into $\C\skf{\bbw}$, and in fact, if $r\in\C\skf{u}$ is defined at
	$0$ then $r\in\C\fps{\bbw}$.
	Since every rational series (see Definition \ref{def:rational series}) is defined at $0$, we have that
	the rational series are exactly the nc rational functions defined at $0$.
\end{remark}

We now introduce a lemma that will be implicitly used throughout the rest of Section~\ref{sec:vals and rats}.

\begin{lemma}
	\label{lem:once referenced}
	Suppose $r\in\C\skf{\bbw}$.
	If $t$ is a central indeterminate, then $r(t\w)\in \C\skf{\bbw}(t)$.
\end{lemma}

\begin{proof}
	The proof follows quickly from induction on the inversion height.
\end{proof}

\begin{example}
	If $\rho$ is a rational function in commuting variables then $\rho$ can be written as a fraction of polynomials;
	$\rho = pq^{-1}$.
	Hence, it makes sense to talk about a rational degree map, $\deg(\rho) = \deg(p)-\deg(q)$.
	
	In the noncommutative case we cannot guarantee that a rational function $r$ can be written as a fraction of polynomials.
	However, the commuting indeterminate $t$ we introduce acts as a yardstick for the rational degree of $r$.
	Since $\deg_t(q(t)^{-1}) = -\deg_t(q(t))$ we can unpack a rational function by moving iteratively through the inversion heights.
	
	For example, if
	\[
		r(x_1,x_2) = x_1(1-x_1(1-x_2)^{-1}x_1)^{-1}
	\]
	then we can guess $\deg(x_1(1-x_2)^{-1}x_1) = 2-1$, $\deg((1-x_1(1-x_2)^{-1}x_1)^{-1}) = 1-2$ and,
	$\deg(x_1(1-x_1(1-x_2)^{-1}x_1)^{-1}) = 2 - 2$.
	Introducing the commuting $t$ we get
	\begin{align*}
		r(tx_1,tx_2) &= tx_1(1-tx_1(1-tx_2)^{-1}tx_1)^{-1}\\
		&= tx_1x_1^{-1}(x_1^{-1}-t^2x_1(1-tx_2)^{-1})^{-1}\\
		&= t(1-tx_2)(x_1^{-1}(1-tx_2)-t^2x_1)^{-1}\\
		&= (t-t^2x_2)(x_1^{-1}-tx_1^{-1}x_2-t^2x_1)^{-1},
	\end{align*}
	a fraction of polynomials in $t$ with coefficients in $\C\skf{\bbx}$.
	In fact, as we guessed, $\deg_t(r(t\x)) = 2-2 = 0$.
\end{example}

We are now in a position to prove Proposition~\ref{prop:r deg bdd}.
%Before proceeding to the proof of Proposition~\ref{prop:r deg bdd} we recall a few of the hypotheses:
%\begin{enumerate}
%	\item $\mathbb{S}$ is a field containing $\C$,
%	\item $s_1,\dots,s_h\in U\subset  M_N(\mathbb{S})$, each $s_i\neq 0$, and $U$ is a skew field.
%\end{enumerate}
%Furthermore, let $\nu_t$ be a rational degree valuation on $\mathbb{S}(t)$, that is, $\nu_t$ extends $-\deg_t$ from $\mathbb{S}[t]$
%to a valuation on $\mathbb{S}(t)$.

\PROPratdegbdd*

\begin{proof}
	Suppose $\u = (u_1,\dots,u_h)$ and $\kappa\in U[t]$ is a nonzero polynomial such that $\kappa(\u)[t] = \sum_{j=1}^m k_j(\u)t^j$ 
		with $k_m(\u)\neq 0$.
	Note 
	\[
		\deg_t(\kappa(\u)[t]) = \max_{1\leq i,j\leq N} \set{\deg_t(\kappa(\u)[t]_{i,j})} = m.
	\]
	We first show that $\deg_t(\det(\kappa(\u)[t])) = Nm$.
%	In this case, we have the expected $t$-degree of a polynomial in $U[t]$:
%	$\deg_t(\kappa(\u)[t]) = m$.
%	In fact, $\deg_t(\kappa(\u)[t]) = \max_{1\leq i,j\leq N} \set{\deg_t(\kappa(\u)[t]_{i,j})}$.
	
	From the definition of the determinant, $\deg_t(\det(\kappa(\u)[t]))\leq Nm$.
	In fact, $\det(\kappa(\u)[t])$ is a sum of products of $N$ entries of $\kappa(\u)[t]$.
	Hence, $\det(\kappa(\u)[t])\in \mathbb{S}[t]$ and its $t^{Nm}$ coefficient is exactly $\det(k_m(\u))$.
	Since $\mathbb{S}$ is a field, for any $A\in M_N(\mathbb{S})$, $\det(A)\neq 0$ if and only if $A$ is invertible.
	Next, $U$ being a skew field and $k_m\in U$ being nonzero imply $k_m$ is invertible, hence $\det(k_m) \neq 0$.
	Thus, $\deg_t(\det(\kappa(\u)[t])) = Nm$.
	That is,
	\begin{equation}
		\label{eq:det deg equal}
		\deg_t(\det(\kappa(\u)[t])) = N\deg_t(\kappa(\u)[t]).
	\end{equation}

	Next, we recall that the adjugate of any $N\times N$ matrix is a matrix of determinants of $(N-1)\times (N-1)$ sub-matrices.
	Hence, for all $1\leq i,j\leq N$, 
	\begin{equation}
		\label{eq:adj deg submultitplicative}
		\deg_t(\text{adj}(\kappa(\u)[t])_{i,j})\leq (N-1)\deg_t(\kappa(\u)[t]).
	\end{equation}
	Since $\kappa(\u)[t]^{-1} = \det(\kappa(\u)[t])^{-1}\text{adj}(\kappa(\u)[t])$, 
		Equations \eqref{eq:det deg equal} and \eqref{eq:adj deg submultitplicative} imply
	\begin{equation}
		\label{eq:deg of inverse of poly is negative}
		\deg_t(\kappa(\u)[t]^{-1}) = 
			\max_{i,j}\set{\deg_t\left(\frac{\text{adj}(\kappa(\u)[t])_{i,j}}{\det(\kappa(\u)[t])}\right)}
			\leq 0.
	\end{equation}

	Next, $r(t\w) = \alpha(\w)[t]\beta(\w)[t]^{-1}$ for some $\alpha,\beta\in \C\skf{\bbw}[t]$
	and assuming $r(\u)$ is defined, $ r(t\u)=\alpha(\u)[t]\beta(\u)[t]^{-1}$.
	Hence,
	\begin{align*}
		\deg_t( r(t\u)_{i,j})
		&= \deg_t\left(\sum_{\ell=1}^N\alpha(\u)[t]_{i,\ell}\left(\beta(\u)[t]^{-1}\right)_{\ell,j}\right)\\
		&\leq \max_{\ell}\set{\deg_t\left(\alpha(\u)[t]_{i,\ell}\left(\beta(\u)[t]^{-1}\right)_{\ell,j}\right)}\\
		&\leq \max_{\ell}\set{\deg_t\left(\alpha(\u)[t]_{i,\ell}\right)}
			+\max_{\ell}\set{\deg_t\left(\left(\beta(\u)[t]^{-1}\right)_{\ell,j}\right)}\\
		&\leq \deg_t(\alpha(\u)[t])
			+\deg_t\left(\beta(\u)[t]^{-1}\right)\\
		&\leq \deg_t(\alpha(\u)[t]).
	\end{align*}
	Where the last inequality uses \eqref{eq:deg of inverse of poly is negative}.
	Finally, $\deg_t(\alpha(\u)[t])\leq \deg_t(\alpha(\w)[t])$ implies
	\[
		\deg_t( r(t\u)_{i,j}) \leq \deg_t(\alpha(\u)[t]) \leq \deg_t(\alpha(\w)[t])
	\]
	for all $1\leq i,j\leq N$.
\end{proof}

It should be emphasized that \eqref{eq:det deg equal} is not true for any polynomial in $M_N(\mathbb{S})[t]$.
Rather, it holds true if and only if the leading coefficient is an invertible matrix.
For example, if
\[
	\lambda(t) = \bpm 1 & 0 \\ 0 & 1\epm+\bpm 1 & 0 \\ 0 & 0\epm t^2 + \bpm 0 & 1 \\ 1 & 0\epm t^3 +\bpm 0 & 0 \\ 0 & 1\epm t^4
		= \bpm 1+t^2 & t^3 \\ t^3 & 1+t^4 \epm
\]
then $\deg_t(\lambda(t)) = 4$ while $\deg_t(\det(\lambda(t))) = \deg_t(1+t^2+t^4) = 4 < 8$.
On the other hand, if
\[
	\lambda(t) = \bpm 1 & 0 \\ 0 & 1\epm+\bpm 0 & 1 \\ 1 & 0\epm t^2 +\bpm 1 & 0 \\ 0 & 1\epm t^4
		= \bpm 1+t^4 & t^2 \\ t^2 & 1+t^4 \epm
\]
then $\deg_t(\lambda(t)) = 4$ and $\deg_t(\det(\mu(t))) = \deg_t(1+t^4 + t^8) = 8$.

Proposition~\ref{prop:r deg bdd} gives credence to the notion that if $r$ is a noncommutative rational function, 
then $r(\bm{s})$ is a matrix of rational functions whose behavior is modeled by $r$.
In particular we will apply this idea to generic matrix algebras.

%\begin{lemma}
%	\label{lem:isom val}
%	Suppose $U$ and $\hat{U}$ are isomorphic skew fields and $\deg_t$ and $\hat{\deg}_t$ are the rational degree
%	maps (see Definition \ref{def:rat deg def}) on $U(t)$ and $\hat{U}(t)$, respectively.
%	If $\sigma:U(t)\to \hat{U}(t)$ is an isomorphism, then $\deg_t(p) = \hat{\deg}_t(\sigma(p))$ for
%	all $p\in U(t)$.
%\end{lemma}
%
%\begin{proof}
%	If $a\in S$ is nonzero, then $\sigma(a t^K) = \sigma(a) t^K$ and $\deg_t(a t^K) = -K = \hat{\deg}_t(\sigma(a) t^K)$.
%	Hence by linearity, if $\alpha[t]\in S[t]$ then $\deg_t(\alpha[t]) = \hat{\deg}_t(\sigma(\alpha[t]))$.
%	Thus, 
%	\begin{align*}
%		\deg_t(\alpha[t]\beta[t]^{-1}) &= \deg_t(\alpha[t])-\deg_t(\beta[t]) 
%		= \hat{\deg}_t(\sigma(\alpha[t]))-\hat{\deg}_t(\sigma(\beta[t]))\\
%		&= \hat{\deg}_t(\sigma(\alpha[t]\beta[t]^{-1})).
%	\end{align*}
%\end{proof}

%
%    ######  ##     ## ########   ######  ########  ######  ######## ####  #######  ##    ## 
%   ##    ## ##     ## ##     ## ##    ## ##       ##    ##    ##     ##  ##     ## ###   ## 
%   ##       ##     ## ##     ## ##       ##       ##          ##     ##  ##     ## ####  ## 
%    ######  ##     ## ########   ######  ######   ##          ##     ##  ##     ## ## ## ## 
%         ## ##     ## ##     ##       ## ##       ##          ##     ##  ##     ## ##  #### 
%   ##    ## ##     ## ##     ## ##    ## ##       ##    ##    ##     ##  ##     ## ##   ### 
%    ######   #######  ########   ######  ########  ######     ##    ####  #######  ##    ##
%
\subsection{Generic matrix algebras}
	\label{ssec:generic}
Suppose $n\in\Z^+$.
For each $i\in \Z^+$, $1\leq j\leq g$, and $1\leq k,\ell\leq n$, let $\xi^{(i),j}_{n,k,\ell}$ be a commuting indeterminate.
Next, for each $i\in\Z^+$, set $\xi^{(i)}_{n}=\{\xi^{(i),j}_{n, k,\ell}:1\leq j\leq g,1\leq k,\ell\leq n\}$.
If $i,\hat{\imath}\in\Z^+$, then the algebras $\C[\xi^{(i)}_{n}]$ and $\C[\xi^{(i)}_{n}\cup \xi^{(\hat{\imath})}_{n}]$ 
have fields of fractions $\C(\xi^{(i)}_{n})$ and $\C(\xi^{(i)}_{n}\cup\xi^{(\hat{\imath})}_{n})$, respectively.
 % and set $\xi_\bbn=\xi_{n_1,\dots,n_H} = \bigcup_{i=1}^H \xi^{(i)}_{n_i}$.

For $i,n\in\Z^+$ and $1\leq j\leq g$, define $\Xi^{(i),j}_{n} = (\xi^{(i),j}_{n, k, \ell})_{k,\ell=1}^{n}\in M_n(\C[\xi^{(i)}_n])$ 
to be a \dfn{generic matrix} of size $n$.
Define $\GM_{n}(\Xi^{(i)})$ to be the \index{generic matrix!algebra}{\bf {algebra of generic matrices}};
that is, the unital $\C$-subalgebra of $M_{n}(\C[\xi^{(i)}_{n}])$ generated by 
$\Xi^{(i),1}_{n},\dots,\Xi^{(i),g}_{n}$.
Let 
\[
	\bm{\Xi}^{(i)}_{n} = (\Xi^{(i),1}_{n},\dots,\Xi^{(i),g}_{n})
\]
be a $g$-tuple of generic matrices.

To prepare for their use in Section~\ref{sec:Hyporationals}, we define
$\GM_{n}((\Xi^{(i)})^T)$ to be the algebra of transposed generic matrices, that is,
$\GM_{n}((\Xi^{(i)})^T)$ is the algebra generated by $(\Xi^{(i),1}_{n})^T,\dots,(\Xi^{(i),g}_{n})^T$.
Let 
\[
	(\bm{\Xi}^{(i)}_{n})^T  = ((\Xi^{(i),1}_{n})^T,\dots,(\Xi^{(i),g}_{n})^T)
\] 
be a $g$-tuple of transposed generic matrices and let $\bm{\Xi}^{T, (j)}_{n}$ denote the $2g$-tuple
\[
	\left((\Xi^{(j),1}_{n})^T\otimes I_n,\dots,(\Xi^{(j),g}_{n})^T\otimes I_n,
	I_n\otimes\Xi^{(j),1}_{n},\dots,I_n\otimes\Xi^{(j),g}_{n}\right)
\]
where $\otimes$ is the Kronecker product.
Lastly, let 
\[
	\GM_n(\Xi^{T,(j)}) = \GM_n((\Xi^{(j)})^T)\otimes\GM_n(\Xi^{(j)})
\]
and observe $\GM_n(\Xi^{T,(j)})$ is generated by $\{(\Xi^{(i)}_n)^T\otimes I_n\}\cup\{I_n\otimes \Xi^{(j)}_n\}$.

\begin{remark}
	\label{rem:transpose ok}
	By Lemma~2.5 and Proposition~2.6 in \cite{KVV16}, $\GM_{n}(\Xi^{(i)}) \otimes \GM_{n}(\Xi^{(j)})$ is contained in a skew field,
	$\UD_{n}(\Xi^{(i),(j)})$.
	Since $\GM_{n}(\Xi^{(i)}) \otimes \GM_{n}(\Xi^{(j)})$ and $\GM_{n}(\Xi^{T,(j)})$ are isomorphic as algebras, $\GM_{n}(\Xi^{T,(j)})$ must
	be contained in some skew field, $\UD_{n}(\Xi^{T,(j)})$.

	Thus, for any $n\in\Z^+$, 
	\begin{enumerate}[label=(\roman*)]
		\item $\C(\xi^{(1)}_n)$ is a field,
		\item $\GM_{n}(\Xi^{T,(1)})\subset M_{n^2}(\C(\xi^{(1)}_n))$ is the $\C$-algebra generated by
				$\{(\Xi^{(1)}_n)^T\otimes I_n\}\cup\{I_n\otimes \Xi^{(1)}_n\}$, 
		\item  $\GM_{n}(\Xi^{T,(1)})\subset \UD_{n}(\Xi^{T,(1)})\subset M_{n^2}(\C(\xi^{(1)}_n))$
			and $\UD_{n}(\Xi^{T,(1)})$ is a skew field.
	\end{enumerate}
	Hence, Proposition~\ref{prop:r deg bdd} is applicable; 
	if $r\in\C\skf{\bby\cup\bbx}$, then there exists a $d_r\in\Z^+$ such that
	$\deg_t(r(t\bm{\Xi}^{T, (1)}_n )_{i,j})\geq d_r$, for all $n\in\Z^+$ and $1\leq i,j\leq n^2$,
	whenever $r(\bm{\Xi}^{T, (1)}_n)$ is defined.
\end{remark}

\begin{example}
	By the Cayley-Hamilton theorem, any $X\in M_2(\C)$ satisfies the relation $X^2 = c_1 X + c_0 I_2$, for some scalars $c_1,c_0\in \C$.
	Take the commutator of both sides against $Y\in M_2(\C)$,
	\[
		[X^2,Y] = c_1[X,Y] + c_0[I_2,Y] = c_1[X,Y].
	\]
	Next, take the commutator of both sides against $[X,Y]$,
	\[
		[[X^2,Y],[X,Y]] = c_1[[X,Y],[X,Y]] = 0.
	\]
	Thus, $p(y,x) = [[x^2,y],[x,y]]$ vanishes on $M_2(\C)$, i.e. $p$ is a polynomial identity for $M_2(\C)$.
	Let $r(y,x) = x^2(1-p(y,x))^{-1}$ and note $r$ is a fraction of the polynomials $x^2$ and $1-p(y,x)$.
	Thus, Proposition~\ref{prop:r deg bdd} implies $\deg_t(r(t\bm{\Xi}_n^{T,(1)})_{i,j})\leq \deg_t(t^2x^2) = 2$, for all $n\in\Z^+$.
	In particular, if $n<3$ then $p(t\bm{\Xi}_n^{T,(1)}) = I_2$ and
	 $\deg_t(r(t\bm{\Xi}_n^{T,(1)})_{i,j}) = 2$.
\end{example}

\begin{lemma}
	\label{lem:t deg bdd poly}
	Suppose $\eta:M(\C)^g\to M(\C)^g$ is a free analytic function with a power series that converges for each 
	$X\in M(\C)^g$ and $\eta(\bm{\Xi}^{(1)}_n)$ is a polynomial for each $n$.
%	Let $\eta(\bm{\Xi}^{(1),(2)}_n)\in (M_n(\C[\xi^{(1)}_n\cup\xi^{(2)}_n]))^{2g}$, 
%	i.e. $\eta(\bm{\Xi}^{(1),(2)}_n)$ is a polynomial for each $n$.
	Define
	\[
		T = \set{\deg_t(\eta_j(t\bm{\Xi}^{(1)}_n)_{k,\ell}):n\in\Z^+,1\leq j\leq g,1\leq k,\ell\leq n}.
	\]	
	If $T$ is bounded, then $\eta$ is a free polynomial.
\end{lemma}

\begin{proof}
	This is Proposition~3.1 in \cite{KlSp17}.
\end{proof}

\section{Hyporational series}
	\label{sec:Hyporationals}
Proposition~\ref{prop:p bij F bij} shows that if $p$ is a free polynomial, then $F(\x,\y) = (Dp(\y)[\x],\y)$ is a free polynomial,
and $p$ is injective if and only if $F$ is injective.
Lemma~\ref{lem:F poly inv} implies $G$, the inverse of $F$, is the unique solution to a 
proper algebraic polynomial that is $\z$-affine linear.
The $\z$-affine linearity is reminiscent of realizations of nc rational functions, see \cite{Vol15} and \cite{KlVo17}.
With this similarity to realizations in mind, we generalize the class of rational series (see Definition \ref{def:rational series})
to a slightly larger class of formal power series that we call the hyporational series.
In particular, the scion of a free polynomial has a hyporational series as its compositional inverse.
We show that every rational series is hyporational and Theorem~\ref{thm:hyporat domain poly}, 
the main result of subsection~\ref{ssec:hyporeal and hyporat}, says that a hyporational series without singularities is a free polynomial.

In subsection~\ref{ssec:bij crit}, we apply the same techniques used to analyze hyporational series to the free derivative of polynomials. 
This leads to the construction of the hypo-Jacobian matrix of a free polynomial mapping and 
Theorem~\ref{thm:matrix free inverse function theorem}.
Finally, combining Theorem~\ref{thm:matrix free inverse function theorem} with results on automorphisms of $\C\fax$ proves the main
result of this paper, Theorem~\ref{thm:Free Groth Thm}.

\subsection{Hyporealizations and hyporational series}
	\label{ssec:hyporeal and hyporat}

\begin{definition}
	Once more, suppose $\bbz = \set{z_1,\dots,z_h}$ is a set of freely noncommuting indeterminates where $h$ is not necessarily
	equal to $g$.
	Let $\bba\in (\C\fralg{\bbx\cup\bbz})^h$.
	Recall from Definition \ref{def:prop alg poly} that $\bba$ is a {\bf proper algebraic polynomial} if
	$\bba$ has no constant term and $\bbd_z(\bba)>1$, that is, if $w$ is a monomial appearing in $\bba$ 
	with $\abs{w}_z >0$ then $\abs{w}\geq 2$.
	We say $\bba$ is $\z$-affine linear if
	$\bba(\x)[\z] = \mfa(\x) + \alpha(\x)[\z]$, where $\alpha(\x)[\z + \w] = \alpha(\x)[\z] + \alpha(\x)[\w]$.
	If $\bba$ is both a proper algebraic polynomial and $\z$-affine linear then we say $\bba$ is a \dfn{hyporealization}.
	
	Suppose $r\in \C\fpsx$ with constant term $r_\ew$.
	We define $r$ to be a \dfn{hyporational series} if there exists a hyporealization $\bba$ such that 
	$r-r_\ew$ is a component of the solution of $\bba$.
	Namely, there exists $\bm{r} = (\bm{r}_1,\dots,\bm{r}_h)\in (\C\fps{\bbx}_+)^h$ such that $\bba(\x)[\bm{r}] = \bm{r}$ 
	and $r-r_\ew = \bm{r}_i$.
	Let $\C_{\textup{hyp}}\fpsx\subset \C\fpsx$ denote the collection of all hyporational series.

%	We say $\bba$ is a \dfn{hyporealization} of $r$.
%	In fact, without referencing $r$, we say that any proper algebraic polynomial $\bba$ that is $\z$-affine linear
%	is a {\bf hyporealization}.
\end{definition}

Recall from Definition \ref{def:rational series} that $\C_{\textup{rat}}\fpsx$ is the algebra of rational series.

\begin{remark}
	\label{rem:rat series regular at 0}
	We recall several facts from realization theory.	
	Let $\C\skf{\bbx}_0\subset \C\skf{\bbx}$ denote the subring of nc rational functions that are regular at the origin:
	\[
		\C\skf{\bbx}_0 = \set{r\in\C\skf{\bbx}:0\in\dom r}.
	\]
	As was mentioned in Remark \ref{rem:reg at 0 are rats}, $\C\skf{\bbx}_0 = \C_\textup{rat}\fpsx$.
	If $r\in \C\skf{\bbx}_0$, then $r$ has a realization; there exist $d\in\Z^+$, $c,b\in M_{d\times 1}(\C)$ 
	and $A_1,\dots,A_g\in M_d(\C)$ such that
	\[
		r = c^T\left(I_d - \sum_{j=1}^g A_j x_j\right)^{-1}b.
	\]
	Classical realization theory has a long and storied history in both mathematics and applied fields.
	We use definitions and results from \cite{KlVo17}, which provides an excellent exposition of realizations
	of nc rational functions and their domains.
\end{remark}

\begin{remark}
	\label{rem:rational mat rep}
	Every rational series is hyporational, that is, $\C_{\textup{rat}}\fpsx\subset \C_{\textup{hyp}}\fpsx$.
	We omit the proof of the above statement, however it follows readily from a rearrangement of the realization of a given nc rational map.
	
	In fact, if $r$ is a formal power series with constant term $r_\ew$, 
	then $r\in\C_\textup{rat}(\fpsx)$ if and only if there exist $\mathscr{A}\in M_h(\C\fax_+)$ and $\mfa\in (\C\fax_+)^h$, 
	such that $r-r_\ew$ is a component of the solution to the hyporealization $\bba(\x)[\z] = \mfa(\x) + \z \mathscr{A}(\x)$.
	This condition precisely delineates the difference between rational series and hyporational series that are not rational.
\end{remark}

Realization theory tells us that there is a very intimate relationship between rational functions and linearity.
Example \ref{ex:classic alg} provides us with a function that is hyporational but not rational.

\begin{example}
	\label{ex:classic alg revisited}	
	The hyporealization $\mathbbm{s}(\x)[\z] = x_2 + x_1z_2x_1$ has the solution $s(\x)=\sum_{n=0}^\infty x_1^nx_2x_1^n$.
	Thus $s$ is hyporational and $\mathbbm{s}(\x)[\z]$ is a hyporealization of $s$.
	
	Arguing by contradiction, suppose $s$ is rational.
	Hence $s$ has a minimal representation $s(\x) = c^T (I- Ax_1 - B x_2)^{-1} b$,
	where $A,B\in M_n(\C)$ and $b,c\in M_{n\times 1}(\C)^g$.
	Let $m_A(t) = a_0 + a_1 t + \dots + a_d t^d$ be the minimal polynomial of $A$ and note there is a $k$
	so that $a_k\neq 0$ since $A\neq 0$.
	Observe $c^T A^i B A^j b = \delta(i,j)$ where $\delta$ is the Kronecker delta.
	Hence
	\begin{align*}
		0 
			&= c^T 0 b = c^T \overline{a_k}m_A(A)BA^k b = \sum_{i=0}^d a_i\overline{a_k} c^T\!\!\left(A^i B A^k\right)b\\
			&= \abs{a_k}^2 c^T\!\!\left(A^k B A^k\right)b = \abs{a_k}^2\\
			&>0,
	\end{align*}
	a contradiction.
	Therefore, $s$ is not rational.
\end{example}

One of the main advantages of the realization theory of nc rational functions is that the intrinsic linearity of rationals is expressed
through matrices.
Since hyporational series are generated by linear proper algebraic polynomials, we would like to imitate rational realization theory
for hyporational series.
This is precisely what we do in subsection \ref{ssec:hypomat reps}.

%
%    ######  ##     ## ########   ######  ########  ######  ######## ####  #######  ##    ## 
%   ##    ## ##     ## ##     ## ##    ## ##       ##    ##    ##     ##  ##     ## ###   ## 
%   ##       ##     ## ##     ## ##       ##       ##          ##     ##  ##     ## ####  ## 
%    ######  ##     ## ########   ######  ######   ##          ##     ##  ##     ## ## ## ## 
%         ## ##     ## ##     ##       ## ##       ##          ##     ##  ##     ## ##  #### 
%   ##    ## ##     ## ##     ## ##    ## ##       ##    ##    ##     ##  ##     ## ##   ### 
%    ######   #######  ########   ######  ########  ######     ##    ####  #######  ##    ##
%
\subsection{Hypomatrix representations}
	\label{ssec:hypomat reps}
	
\begin{definition}
	For any $\C$-algebra $R$ and $n\in\Z^+$, the map $\vecc[n]:M_n(R)\to M_{1\times n^2}(R)$ given by
	\[
		\vecc[n](A)_{(i-1)n+j} = A_{i,j}
	\]
	is a linear isomorphism taking an $n\times n$ matrix with entries in $R$ to a length $n^2$ row vector
	with entries in $R$.
	
	If $U,V\in M_n(R)$ then $U\otimes V\in M_{n^2}(R)$ is the standard Kronecker product.
	Furthermore, if $Z,U,V\in M_n(R)$ then the product $\vecc[n](Z) \cdot (U\otimes V)$ is the product of a $1\times n^2$
	row vector and an $n^2\times n^2$ matrix.
\end{definition}

	Let us see $\vecc[n]$ in action. 
	If $A\in M_3(\C[x_1,\dots,x_9])$ with
	\[
		A = \bpm x_1 & x_2 & x_3\\ x_4 & x_5 & x_6\\ x_7 & x_8 & x_9 \epm,
	\]
	then $\vecc[3](A) = \bpm x_1, & x_2, & x_3, & x_4, & x_5, & x_6, & x_7, & x_8, & x_9 \epm$.
	
	For notational convenience, we allow $\vecc[n]$ to apply coordinate-wise to tuples:
	\[
		\vecc[n]\left((A_1,\dots,A_m)\right) = (\vecc[n](A_1),\dots,\vecc[n](A_m)).
	\]

\begin{lemma}
	\label{lem:vecc kron}
	If $U,V,Z\in M_n(R)$ then 
	\[
		\vecc[n](UZV) = \vecc[n](Z)\cdot\left(U^T\otimes V\right)\in M_{1\times n^2}(R).
	\]
	
\end{lemma}

\begin{proof}	
	Our definition of $\vecc[n]$ is a left sided version of the $\vecc$ function defined at 4.2.9 in \cite{HoJo94}.
	By adapting Lemma~4.3.1 in \cite{HoJo94} we conclude $\vecc[n](UZV) = \vecc[n](Z)\cdot (U^T\otimes V)$.
\end{proof}

Recall
\[
\bm{\Xi}^{(1)}_n 
= \left(\Xi^{(1),1}_{n},\dots,\Xi^{(1),g}_{n}\right)
\]
is a $g$-tuple of generic matrices while
\[
\bm{\Xi}^{T, (1)}_n 
= \left((\Xi^{(1),1}_{n})^T\otimes I_n,\dots,(\Xi^{(1),g}_{n})^T\otimes I_n,
I_n\otimes\Xi^{(1),1}_{n},\dots,I_n\otimes\Xi^{(1),g}_{n}\right),
\]
is a $2g$-tuple of $n^2$-matrices over $\C[\xi^{(i)}]$.

Borrowing from \cite{KVV16}, we define
\begin{align*}
	\C\fralg{\bby\lra \bbx} 
		\defeq&\hphantom{t}  \C\fralg{\bby}\otimes \C\fralg{\bbx}\\
		\cong&\hphantom{t}  \C\fralg{\bby\cup\bbx}\big / \big([y_i,x_j]:1\leq i,j\leq g\big),
\end{align*}
to be the \dfn{bipartite free $\C$-algebra}.
The algebra $\C\fralg{\bby\lra \bbx}$ is contained in a skew field of fractions, 
$\C\skf{\bby\lra \bbx}$, the bipartite rational functions.

\begin{remark}
	We briefly define the transpose of a polynomial.
	For any $w\in \fax$ with $w = x_{i_1}x_{i_2}\dots x_{i_n}$, we say $w^T = x_{i_n}x_{i_{n-1}}\dots x_{i_1}$.
	Hence, for any polynomial $p = \sum p_w w$, we define $p^T = \sum p_w w^T$.
	In particular, if $X = (X_1,\dots, X_g)$ is a tuple of matrices, then $p^T(X^T) = (p(X))^T$.
\end{remark}

\begin{proposition}
	\label{prop:hyporat has hypomat}
	Suppose $\bba\in (\C\fralg{\bbx\cup\bbz})^g$ is a hyporealization
	with $\bba(\x)[\z] = \mfa(\x) + \alpha(\x)[\z]$.
	If $a\in (\C\fpsx_+)^g$ is the hyporational series such that $\bba(\x)[a(\x)] = a(\x)$,
	then there exists $\Phi\in M_g(\C\fralg{\bby\lra\bbx})$ such that
	\begin{enumerate}[label=(\roman*)]
		\item $\vecc[n]\left(\bba(X)[Z]\right) = \vecc[n](\mfa(X)) + \vecc[n](Z)\cdot\Phi(X^T\otimes I,I\otimes X)$,
	 		for all $n$ and $X,Z\in M_n(\C)^g$;
	 	\item $(I-\Phi(\y,\x))$ is invertible as a matrix over $\C\skf{\bby\lra \bbx}$;
	 	\item $\dom_{nm}((I - \Phi)^{-1}_{i,j})\neq \varnothing$ for all $n,m\in\Z^+$ and $1\leq i,j\leq g$;
	 	\item $(I-\Phi(\bm{\Xi}^{T,(1)}_n))^{-1}$ is defined for all $n\in \Z^+$;
	 	\item $\vecc[n]\big(a\big(\bm{\Xi}^{(1)}_n\big)\big) 
	 	= \vecc[n]\big(\mfa\big(\bm{\Xi}^{(1)}_n\big)\big)\cdot \big(I-\Phi\big(\bm{\Xi}^{T,(1)}_n\big)\big)^{-1}$,
	 	for all $n\in\Z^+$.
	 \end{enumerate}
\end{proposition}

\begin{proof}
	For $1\leq i\leq g$, we write
	\[
		\alpha^i(\x)[\z] = \sum_{\ell}\sum_{j=1}^g U^\ell_{i,j}(\x)z_j V^\ell_{i,j}(\x),
	\]
	where each $U_{i,j},V_{i,j}\in \C\fax$.
	Set $\Phi_{j,i}(\y,\x) = \sum_{\ell}(U_{i,j}^\ell)^T(\y)\otimes V_{i,j}^\ell(\x)$ and note
	\begin{align*}
		\Phi\left(X^T\otimes I, I\otimes X\right) &= \left(\Phi_{j,i}\left(X^T\otimes I, I\otimes X\right)\right)_{j,i=1}^g\\
			&= \Big(\sum_\ell (U_{i,j}^\ell)^T (X^T\otimes I)V_{i,j}^\ell(I\otimes X) \Big)_{j,i=1}^g,
	\end{align*}
	for all $X\in M(\C)^g$.
	Let $1\leq i\leq g$,
	\begin{align*}
		\vecc[n](\alpha^i(X)[Z]) &= \vecc[n]\Big( \sum_{\ell}\sum_{j=1}^g U^\ell_{i,j}(X)Z_j V^\ell_{i,j}(X) \Big)\\
			&=\sum_{\ell}\sum_{j=1}^g \vecc[n]\left(U^\ell_{i,j}(X)Z_j V^\ell_{i,j}(X)\right)\\
			&=\sum_{j=1}^g\sum_{\ell} \vecc[n](Z_j)\cdot\left(U^\ell_{i,j}(X)^T\otimes V^\ell_{i,j}(X)\right),
	\end{align*}
	where the last equality is using Lemma~\ref{lem:vecc kron}.
	Continuing on,
	\begin{align*}
		\vecc[n](\alpha^i(X)[Z])
			&=\sum_{j=1}^g\sum_{\ell} \vecc[n](Z_j)\cdot\left(U^\ell_{i,j}(X)^T\otimes V^\ell_{i,j}(X)\right)\\
			&=\sum_{j=1}^g\vecc[n](Z_j) \sum_{\ell} \left((U^\ell_{i,j})^T(X^T)\otimes V^\ell_{i,j}(X)\right)\\
			&=\sum_{j=1}^g\vecc[n](Z_j) \sum_{\ell} \left((U^\ell_{i,j})^T(X^T\otimes I)V^\ell_{i,j}(I\otimes X)\right)\\
			&=\sum_{j=1}^g\vecc[n](Z_j)\sum_{j=1}^g \Phi_{j,i}(X^T\otimes I,I\otimes X)\\
			&=\left(\vecc[n](Z)\cdot\Phi\left(X^T\otimes I,I\otimes X\right)\right)_i.
	\end{align*}
	Therefore, $\vecc[n](\alpha(X)[Z]) = \vecc[n](Z)\Phi(X^T\otimes I,I\otimes X)$ and 
	\begin{equation}
		\label{eq:bba Phi}
		\vecc[n](\bba(X)[Z]) = \vecc[n](\mfa(\x)) + \vecc[n](Z)\cdot\Phi(X^T\otimes I,I\otimes X)
	\end{equation}
	for all $n$ and $X,Z\in M_n(\C)^g$.
	Thus, item (i) is proved.
	
	For item (ii), note $I-\Phi(\y,\x)\in M_g(\C\fralg{\bby\lra \bbx})$, hence
	$I-\Phi(\y,\x)\in M_g(\C\skf{\bby\lra \bbx})$.
	Since $I-\Phi(0,0) = I$, is invertible, Proposition~3.8 in \cite{KVV16} implies $I-\Phi$ is invertible
	and $(I-\Phi(\y,\x))^{-1}\in M_g(\C\skf{\bby\lra \bbx})$.
	
	For item (iii) we note $(0_n\otimes I_m,I_n\otimes 0_m) = (0_{nm},0_{nm})\in \dom_{nm}((I-\Phi)^{-1}_{i,j})$ 
	since $I_{nm}-\Phi(0_{nm},0_{nm}) = I_{nm}$ is invertible.
	
	Item (iv) is simply a consequence of (iii);
	\[
		\dom_{n^2}((I-\Phi)^{-1}_{i,j})\subset \dom ((I-\Phi(\bm{\Xi}^{T,(1)}_n))^{-1}_{i,j}).
	\]
	Finally, substitute $\bm{\Xi}^{(1)}_n$ in for $X$ and $a(\bm{\Xi}^{(1)}_n)$ in for $Z$ in \eqref{eq:bba Phi} to get
	\begin{align*}
		\vecc[n](a(\bm{\Xi}^{(1)}_n)) &= \vecc[n](\bba(\bm{\Xi}^{(1)}_n)[a(\bm{\Xi}^{(1)}_n)])\\
			&= \vecc[n](\mfa(\bm{\Xi}^{(1)}_n)) + \vecc[n](a(\bm{\Xi}^{(1)}_n))\cdot\Phi(\bm{\Xi}^{T,(1)}_n).
	\end{align*}
	Hence by rearranging,
	\[
		\vecc[n]\big(a\big(\bm{\Xi}^{(1)}_n\big)\big)\cdot\big(I_{g n^2} - \Phi\big(\bm{\Xi}^{T,(1)}_n\big)\big)
			= \vecc[n]\big(\mfa\big(\bm{\Xi}^{(1)}_n\big)\big).
	\]
	Multiplying both sides on the right by $(I_{g n^2}-\Phi(\bm{\Xi}_n^{T,(1)}))^{-1}$ yields
	\[
		\vecc[n]\big(a\big(\bm{\Xi}^{(1)}_n\big)\big) 
			= \vecc[n]\big(\mfa\big(\bm{\Xi}^{(1)}_n\big)\big)
			\cdot\big(I_{g n^2} - \Phi\big(\bm{\Xi}^{T,(1)}_n\big)\big)^{-1},
	\]
	as desired.
\end{proof}

If $a$ is a hyporational series then $a[n] = a\lvert_{M_n(\C)^g}$ is a rational function in $g n^2$ commuting indeterminates by
Proposition~\ref{prop:hyporat has hypomat}(v).
In fact, Proposition~\ref{prop:hyporat has hypomat} shows that a small amount of commutativity is all the prevents a hyporational
from being rational.

\begin{definition}
	\label{def:hypomatrix}
	Suppose $\bba(\x)[\z] = \mfa(\x) + \alpha(\x)[\z]$ is hyporealization.
	Let  $\Phi\in M_g(\C\fralg{\bby\lra\bbx})$ be the matrix constructed in Proposition~\ref{prop:hyporat has hypomat}. 
	We define $\Phi$ to be the \dfn{hypomatrix representation} of $\bba$.
	That is,
	\[
		\vecc[n](\bba(X)[Z]) = \vecc[n](\mfa(X)) + \vecc[n](Z)\cdot \left( I - \Phi(X^T\otimes I, I\otimes X) \right)
	\]
	for all $X,Z\in M_n(\C)^g$ and $n\in\Z^+$.
		
	Let $a\in \C_{\text{hyp}}\fpsx$ be hyporational.
	Define
	\[
		\dom_n(a) = \bigcup_{\Phi\in \mathscr{A}} \dom_n((I-\Phi)^{-1}),
	\]
	where $\mathscr{A}$ is the collection of all hypomatrix representations of $a$.
	
	Suppose $X\in \dom_n(a)$.
	Let $\bba(\x)[\z] = \mfa(\x) + \alpha(\x)[\z]$ be a hyporealization such that $a$ is the first
	component of the solution of $\bba$ and $X\in \dom_n( (I - \Phi)^{-1})$, where $\Phi$ is the associated
	hypomatrix representation of $\bba$.
	We define
	\[
		a(X) = \left(\vecc[n]^{-1}\left(	\vecc[n](\mfa(X))\cdot \left(	I - \Phi(X^T\otimes I, I \otimes X)	\right)^{-1}\right)	\right)_1.
	\]
	Thus, we can evaluate $a$ at any $X\in\dom_n(a)$.
	\end{definition}

\begin{example}
	Let $p(x_1,x_2,x_3) = (x_1, x_2 - x_1^2, x_3 + x_1  (-x_2 + x_1 x_2) - x_2^2)$,
	hence 
	\[
		p(\x) = \x
		\bpm
			1 & -x_1 & -x_2 + x_1 x_2 \\ 0 & 1 & -x_2\\ 0 & 0 & 1
		\epm
		=\x
		\bpm
			1 & x_1 & x_2 \\ 0 & 1 & x_2\\ 0 & 0 & 1
		\epm^{-1}
	\]
	and $\bbq(\x)[\z] = (x_1,x_2 + x_1 z_1, x_3 + x_1 z_2 + x_2 z_2)$.
	Note $\bbq$ is a hyporealization and setting
	\[
		\Phi(\y,\x) = \bpm 0 & y_1\otimes 1 & 0 \\ 0 & 0 & (y_1+y_2)\otimes 1 \\ 0 & 0 & 0\epm,
	\]
	$\Phi$ is the hypomatrix representation of $\bbq$.
	Observe $\Phi$ is nilpotent, and
	\[
		(I-\Phi(\y,\x))^{-1} 
		= \bpm 1 & y_1\otimes 1 & y_1(y_1+y_2)\otimes 1 \\ 0 & 1 & (y_1+y_2)\otimes 1 \\ 0 & 0 & 1\epm,
	\]
	and
	\[
		\vecc[n](\bm{\Xi}^{(1)}_n)\left(I-\Phi\left(\bm{\Xi}_n^{T,(1)}\right)\right)^{-1}
		= \vecc[n](q(\bm{\Xi}_n^{(1)})).
	\]
	Thus $q(\x) = (x_1, x_2 + x_1^2, x_3 + (x_1+x_2)x_2 + (x_1 + x_2)x_1^2)$.
\end{example}

\begin{example}
	\label{ex:classic alg revisted again}
	We once again revisit Example \ref{ex:classic alg}.
	That is, $p(\x) = (x_1,x_2-x_1x_2x_1)$,	$q(\x) = (x_1,\sum_{n=0}^\infty x_1^nx_2x_1^n)$
	and the auxiliary inverse of $p$ is $\bbq(\x)[\z] = (x_1,x_2 + x_1z_2z_1)$.
	Since $q^1(\x) = x_1$, $\hat{\bbq}(\x)[\z] = \bpm x_1, & x_2 + x_1z_2x_1 \epm$ is a hyporealization with
	$\hat{\bbq}(\x)[q(\x)] = q(\x)$.
	The hypomatrix representation of $\hat{\bbq}$ is
	\[
		I-\Phi(\y,\x) = \bpm 1\otimes 1 & 0 \\ 0 & 1\otimes 1-y_1\otimes x_1 \epm,
	\]
	with inverse,
	\[
		(I-\Phi(\y,\x))^{-1} = \bpm 1\otimes 1 & 0 \\ 0 & (1\otimes 1-y_1\otimes x_1)^{-1} \epm.
	\]
	Recall that the inverse of $J_p$, the Jacobian matrix of $p$, is a polynomial matrix, however $p$ is not injective.
	In this case, the hypomatrix representation witnesses the non-injectivity of $p$ since $(I-\Phi)^{-1}$ is not a polynomial matrix.
\end{example}

Recall (see Lemma~\ref{lem:once referenced}) that if $r\in \C\skf{\bbx}$ then $r(t\x)\in \C\skf{\bbx}(t)$.
Since $\C\skf{\bbx}$ is a skew field, $\C\skf{\bbx}(t)$ is the classical ring of quotients of $\C\skf{\bbx}[t]$.
Hence there exist $\alpha,\beta\in\C\skf{\bbx}[t]$ such that $r(t\x) = \alpha(\x)[t]\beta(\x)[t]^{-1}$ and
we can extend $\deg_t$ to $\C\skf{\bbx}(t)$.

\begin{lemma}
	\label{lem:poly deg lower bound}
	If $a\in (\C\fpsx)^g$ is a tuple of hyporational series then there exists $\Delta\in\Z$ such that for all $n\in\Z^+$, 
	\begin{equation}
		\label{eq:deg a bdd Delta}
		\max \set{ \deg_t\big(a_k(t\bm{\Xi}_n^{(1)})_{i,j}\big) : 1\leq k\leq g,  1\leq i,j\leq n} \leq \Delta.
	\end{equation}
\end{lemma}

\begin{proof}
	Let $\bba(\x)[\z] = \mfa(\x) + \alpha(\x)[\z]$ be a hyporealization of $a$ and let  $\Phi$ be the hypomatrix representation of $\bba$.
	We begin by noting that $\mfa$ is a free polynomial, hence for $1\leq i,j\leq n$ and $1\leq k\leq g$,
		$\deg_t(\mfa_k(t\bm{\Xi}_n^{(1)})_{i,j})\leq \deg_t(\mfa_k(t\x)).$
	Set $\delta = \max_{k} \set{\deg_t(\mfa_k(t\x))}$.
	Since $\vecc[n]$ preserves the entries of matrices,
	\begin{align*}
		\max_{1\leq \ell\leq g n^2} \set{\deg_t\big(\vecc[n]\big(\mfa(t\bm{\Xi}_n^{(1)})\big)_\ell\big)} 
			&= \max_{\substack{1\leq k\leq g\\ 1\leq i,j\leq n}}\set{\deg_t(\mfa_k(t\bm{\Xi}_n^{(1)})_{i,j})}\\
			&\leq \max_{1\leq k\leq g}\set{\deg_t(\mfa_k(t\x))} = \delta.
	\end{align*}
	Next we recall $(I-\Phi)^{-1}\in M_g(\C\skf{\bby\cup\bbx})$ and $(I-\Phi(t\bm{\Xi}^{T,(1)}_n))^{-1}$ exists for all $n\in\Z^+$.
	For notational ease we set $\Psi = (I-\Phi)^{-1}$.
	Each $\Psi_{i,j}(t\y,t\x)\in \C\skf{\bby\lra\bbx}(t)$, thus by Proposition~\ref{prop:r deg bdd} we know there exists
	$d_{i,j}\in \Z^+$ such that for all $n\in\Z^+$, $\deg_t(\Psi_{i,j}(t\bm{\Xi}_n^{T,(1)}))\leq d_{i,j}$.
	If we set $D = \max_{i,j}\set{d_{i,j}}$, then for all $n\in \Z^+$,
	\[
		\max_{1\leq k,\ell\leq g n^2} \set{\deg_t\big(\Psi \big( t\bm{\Xi}_n^{T,(1)} \big)_{k,\ell} \big)} \leq D.
	\]

	Proposition~\ref{prop:hyporat has hypomat} says for each $n\in \Z^+$, $\vecc[n](a(t\bm{\Xi}_n^{(1)}))$ is the product
	of the row vector $\vecc[n](\mfa(t\bm{\Xi}_n^{(1)}))$ and the matrix $\Psi(t\bm{\Xi}_n^{T,(1)})$.
	Thus, in light of the degree bounds found above and Lemma~\ref{lem:min val bdd}, we have
	\[
		\max_{1\leq k\leq gn^2} \set{\deg_t\big(\vecc[n]\big(a(t\bm{\Xi}_n^{(1)})\big)_k\big)} \leq \delta+D.
	\]
	Finally, set $\Delta = \delta + D$ and observe once more that since $\vecc[n]$ preserves the entries of matrices it must preserve
	the degrees of the entries.
	Therefore, for all $n\in\Z^+$,
	\begin{align*}
		\max \set{\deg_t\big(a_k(t\bm{\Xi}_n^{(1)})_{i,j}\big):1\leq k\leq g,  1\leq i,j\leq n} \leq \Delta.
		\tag*{$\square$}
	\end{align*}
\end{proof}

\THMhyporatdomainpoly

\begin{proof}
	The hyporationality of $a$ implies $a\lvert_{M_n(\C)^g}$ is a matrix of commutative rational functions.
	In fact, $\vecc[n]\circ a\lvert_{M_n(\C)^g}\circ \vecc[n]^{-1}:\C^{g n^2}\to \C^{n^2}$ is an $n^2$-tuple
	of rational functions in $g n^2$ commuting indeterminates.
	Since $\dom_n(a) = M_n(\C)^g$, each $(\vecc[n]\circ a\lvert_{M_n(\C)^g}\circ \vecc[n]^{-1})_k$ is a rational function with domain of $\C^{g n^2}$, hence each is a polynomial.
	Thus, each $a(\bm{\Xi}^{(1)}_n)_{i,j}$ is a polynomial and in particular, $a(\bm{\Xi}^{(1)}_n)$ is a polynomial.
	However, Lemma~\ref{lem:poly deg lower bound} tells us that there is some $\Delta\in\Z^+$ such that for all $n\in \Z^+$,
	$\max_{i,j,k} \{\deg_t(a_k(t\bm{\Xi}_n^{(1)})_{i,j})\} \leq \Delta$.
	That is, $a$ is a free analytic function such that $a(\bm{\Xi}^{(1)}_n)$ is a polynomial for each $n$,
	and	the degree of the polynomials is bounded.
	Therefore, Lemma~\ref{lem:t deg bdd poly} implies $a$ is a free polynomial.
\end{proof}

%\CORpbijqhyp
%
%\begin{proof}
%	By Lemma~\ref{lem:q agrees with poly} we know that $q\lvert_{M_n(\C)^g}$ agrees with a free polynomial for each $n\in\Z^+$.
%	Hence, $q$ is hyporational and $\dom_n(q) = M_n(\C)^g$, for all $n$.
%	Therefore, Theorem~\ref{thm:hyporat domain poly} implies $q$ is a free polynomial.
%\end{proof}

\CORFpolyinv

\begin{proof}
	Suppose $p:M(\C)^g\to M(\C)^g$ is a free polynomial and $F$ is its scion.
	Recall $F(\x,\y) = (Dp(\y)[\x],\y)$.
	If $F$ has a free polynomial inverse then $F$ is bijective.
	Hence Proposition~\ref{prop:p bij F bij} implies $p$ is bijective.
	
	Suppose $p$ is bijective. 
	An application of Proposition~\ref{prop:p bij F bij} shows $F$ is bijective.
	Let $G$ be the inverse of $F$.
	Theorem~\ref{thm:PMI is poly} implies $J_F$, the Jacobian matrix of $F$ is invertible as a polynomial matrix.
	Hence, $\bbG(\x,\y)[\z,\y] = (\x,\y) J_F^{-1}(\z,\y)$, the auxiliary inverse of $F$, must be a polynomial
	and by Lemma~\ref{lem:F poly inv}, $\bbG$ is a hyporealization.
	Thus, $G$ is the solution of the hyporealization $\bbG(\x,\y)[\z,\y]$ and $G$ is hyporational.
	
	Since $F$ is bijective and $G$ is its inverse, Lemma~\ref{lem:q agrees with poly} says $G\lvert_{M_n(\C)^g}$ agrees with a 
	free polynomial, for each $n\in\Z^+$.
	In particular, $\dom_n(G) = M_n(\C)^g$ for each $n\in\Z^+$.
	Thus, Theorem~\ref{thm:hyporat domain poly} implies $G$ is a free polynomial.
\end{proof}

Theorem~\ref{thm:Free Jacobian Conjecture} - the Free Jacobian conjecture - tells us that a free polynomial $p$ is injective if and only
if $Dp(Y)[X]$ is nonsingular for all $X\in M(\C)^g$.
Corollary~\ref{cor:F poly inv} strengthens this condition.

%
%    ######  ##     ## ########   ######  ########  ######  ######## ####  #######  ##    ## 
%   ##    ## ##     ## ##     ## ##    ## ##       ##    ##    ##     ##  ##     ## ###   ## 
%   ##       ##     ## ##     ## ##       ##       ##          ##     ##  ##     ## ####  ## 
%    ######  ##     ## ########   ######  ######   ##          ##     ##  ##     ## ## ## ## 
%         ## ##     ## ##     ##       ## ##       ##          ##     ##  ##     ## ##  #### 
%   ##    ## ##     ## ##     ## ##    ## ##       ##    ##    ##     ##  ##     ## ##   ### 
%    ######   #######  ########   ######  ########  ######     ##    ####  #######  ##    ##
%
\subsection{Bijectivity criteria}
	\label{ssec:bij crit}
Proposition~\ref{prop:p bij F bij} tells us that a free polynomial is injective if and only if its scion is injective.
Thus, when testing the bijectivity of a free polynomial, it suffices to only test for the bijectivity of its scion.
The main result of this subsection, Theorem~\ref{thm:matrix free inverse function theorem},
combines Corollary~\ref{cor:F poly inv} with Pascoe's Free Jacobian conjecture to
get a more direct analog to the classical Jacobian conjecture.

Let $f\in (\C\fralg{\bby\cup\bbx})^g$.
% with $f = \sum_w f_w w$ and set
%\[
%	W = \set{w\in\fralg{\bby\cup\bbx}: f_w\neq 0}.
%\]
We say $f$ is $\x$-linear if $\abs{w}_x = 1$ for all monomials $w$ appearing in $f$.
In other words, $f$ is a sum of monomials that contain exactly one $\x$-term.

\begin{lemma}
	\label{lem:hyp jac matrix}
	Suppose $f\in (\C\fralg{\bby\cup\bbx})^g$.
	If $f$ is $\x$-linear, then there exists a matrix of bipartite polynomials, $\mathcal{J}\in M_g(\C\fralg{\bbz\lra \bby})$,
	such that
	\[
		\vecc[n]\left(f(Y)[X] \right) = \vecc[n](X) \cdot \mathcal{J}(Y^T\otimes I_n, I_n\otimes Y),
	\]
	for all $X,Y\in M_n(\C)^g$ and $n\in \Z^+$.
\end{lemma}

\begin{proof}
	We omit the details of the construction of $\mathcal{J}$ since it is almost exactly the same as 
	the construction of the hypomatrix representation found in Proposition~\ref{prop:hyporat has hypomat}.
%	Since $f$ is $\x$-linear there exist $U^\ell_{i,j},V^\ell_{i,j}\in \C\fralg{y}$ such that
%	\[
%		f^i(\y)[\x] = \sum_\ell \sum_{j=1}^g U^\ell_{i,j}(\y) x_j V^\ell_{i,j}(\y),
%	\]
%	for $1\leq i\leq g$.
%	Recall that $(U^\ell_{i,j})^T(X^T) = (U^\ell_{i,j}(X))^T$ and set 
%	\[
%		\mathcal{J}(\x,\y) = \left( \sum_{\ell} (U^\ell_{i,j})^T(\x)V^\ell_{i,j}(\y) \right)_{j,i=1}^g.
%	\]
%	Let $X,Y\in M_n(\C)^g$ and observe
%	\begin{align*}
%		\vecc[n]\left(f(Y)[X]\right)
%			&=\vecc[n]\left(\sum_{\ell} \sum_{j=1}^g U^\ell_{i,j}(Y)X_j V^\ell_{i,j}(Y)\right)\\
%			&=\sum_{\ell} \sum_{j=1}^g \vecc[n]\left(U^\ell_{i,j}(Y)X_j V^\ell_{i,j}(Y)\right)\\
%			&=\sum_{j=1}^g \sum_{\ell} \vecc[n](X)\cdot 
%				\left(U^\ell_{i,j}(Y)^T\otimes V^\ell_{i,j}(Y)\right) \\
%			&=\sum_{j=1}^g \vecc[n](X)
%				\sum_{\ell}(U^\ell_{i,j})^T(Y^T\otimes I)V^\ell_{i,j}(I\otimes Y) \\
%			&=\sum_{j=1}^g \vecc[n](X)\cdot \mathcal{J}\left(Y^T\otimes I_n,I_n\otimes Y\right)_{j,i}.
%	\end{align*}
%	Thus,
%	\[
%		\vecc[n]\left(f(Y)[X])\right) 
%			= \vecc[n]\left(X \right)\cdot \mathcal{J}\left(Y^T\otimes I_n, I_n\otimes Y\right),
%	\]
%	for all $X,Y\in M_n(\C)^g$ and $n\in \Z^+$.
\end{proof}

\begin{definition}
	\label{def:hypo-Jacobian}
	Suppose $p:M(\C)^g\to M(\C)^g$ is a free polynomial with derivative, $Dp(\y)[\x]\in (\C\fralg{\bby\cup\bbx})^g$.
	Since $Dp(\y)[\x]$ is $\x$-linear, Lemma~\ref{lem:hyp jac matrix} implies there exists a matrix, 
	$\hypj{p}\in M_g(\C\fralg{\bbz\lra \bby})$,
	such that
	\[
		\vecc[n](Dp(Y)[X]) = \vecc[n](X)\cdot \hypj{p}\left(Y^T\otimes I_n, I_n\otimes Y\right),
	\]
	for all $X,Y\in M_n(\C)^g$ and $n\in \Z^+$.
	We define $\hypj{p}$ to be the \dfn{hypo-Jacobian matrix} of $p$.
	The hypo-Jacobian matrix is unique.
\end{definition}

\begin{remark}
	Hypo-Jacobian matrices satisfy the chain rule.
	Namely,
	\[
		\hypj{\alpha\circ \beta}(\z,\y) = \hypj{\beta}(\z,\y)\hypj{\alpha}(\beta^T(\z),\beta(\y)),
	\]
	for all $\alpha,\beta\in (\C\fax_+)^g$.
\end{remark}

Any endomorphism of the free associative algebra $\C\fax$ has a Jacobian matrix (see \cite{DL82} and \cite{Sch85}) that
exactly corresponds with the hypo-Jacobian matrix found in this section.
The Jacobian matrix of an endomorphism is a matrix over $\C\fralg{\bbz}^{opp}\otimes \C\fax$, where
$\C\fralg{\bbz}^{opp}$ is the opposite ring of $\C\fralg{\bbz}$ (the order of multiplication is reversed).
The construction of the hypo-Jacobian matrix sends terms of the form $U(\y)x_i V(\y)$ to $U^T(\z)\otimes V(\y)$.
Since we can view the map $U\mapsto U^T$ as the canonical anti-isomorphism from $\C\fralg{\bbz}\to \C\fralg{\bbz}^{opp}$, we
see that the hypo-Jacobian matrix of a polynomial mapping and the Jacobian matrix of an endomorphism of $\C\fax$ are indeed the same.

%\begin{proof}
%	
%	Take $\alpha,\beta\in (\C\fax_+)^g$ with hypo-Jacobian matrices, $\hypj{\alpha}$ and $\hypj{\beta}$, respectively.
%	The chain rule for free derivatives tells us that $D(\alpha\circ \beta)(\y)[\x] = D\alpha(\beta(\y))[D\beta(\y)[\x]]$.
%	Hence,
%	\begin{align*}
%		&\vecc[n](D\alpha\circ \beta(Y)[X])\\
%			&=\vecc[n]\left(D\alpha(\beta(Y))[D\beta(Y)[X]]\right) \\
%			&= \vecc[n]\left(D\beta(Y)[X]\right)\cdot \hypj{\alpha}\left(\beta(Y)^T\otimes I_n, I_n\otimes \beta(Y)\right)\\
%			&= \vecc[n](X)\cdot \hypj{\beta}\left(Y^T\otimes I_n,I_n\otimes Y\right)
%				\hypj{\alpha}\left(\beta(Y)^T\otimes I_n, I_n\otimes \beta(Y)\right)\\
%			&= \vecc[n](X)\cdot \hypj{\beta}\left(Y^T\otimes I_n,I_n\otimes Y\right)
%				\hypj{\alpha}\left(\beta^T(Y^T)\otimes I_n, I_n\otimes \beta(Y)\right),
%	\end{align*}
%	for all $X,Y\in M_n(\C)^g$.
%	Thus,
%	\[
%		\hypj{\alpha\circ \beta}(\z,\y) = \hypj{\beta}(\z,\y)\hypj{\alpha}(\beta^T(\z),\beta(\y)),
%	\]
%	as desired.
%\end{proof}

\THMmatfreeinverse

\begin{proof}
	Suppose $p$ is injective, $q$ is its inverse and $F = (Dp(\y)[\x],\y)$ is the scion of $p$.
	Letting $G$ be the inverse of $F$, Lemma~\ref{lem:F poly inv} shows $G(\x,\y) = (Dq(p(\y))[\x],\y)$.
	Corollary~\ref{cor:F poly inv} implies $G$ is a free polynomial, hence $Dq(p(\y))[\x]$ is a free polynomial.
	Since $Dq(p(\y))[\x]$ is $\x$-linear, Lemma~\ref{lem:hyp jac matrix} implies there exists a matrix 
	$\mathcal{J}\in M_g(\C\fralg{\bbz\lra\bby})$, such that
	\[
		\vecc[n](Dq(p(Y))[X]) = \vecc[n](X)\cdot \mathcal{J}(Y^T\otimes I_n,I_n\otimes Y).
	\]
	The chain rule tells us $Dq\circ p(\y)[\x] = Dq(p(\y))[Dp(\y)[\x]] = \x$, thus
	\begin{align*}
		\vecc[n](X) 
			&= \vecc[n](Dp(Y)[X])\cdot \mathcal{J}(Y^T\otimes I_n,I_n\otimes Y)\\
			&= \vecc[n](X)\cdot (\hypj{p}\cdot \mathcal{J})(Y^T\otimes I_n,I_n\otimes Y).
	\end{align*}
	Next, $Dp\circ q(p(\y))[\x] = Dp(\y)[Dq(p(\y))[\x]] = \x$, hence
	\begin{align*}
		\vecc[n](X) 
			&= \vecc[n](Dq(p(Y))[X])\cdot \hypj{p}(Y^T\otimes I_n,I_n\otimes Y)\\
			&= \vecc[n](X)\cdot (\mathcal{J}\cdot \hypj{p})(Y^T\otimes I_n,I_n\otimes Y).
	\end{align*}
	Thus,
	\begin{align*}
		\vecc[n](X)
			&= \vecc[n](X)\cdot (\hypj{p}\cdot \mathcal{J})(Y^T\otimes I_n,I_n\otimes Y)\\
			&= \vecc[n](X)\cdot (\mathcal{J}\cdot \hypj{p})(Y^T\otimes I_n,I_n\otimes Y),
	\end{align*}
	for all $X,Y\in M_n(\C)^g$, and $n\in\Z^+$.
	In other words, $(\hypj{p})^{-1} = \mathcal{J}\in M_g(\C\fralg{\bbz\lra \bby})$.

	Conversely suppose $(\hypj{p})^{-1}\in M_g(\C\fralg{\bbz\lra\bby})$.
	Let $\hat{G}$ be the free polynomial defined by
	\[
		\hat{G}(Y)[X] = \vecc[n]^{-1}\left(\vecc[n](X)\cdot(\hypj{p})^{-1}(Y^T\otimes I_n,I_n\otimes Y)\right),
	\]
	for all $X,Y\in M_n(\C)^g$ and $n\in\Z^+$.
	Observe
	\begin{align*}
		\vecc[n](X) 
			&= \vecc[n](X)\cdot \left(\hypj{p}\cdot(\hypj{p})^{-1}\right)(Y^T\otimes I_n,I_n\otimes Y)\\
			&= \vecc[n](Dp(Y)[X])\cdot(\hypj{p})^{-1}(Y^T\otimes I_n,I_n\otimes Y)\\
			&= \vecc[n](\hat{G}(Y)[Dp(Y)[X]]).
	\end{align*}
	Hence, $X = \hat{G}(Y)[Dp(Y)[X]]$ and $\x = \hat{G}(\y)[Dp(\y)[\x]]$.
	On the other hand,
	\begin{align*}
		\vecc[n](X) 
			&= \vecc[n](X)\cdot \left((\hypj{p})^{-1}\cdot \hypj{p}\right)(Y^T\otimes I_n,I_n\otimes Y)\\
			&= \vecc[n](\hat{G}(Y)[X])\cdot \hypj{p}(Y^T\otimes I_n,I_n\otimes Y)\\
			&= \vecc[n](Dp(Y)[\hat{G}(Y)[X]]).
	\end{align*}
	Hence, $X = Dp(Y)[\hat{G}(Y)[X]]$ and consequently, $\x = Dp(\y)[\hat{G}(\y)[\x]]$.
	By setting $G(\x,\y) = (\hat{G}(\y)[\x],\y)$, we get 
	\[
		G(F(\x,\y)) = G(Dp(\y)[\x],\y) = (\hat{G}(\y)[Dp(\y)[\x]],\y) = (\x,\y),
	\]
	and 
	\[
		F(G(\x,\y)) = F(\hat{G}(\y)[\x],\y) = (Dp(\y)[\hat{G}(\y)[\x]],\y) = (\x,\y).
	\]
	Thus, $G$ is the inverse of $F$.
	Therefore, by Corollary~\ref{cor:F poly inv}, $p$ is an injective free polynomial.
\end{proof}

Before we finally move on to the proof of Theorem~\ref{thm:Free Groth Thm} we connect the composition of polynomial mappings
to the composition of endomorphisms of $\C\fax$.

\begin{definition}
	Suppose $p\in (\C\fax)^g$ is a free polynomial mapping and let $\phi:\C\fax\to\C\fax$ be an algebra homomorphism.
	We say $\phi$ is \dfn{induced by} $p$ if $\phi(x_i) = p_i(\x)$, for $1\leq i\leq g$.
	Similarly, we say $p$ is \dfn{induced by} $\phi$ if $p(\x) = (\phi(x_1),\dots, \phi(x_g))$.
\end{definition}

\begin{lemma}
	\label{lem:induced comp}
	Suppose $\phi,\psi:\C\fax\to\C\fax$ are algebra homomorphisms.
	If $p,q$ are the induced polynomial mappings of $\phi$ and $\psi$, respectively, then
	\[
		(q\circ p)(\x) = \left((\phi\circ\psi)(x_1),  \dots,  (\phi\circ\psi)(x_g)\right).
	\]
\end{lemma}

\begin{proof}
	This is verified rather easily from definitions. The details are left to the reader.
%	Let $q_i = \sum_w q_w^i w$ and observe
%	\[
%		(\phi\circ \psi)(x_i) = \phi(q_i(\x)) = \phi\Big(\sum_{w\in\fax} q_w^i w \Big) = \sum_{w\in\fax} q^i_w \phi(w).
%	\]
%	However, if $w = x_{i_1}\dots x_{i_n}$ then $\phi(w) = p_{i_1}(\x)\dots p_{i_n}(\x) = w(p(\x))$.
%	Hence,
%	\[
%		(\phi\circ \psi)(x_i) = \sum_{w\in\fax} q^i_w \phi(w) = \sum_{w\in\fax} q^i_w w(p(\x)) = (q_i\circ p)(\x).
%	\]
%	Therefore,
%	\[
%		\left((\phi\circ\psi)(x_1), \, \dots, \, (\phi\circ\psi)(x_g) \right) 
%			= \left((q_1\circ p)(\x), \, \dots, \, (q_g\circ p)(\x)\right) = (q\circ p)(\x).
%	\]
\end{proof}

%\begin{proposition}
%	A free polynomial mapping $p$ is injective if and only if its induced endomorphism $\phi:\C\fax\to\C\fax$ is epic.
%\end{proposition}
%
%\begin{proof}
%	Suppose $p$ is injective.
%	Let $\alpha,\beta$ be endomorphisms such that $\alpha\circ \phi = \beta\circ \phi$ and let $a,b$ be their respective induced mappings.
%	By Lemma~\ref{lem:induced comp},
%	\[
%		(p\circ a)(\x) = \left((\alpha\circ\phi)(x_1), \, \dots, \, (\alpha\circ\phi)(x_g)\right)
%		 = \left((\beta\circ\phi)(x_1), \, \dots, \, (\beta\circ\phi)(x_g)\right) = (p \circ b)(\x).
%	\]
%	Since $p$ is injective, $(p\circ a)(\x) = (p\circ b)(\x)$ implies $a = b$.
%	Thus, $\alpha(x_i) = a_i(\x) = b_i(\x) = \beta(x_i)$ and $\alpha = \beta$.
%	Therefore, $\phi$ is epic.
%\end{proof}

If $\phi$ is an endomorphism of $\C\fax$ then the Jacobian matrix of $\phi$ is a $g\times g$ matrix over $\C\fralg{\bbz}^{opp}\otimes \C\fax$.
More specifically, if $p$ is the polynomial mapping induced by $\phi$, then the Jacobian matrix of $\phi$ is found by applying the
natural anti-isomorphism $M_g(\C\fralg{\bbz}\otimes \C\fax)\to M_g(\C\fralg{\bbz}^{opp}\otimes \C\fax)$ to $\hypj{p}$.

\begin{theorem}
	\label{thm:Pre-Free Groth Thm}
	Suppose $p:M(\C)^g\to M(\C)^g$ is a free polynomial mapping.
	The following are equivalent;
	\begin{enumerate}[label=(\roman*)]
		\item $p$ is injective;
		\item $p$ is bijective;
		\item $Dp(Y)$ is a nonsingular map for all $Y\in M_n(\C)^g$ and all $n\in\Z^+$;
		\item $\hypj{p}$ is invertible;
		\item $p^{-1}$ exists and is a free polynomial.
	\end{enumerate}
\end{theorem}

\begin{proof}
	(i)$\Leftrightarrow$(ii)$\Leftrightarrow$(iii) is Theorem~\ref{thm:Free Jacobian Conjecture}. 	
	(i)$\Leftrightarrow$(iv) is Theorem~\ref{thm:matrix free inverse function theorem}.	
	(v)$\Rightarrow$(i) is clear.
	
	To show (iv)$\Rightarrow$(v), we assume $\hypj{p}$ is invertible. 
	Let $\phi$ be the endomorphism of $\C\fax$ induced by $p$ (so $\phi(x_i) = p_i$ for $1\leq i\leq g$).
	Since $\hypj{p}$ is invertible, it follows that that Jacobian matrix of $\phi$ is invertible.
	Thus, Proposition~3.1 in \cite{DL82} implies $\phi$ is an epic endomorphism and Theorem~12.7 in \cite{Sch85} 
		implies $\phi$ is an automorphism of $\C\fax$.
	So $\phi^{-1}$ exists and is an automorphism itself.
	
	Let $q = (\phi^{-1}(x_1),\dots, \phi^{-1}(x_g))$ be polynomial mapping induced by $\phi^{-1}$.
	By Lemma~\ref{lem:induced comp}, 
	\[
		(x_1,\dots,x_g) = ((\phi\circ \phi^{-1})(x_1), \dots, (\phi\circ \phi^{-1})(x_g)) = (q\circ p)(\x)
	\]
	and 
	\[
		(x_1,\dots,x_g) = ((\phi^{-1}\circ \phi)(x_1), \dots, (\phi^{-1}\circ \phi)(x_g)) = (p\circ q)(\x).
	\]
	Thus, $p$ and $q$ are inverse mappings.
	Therefore, $p$ is injective.
\end{proof}

\THMfreeGroth

\begin{proof}
	This is exactly (i)$\Rightarrow$(v) in Theorem~\ref{thm:Pre-Free Groth Thm}.
\end{proof}

\section{Computing inverses}
	\label{sec:compute}
Suppose $p:M(\C)^g\to M(\C)^g$ is a free polynomial, $F$ is its scion and $q$ and $G$ are the inverses 
(when they exist) of $p$ and $F$, respectively.
By either Corollary~\ref{cor:F poly inv} or Theorem~\ref{thm:matrix free inverse function theorem} we know that
$p$ is injective if and only if $G$ is free polynomial.
Recall that if $p$ has a free polynomial inverse $q$, then Lemma~\ref{lem:F poly inv} tells us that
$\deg(q)\leq \deg(G)\leq \deg(p)\deg(q)$.
Thus, an upper bound on $\deg(G)$ gives us an upper bound on the possible degree of $q$.

\begin{definition}
	Let $V_t(0) = 0$ and for any nonzero rational function $r\in\C\skf{\bbx}$ define
	\[
		V_t(r(t\x)) = \min\set{\max\set{\deg_t(\alpha(\x)[t]),\deg_t(\beta(\x)[t])}:r(t\x) = \alpha(\x)[t]\beta(\x)[t]^{-1}}.
	\]
	If $M\in M_{m\times n}(\C\skf{\bbx})$ then define
	\[
		V_t(M(t\x)) = \max\set{V_t(M(t\x)_{i,j}):1\leq i\leq m, 1\leq j\leq n}.
	\]
	Note that if $r\in \C\fax$ then $V_t(r(t\x)) = \deg(r(\x))$.
\end{definition}

\begin{remark}
	It is straightforward to see that $\abs{\deg_t(r(t\x))}\leq V_t(r(t\x))$, for any nonzero rational function.
	By appealing to evaluations on generic matrices we get that
	\[
		V_t(r(t\x)s(t\x)) \leq V_t(r(t\x)) + V_t(s(t\x))
	\]
	and
	\[
		V_t(r(t\x)+s(t\x)) \leq V_t(r(t\x))+V_t(s(t\x)).
	\]
	Hence, if $M\in M_{\ell\times m}(\C\skf{\bbx})$ and $N\in M_{m\times n}(\C\skf{\bbx})$ then
	\begin{align*}
		V_t(M(t\x)N(t\x)) &\leq \max_{i,k}\Big\{\sum_{j=1}^m V_t\left(M(t\x)_{i,j}N(t\x)_{j,k}\right)\Big\}\\
			&\leq \max_{i,k}\Big\{\sum_{j=1}^m V_t\left(M(t\x)_{i,j}\right) + V_t\left(N(t\x)_{j,k}\right)\Big\}\\
			&\leq m\cdot V_t(M(t\x) + m\cdot V_t(N(t\x))).
	\end{align*}
\end{remark}

%Let $M$ be a block matrix and observe
%\begin{align*}
%	\abs{\nu_t(A-bd^{-1}c)}&\leq V_t(A-bd^{-1}c)\leq V_t(A)+V_t(bd^{-1}c)\\
%		&\leq V_t(M) + 3nV_t(M) = (3n+1)V_t(M)
%\end{align*}
%and
%\begin{align*}
%	\abs{\nu_t(1+cNbd^{-1})}&\leq V_t(1+cNbd^{-1}) = V_t(d^{-1}) + V_t(cNb)\\
%		&\leq V_t(d^{-1}) + \sum_{i=1}^n \left(V_t(c_i)\right) + \sum_{j,k=1}^n \left(V_t(N_{j,k}) + V_t(b_k)\right)\\
%		&\leq (2n+1)V_t(M) + n^2V_t(N)
%\end{align*}

%$V_t(N^{-1})\leq 3^n\mathfrak{f}(n)V_t(N)$ 
%
%then
%\begin{align*}
%	(2n+1)V_t(M) + n^2V_t(N)&\leq (2n+1)V_t(M) + (3^{n}\mathfrak{f}(n))(3n^3+n^2) V_t(M)\\
%		&= (3^{n+1}n^3+3^nn^2)\mathfrak{f}(n) + 2n+1)V_t(M)\\
%		&\leq 3^{n+1}(n^3+n^2+2n+1)\mathfrak{f}(n)V_t(M)\\
%		&\leq 3^{-1}(n+1)^3\mathfrak{f}(n)V_t(M) = 3^{n+1}\mathfrak{f}(n+1)V_t(M).
%\end{align*}

\begin{lemma}
	\label{lem:rat mat inv deg bound}
	Let $\mathfrak{f}(1) = 1$ and for any integer $n\geq 1$, define $\mathfrak{f}(n+1) = (n+1)^3\mathfrak{f}(n)$.
	If $M,M^{-1}\in M_n(\C\skf{\bbx})$ then $\abs{\deg_t(M^{-1}(t\x))}\leq 3^{n}\mathfrak{f}(n)V_t(M(t\x))$.
\end{lemma}

\begin{proof}
	We claim $V_t(M^{-1}(t\x))\leq 3^{n}\mathfrak{f}(n)V_t(M(t\x))$ and use induction to prove the claim.

	If $n=1$ then $M,M^{-1}\in\C\skf{\bbx}$ and $\deg_t(M(t\x)^{-1}) = -\deg_t(M(t\x))$,
	hence $V_t(M(t\x)^{-1}) = V_t(M(t\x))\leq 3 V_t(M(t\x))$.

	Now suppose the statement holds for $n$ and consider $M,M^{-1}\in M_{n+1}(\C\skf{\bbx})$.
	Write
	\[
		M = \bpm A & b \\ c & d \epm
	\]
	where $A\in M_n(\C\skf{\bbx})$, $b\in M_{n\times 1}(\C\skf{\bbx})$, $c\in M_{1\times n}(\C\skf{\bbx})$, and $d\in\C\skf{\bbx}$.
	
	If $S\in \textup{GL}_{n+1}(\C)$ then $MS$ is invertible.
	Hence, we may assume $d$ is nonzero and thus, $d^{-1}$ exists.
	Observe
	\[
		V_t(A(t\x)-b(t\x)d(t\x)^{-1}c(t\x))\leq V_t(A)+V_t(bd^{-1}c)\leq (3n+1)V_t(M(t\x)).
	\]
	Hence
	\[
		V_t((A-bd^{-1}c)^{-1})\leq 3^n(3n+1)\mathfrak{f}(n)V_t(M).
	\]
	Set $N(t\x) = (A(t\x)-b(t\x)d(t\x)^{-1}c(t\x))^{-1}$ and observe
	\begin{align*}
		V_t(d^{-1}(1+cNbd^{-1}))&\leq V_t(M)+V_t(1+cNbd^{-1}) = 2V_t(M) + V_t(cNb)\\
			&\leq 2V_t(M) + \sum_{i=1}^n \left(V_t(c_i)\right) + \sum_{j,k=1}^n \left(V_t(N_{j,k}) + V_t(b_k)\right)\\
			&\leq 2(n+1)V_t(M) + n^2V_t(N).
	\end{align*}
	Applying the induction hypothesis,
	\begin{align*}
		2(n+1)V_t(M) + n^2V_t(N)&\leq (2n+1)V_t(M) + n^2(3^{n}(3n+1)\mathfrak{f}(n) V_t(M))\\
			&= ((3^{n+1}n^3+3^nn^2)\mathfrak{f}(n) + 2n+2)V_t(M)\\
			&\leq 3^{n+1}(n^3+n^2+2n+2)\mathfrak{f}(n)V_t(M)\\
			&\leq 3^{n+1}(n+1)^3\mathfrak{f}(n)V_t(M) = 3^{n+1}\mathfrak{f}(n+1)V_t(M).
	\end{align*}
	Since the inverse of $M$ is determined from the Schur complement, we have proven that
	$V_t(M^{-1}(t\x))\leq 3^{n}\mathfrak{f}(n)V_t(M(t\x))$.
	
	Finally, since $\abs{\deg_t(r(t\x))}\leq V_t(r(t\x))$ for any nonzero rational $r$, we have
	\[
		\abs{\deg_t(M^{-1}(t\x))}\leq 3^{n}\mathfrak{f}(n)V_t(M(t\x)),
	\]
	as desired.
\end{proof}

%\begin{proof}
%	If $n=1$ then $M,M^{-1}\in\C\skf{\bbx}$ and $\nu_t(M(t\x)^{-1}) = -\nu_t(M(t\x))$,
%	hence $\abs{\nu_t(M(t\x)^{-1})} = (3-2)\abs{\nu_t(M(t\x))}$.
%	
%
%	Now suppose the statement holds for $n$ and consider $M,M^{-1}\in M_{n+1}(\C\skf{\bbx})$.
%	Write
%	\[
%		M = \bpm A & b \\ c^T & d \epm
%	\]
%	where $A\in M_n(\C\skf{\bbx})$, $b,c\in M_{n\times 1}(\C\skf{\bbx})$ and $d\in\C\skf{\bbx}$.
%	
%	If $S\in \textup{GL}_{n+1}(\C)$ then $MS$ is invertible.
%	Hence, we may assume $d$ is nonzero and thus, $d^{-1}$ exists.
%	Consider
%	\[
%		\abs{\nu_t(A(t\x)-b(t\x)d(t\x)^{-1}c(t\x)^T)}\leq 3\abs{\nu_t(M(t\x))}
%	\]
%	hence
%	\[
%		\abs{\nu_t((A(t\x)-b(t\x)d(t\x)^{-1}c(t\x)^T)^{-1})}\leq (3^n-2)3\abs{\nu_t(M(t\x))}.
%	\]
%	Set $N(t\x) = (A(t\x)-b(t\x)d(t\x)^{-1}c(t\x)^T)^{-1}$ and observe
%	\begin{align*}
%		\abs{\nu_t(d^{-1} + d^{-1}c^T N bd^{-1})} &\leq (3^{n+1}-6+4)\abs{\nu_t(M(t\x))}\\ 
%		&= (3^{n+1}-2)\abs{\nu_t(M(t\x))}.
%	\end{align*}
%	Since the inverse of $M$ is determined from the Schur complement, the proof follows by induction.
%\end{proof}

\begin{lemma}
	\label{lem:poly mat inv deg bound}
	If $M,M^{-1}\in M_n(\C\fax)$ then $\deg(M^{-1})\leq 3^n\mathfrak{f}(n)\deg(M)$.
\end{lemma}

\begin{proof}
	Since $M$ and $M^{-1}$ are matrices of polynomials, $V_t(M(t\x)) = \deg(M(t\x))$ and $V_t(M(t\x)^{-1}) =\deg(M(t\x)^{-1})$.
	Thus, by Lemma~\ref{lem:rat mat inv deg bound},
	\begin{align*}
		\deg(M(t\x)^{-1}) = \abs{\deg_t(M(t\x)^{-1})}\leq 3^n\mathfrak{f}(n)V_t(M) = 3^n\mathfrak{f}(n)\deg(M).
		\tag*{$\square$}
	\end{align*}
\end{proof}

The degree bound in Lemma~\ref{lem:poly mat inv deg bound} is far from optimal.
However, to improve the degree bound in a significant manner would require an altogether different proof; 
the induction hypothesis cannot be applied to $(A-bd^{-1}c)^{-1}$ since it is not necessarily the inverse of a polynomial matrix.

Suppose $\B:\N\times \N\to\N$ is a function such that whenever $M\in M_n(\C\fax)$ and $M^{-1}\in M_n(\C\fax)$,
we have $\deg(M^{-1})\leq \B(n,\deg(M))$.
We call such a function a \dfn{PMID bound} (for Polynomial Matrix Inverse Degree).

\begin{theorem}
	\label{thm:deg p bound}
	Suppose $\B$ is a PMID bound.
	Let $p$ be a bijective free polynomial and let $q$ be its inverse.
	If $q$ is a free polynomial then $\deg(q)\leq \B(g,\deg(p)-1)+1$.
\end{theorem}

\begin{proof}
	We begin by noting that $\deg(Dp(\y)[\x]) = \deg(p)$. 
	Since $\hypj{p}$, the hypo-Jacobian of $p$, is constructed from $Dp(\y)[\x]$, we get that $\deg(\hypj{p}) = \deg(p)-1$.
	By Theorem~\ref{thm:Pre-Free Groth Thm} we know $(\hypj{p})^{-1}$ is a polynomial matrix since $p$ is injective.
	Hence, $\deg((\hypj{p})^{-1})\leq \B(g,\deg(p)-1)$.
	In fact, for all $n\in\Z^+$ and $X,Y\in M_n(\C)^g$,
	\[
		\vecc[n]\left(Dq(p(Y))[X]\right) 
			= \vecc[n](X)\cdot (\hypj{p})^{-1}(Y^T\otimes I_n, I_n\otimes Y).
	\]
	Thus, $\deg(Dq(p(\y))[\x]) = \deg((\hypj{p})^{-1})+1 \leq \B(g,\deg(p)-1)+1$.
	Lemma~\ref{lem:F poly inv} says $G = (Dq(p(\y))[\x],\y)$, where $G$ is the inverse of $F$, the scion of $p$.
	In fact, $\deg(q)\leq \deg(G)$. 
	Therefore
	\[
		\deg(q)\leq \deg(G) = \deg(Dq(p(\y))[\x]) \leq \B(g,\deg(p)-1)+1,
	\]
	as desired.
\end{proof}

Recall from Definition \ref{def:aux inv}, $\bbq(\x)[\z] = \x J_p(\z)$ and $\bbq^{\circ k+1} (\x)[\z] = \bbq^{\circ k}(\x)[\bbq(\x)[\z]]$.
For each $k$, we write
\[
	\bbq^{\circ k}(\x)[\z] = \sum_{w\in \fralg{\bbx\cup\bbz}} \rho^k_w w,
\]
set $\bbd_q^k = \bbd_z(\bbq^{\circ k}) = \inf\set{\abs{w}:\abs{w}_z>0\text{ and } \rho^k_w\neq 0}$ and write
$\bbq^{\circ k}(\x)[\z] = \mfq^k(\x) + r^k(\x)[\z]$, where
\[
	\mfq^k(\x) = \sum_{\substack{w\in\fax \\ \abs{w}<\bbd^k_q}} \rho^k_w w
	\quad \text{ and } \quad
	r^k(\x)[\z] = \sum_{\substack{w\in\fralg{\bbx\cup\bbz} \\ \abs{w}\geq\bbd^k_q}} \rho^k_w w.
\]
From Lemma~\ref{lem:ak and bbd} we know $\bbd_z^k>k$ and $(\mfq^k)_{k=1}^\infty$ is a sequence of polynomials converging to $q$ such that
	$\deg(\mfq^k)\nearrow \deg(q)$.
Moreover, Lemma~\ref{lem:a gaps} shows that $q$ is a polynomial if and only if there exists an $N$ such that
$\deg(p)\deg(\mfq^N)<\bbd^N_z$.

\begin{lemma}
	\label{lem:mfq max bound}
	Suppose $\B$ is a PMID bound.
	Let $p\in (\C\fax_+)^g$ and let $q$ be the compositional inverse of $p$.
	Set $B = \B(g,\deg(p)-1)+1$.
	Then $q$ is a free polynomial if and only if $J_p^{-1}\in M_g(\C\fax)$ and $\deg(\mfq^k)\leq B$ for all $k$.
\end{lemma}

\begin{proof}
	Suppose $J_p^{-1}\in M_g(\C\fax)$ and $\deg(\mfq^k)\leq B$ for all $k$.
	In particular, $\deg(\mfq^{B\deg(p)})\leq B$, hence $\deg(\mfq^{B\deg(p)})\deg(p)\leq B\deg(p) < \bbd^{B\deg(p)}_q$ 
		and $q$ is a polynomial by Lemma~\ref{lem:a gaps}.
	
	Conversely suppose $q$ is a polynomial.
	It follows that $q(X)$ exists for all $X\in M(\C)^g$, and consequently $p(q(X)) = X = q(p(X))$ for all $X\in M(\C)^g$.
	Thus, $p$ is bijective and Theorem~\ref{thm:PMI is poly} implies $J_p^{-1}\in M_g(\C\fax)$.
	Next, Theorem~\ref{thm:deg p bound} implies $\deg(q)\leq \B(g,\deg(p)-1)+1 = B$ and since $\deg(\mfq^k)\leq \deg(q)$,
	we have exactly $\deg(\mfq^k)\leq B$ for all $k$.
\end{proof}

Lemma~\ref{lem:mfq max bound} hints at a simple algorithm for determining whether $q$, 
the inverse of a given polynomial $p\in (\C\fax_+)^g$, is a polynomial.
We still set $B = \B(g,\deg(p)-1)+1$.

\begin{corollary}
\label{cor:eitheror}
	Either $\mfq^{B} = q$ or $q$ is not a polynomial and $p$ is not injective.
\end{corollary}

\begin{proof}
	Let $q = \sum_{w} \rho_w w$.
	Recall $\deg(\mfq^k)$ is an increasing sequence and $q$ and $\mfq^k$ agree on monomials of length less than $\bbd_q^k$.
	If $\mfq^B\neq q$ then there exists a $k>B$ and $w\in \fax$ with $\abs{w}>\bbd_q^B>B$ such that $\rho^k_w\neq 0$.
	In particular $\deg(\mfq^k)>B$, thus by Lemma~\ref{lem:mfq max bound}, $q$ is not a polynomial and $p$ is not injective.
\end{proof}

For any $p$, the algorithmic approach to computing $q$ is as follows.
\begin{itemize}
	\item[$-$] If: $J_p\notin M_g(\C\fax)$ then $q$ is not a polynomial;
	\item[$-$] Else: compute $B = \B(g,\deg(p)-1)+1$;
	\item[$-$] Set $k = 1$ and $\bbq^{\circ 0}(\x)[\z] = \z$;
	\item[$-$] Loop:
	\begin{itemize}
		\item[$\sim$] Compute $\bbq^{\circ k}(\x)[\z] = \bbq(\x)[\bbq^{\circ k-1}(\x)[\z]]$;
		\item[$\sim$] If: $\deg(\mfq^k)>B$ or $k>B$ then $q$ is not a polynomial;
		\item[$\sim$] ElseIf: $\deg(\mfq^k)\deg(p)<\bbd_q^k$ then $\mfq = q$;
		\item[$\sim$] Else: Increase $k$ by one.
	\end{itemize}
\end{itemize}
In at most $B$ loops, the above algorithm would either tell us that $q$ is not a polynomial or would return a polynomial
inverse $q$.

\bibliographystyle{alpha}
\bibliography{Free}

\end{document}